\documentclass[11pt]{amsart}
\usepackage{amsthm}
\usepackage{amsmath}
\usepackage{amssymb}
\usepackage{verbatim}
\usepackage{mathrsfs}
\usepackage{eucal}
\let\mathcal=\CMcal
\usepackage[dvipdfm,bookmarks=true,bookmarksnumbered=true]{hyperref}
\usepackage[yyyymmdd,hhmmss]{datetime}
\usepackage{amscd}
\usepackage{color}
\usepackage{dsfont}
\usepackage{stmaryrd}
\usepackage{expdlist}
\usepackage{enumerate}

\allowdisplaybreaks[4]

\begin{document}
\def\sbt{\raisebox{1.2pt}{$\scriptscriptstyle\,\bullet\,$}}

\def\alp{\alpha}
\def\bet{\beta}
\def\gam{\gamma}
\def\del{\delta}
\def\eps{\epsilon}
\def\zet{\zeta}
\def\tht{\theta}
\def\iot{\iota}
\def\kap{\kappa}
\def\lam{\lambda}
\def\sig{\sigma}
\def\ome{\omega}
\def\vep{\varepsilon}
\def\vth{\vartheta}
\def\vpi{\varpi}
\def\vrh{\varrho}
\def\vsi{\varsigma}
\def\vph{\varphi}
\def\Gam{\Gamma}
\def\Del{\Delta}
\def\Lam{\Lambda}
\def\Tht{\Theta}
\def\Sig{\Sigma}
\def\Ups{\Upsilon}
\def\Ome{\Omega}
\def\vka{\varkappa}
\def\vDe{\varDelta}
\def\vSi{\varSigma}
\def\vTh{\varTheta}
\def\vGm{\varGamma}
\def\vOm{\varOmega}
\def\vPi{\varPi}
\def\vPh{\varPhi}
\def\vPs{\varPsi}
\def\vUp{\varUpsilon}
\def\vXi{\varXi}

\def\frka{{\mathfrak a}}    \def\frkA{{\mathfrak A}}
\def\frkb{{\mathfrak b}}    \def\frkB{{\mathfrak B}}
\def\frkc{{\mathfrak c}}    \def\frkC{{\mathfrak C}}
\def\frkd{{\mathfrak d}}    \def\frkD{{\mathfrak D}}
\def\frke{{\mathfrak e}}    \def\frkE{{\mathfrak E}}
\def\frkf{{\mathfrak f}}    \def\frkF{{\mathfrak F}}
\def\frkg{{\mathfrak g}}    \def\frkG{{\mathfrak G}}
\def\frkh{{\mathfrak h}}    \def\frkH{{\mathfrak H}}
\def\frki{{\mathfrak i}}    \def\frkI{{\mathfrak I}}
\def\frkj{{\mathfrak j}}    \def\frkJ{{\mathfrak J}}
\def\frkk{{\mathfrak k}}    \def\frkK{{\mathfrak K}}
\def\frkl{{\mathfrak l}}    \def\frkK{{\mathfrak K}}
\def\frakl{{\mathfrak l}}    \def\frkL{{\mathfrak L}}
\def\frkm{{\mathfrak m}}    \def\frkM{{\mathfrak M}}
\def\frkn{{\mathfrak n}}    \def\frkN{{\mathfrak N}}
\def\frko{{\mathfrak o}}    \def\frkO{{\mathfrak O}}
\def\frkp{{\mathfrak p}} 
\def\frakp{{\mathfrak p}}    \def\frkP{{\mathfrak P}}\def\frkq{{\mathfrak q}}    \def\frkQ{{\mathfrak Q}}
\def\frkr{{\mathfrak r}}    \def\frkR{{\mathfrak R}}
\def\frks{{\mathfrak s}}    \def\frkS{{\mathfrak S}}
\def\frkt{{\mathfrak t}}    \def\frkT{{\mathfrak T}}
\def\frku{{\mathfrak u}}    \def\frkU{{\mathfrak U}}
\def\frkv{{\mathfrak v}}    \def\frkV{{\mathfrak V}}
\def\frkw{{\mathfrak w}}    \def\frkW{{\mathfrak W}}
\def\frkx{{\mathfrak x}}    \def\frkX{{\mathfrak X}}
\def\frky{{\mathfrak y}}    \def\frkY{{\mathfrak Y}}
\def\frkz{{\mathfrak z}}    \def\frkZ{{\mathfrak Z}}

\def\cal{\fam2}
\def\cala{{\cal A}}
\def\calb{{\cal B}}
\def\calc{{\cal C}}
\def\cald{{\cal D}}
\def\cale{{\cal E}}
\def\calf{{\cal F}}
\def\calg{{\cal G}}
\def\calh{{\cal H}}
\def\cali{{\cal I}}
\def\calj{{\cal J}}
\def\calk{{\cal K}}
\def\call{{\cal L}}
\def\calm{{\cal M}}
\def\caln{{\cal N}}
\def\calo{{\cal O}}
\def\calp{{\cal P}}
\def\calq{{\cal Q}}
\def\calr{{\cal R}}
\def\cals{{\cal S}}
\def\calS{{\cal S}}
\def\calt{{\cal T}}
\def\calu{{\cal U}}
\def\calv{{\cal V}}
\def\calw{{\cal W}}
\def\calx{{\cal X}}
\def\caly{{\cal Y}}
\def\calz{{\cal Z}}

\def\AA{{\mathbb A}}
\def\BB{{\mathbb B}}
\def\CC{{\mathbb C}}
\def\DD{{\mathbb D}}
\def\EE{{\mathbb E}}
\def\FF{{\mathbb F}}
\def\GG{{\mathbb G}}
\def\HH{{\mathbb H}}
\def\II{{\mathbb I}}
\def\JJ{{\mathbb J}}
\def\KK{{\mathbb K}}
\def\LL{{\mathbb L}}
\def\MM{{\mathbb M}}
\def\NN{{\mathbb N}}
\def\OO{{\mathbb O}}
\def\PP{{\mathbb P}}
\def\QQ{{\mathbb Q}}
\def\RR{{\mathbb R}}
\def\SS{{\mathbb S}}
\def\TT{{\mathbb T}}
\def\UU{{\mathbb U}}
\def\VV{{\mathbb V}}
\def\WW{{\mathbb W}}
\def\XX{{\mathbb X}}
\def\YY{{\mathbb Y}}
\def\ZZ{{\mathbb Z}}

\def\bfa{{\mathbf a}}    \def\bfA{{\mathbf A}}
\def\bfb{{\mathbf b}}    \def\bfB{{\mathbf B}}
\def\bfc{{\mathbf c}}    \def\bfC{{\mathbf C}}
\def\bfd{{\mathbf d}}    \def\bfD{{\mathbf D}}
\def\bfe{{\mathbf e}}    \def\bfE{{\mathbf E}}
\def\bff{{\mathbf f}}    \def\bfF{{\mathbf F}}
\def\bfg{{\mathbf g}}    \def\bfG{{\mathbf G}}
\def\bfh{{\mathbf h}}    \def\bfH{{\mathbf H}}
\def\bfi{{\mathbf i}}    \def\bfI{{\mathbf I}}
\def\bfj{{\mathbf j}}    \def\bfJ{{\mathbf J}}
\def\bfk{{\mathbf k}}    \def\bfK{{\mathbf K}}
\def\bfl{{\mathbf l}}    \def\bfL{{\mathbf L}}
\def\bfm{{\mathbf m}}    \def\bfM{{\mathbf M}}
\def\bfn{{\mathbf n}}    \def\bfN{{\mathbf N}}
\def\bfo{{\mathbf o}}    \def\bfO{{\mathbf O}}
\def\bfp{{\mathbf p}}    \def\bfP{{\mathbf P}}
\def\bfq{{\mathbf q}}    \def\bfQ{{\mathbf Q}}
\def\bfr{{\mathbf r}}    \def\bfR{{\mathbf R}}
\def\bfs{{\mathbf s}}    \def\bfS{{\mathbf S}}
\def\bft{{\mathbf t}}    \def\bfT{{\mathbf T}}
\def\bfu{{\mathbf u}}    \def\bfU{{\mathbf U}}
\def\bfv{{\mathbf v}}    \def\bfV{{\mathbf V}}
\def\bfw{{\mathbf w}}    \def\bfW{{\mathbf W}}
\def\bfx{{\mathbf x}}    \def\bfX{{\mathbf X}}
\def\bfy{{\mathbf y}}    \def\bfY{{\mathbf Y}}
\def\bfz{{\mathbf z}}    \def\bfZ{{\mathbf Z}}

\def\scra{{\mathscr A}}
\def\scrb{{\mathscr B}}
\def\scrc{{\mathscr C}}
\def\scrd{{\mathscr D}}
\def\scre{{\mathscr E}}
\def\scrf{{\mathscr F}}
\def\scrg{{\mathscr G}}
\def\scrh{{\mathscr H}}
\def\scri{{\mathscr I}}
\def\scrj{{\mathscr J}}
\def\scrk{{\mathscr K}}
\def\scrl{{\mathscr L}}
\def\scrm{{\mathscr M}}
\def\scrn{{\mathscr N}}
\def\scro{{\mathscr O}}
\def\scrp{{\mathscr P}}
\def\scrq{{\mathscr Q}}
\def\scrr{{\mathscr R}}
\def\scrs{{\mathscr S}}
\def\scrt{{\mathscr T}}
\def\scru{{\mathscr U}}
\def\scrv{{\mathscr V}}
\def\scrw{{\mathscr W}}
\def\scrx{{\mathscr X}}
\def\scry{{\mathscr Y}}
\def\scrz{{\mathscr Z}}

\def\phm{\phantom}
\def\smallstrut{\vphantom{\vrule height 3pt }}
\def\bdm #1#2#3#4{\left(
\begin{array} {c|c}{\ds{#1}}
 & {\ds{#2}} \\ \hline
{\ds{#3}\vphantom{\ds{#3}^1}} &  {\ds{#4}}
\end{array}
\right)}

\def\GL{\mathrm{GL}}
\def\PGL{\mathrm{PGL}}
\def\GSp{\mathrm{GSp}}
\def\Sp{\mathrm{Sp}}
\def\GU{\mathrm{GU}}
\def\PGU{\mathrm{PGU}}
\def\PGSp{\mathrm{PGSp}}
\def\SL{\mathrm{SL}}
\def\Mp{\mathrm{Mp}}
\def\SU{\mathrm{SU}}
\def\SO{\mathrm{SO}}
\def\O{\mathrm{O}}
\def\U{\mathrm{U}}
\def\Mat{\mathrm{M}}
\def\Tr{\mathrm{Tr}}
\def\tr{\mathrm{tr}}
\def\ch{\mathrm{ch}}
\def\Nr{\mathrm{N}}
\def\Ad{\mathrm{Ad}}
\def\ch{\mathrm{ch}}
\def\Her{\mathrm{Her}}
\def\Paf{\mathrm{Paf}}
\def\Pf{\mathrm{Pf}}
\def\Gal{\mathrm{Gal}}
\def\Sym{\mathrm{Sym}}
\def\St{\mathrm{St}}
\def\st{\mathrm{st}}
\def\ad{\mathrm{ad}}
\def\Spec{\mathrm{Spec}}
\def\new{\mathrm{new}}
\def\Wh{\mathrm{Wh}}
\def\FJ{\mathrm{FJ}}
\def\Fj{\mathrm{Fj}}
\def\Sym{\mathrm{Sym}}
\def\Spn{\mathrm{Spn}}
\def\Aut{\mathrm{Aut}}
\def\Dif{\mathrm{Diff}}
\def\GK{\mathrm{GK}}
\def\supp{\mathrm{supp}}
\def\proj{\mathrm{proj}}
\def\vol{\mathrm{vol}}
\def\Hom{\mathrm{Hom}}
\def\frakN{\mathfrak{N}}
\def\End{\mathrm{End}}
\def\Ker{\mathrm{Ker}}
\def\Res{\mathrm{Res}}
\def\res{\mathrm{res}}
\def\cusp{\mathrm{cusp}}
\def\Irr{\mathrm{Irr}}
\def\rank{\mathrm{rank}}
\def\sgn{\mathrm{sgn}}
\def\diag{\mathrm{diag}}
\def\nd{\mathrm{nd}}
\def\lw{\mathrm{lw}}
\def\d{\mathrm{d}}
\def\rmK{\mathrm{K}}
\def\rmd{\mathrm{d}}
\def\x{\times}
\def\ol{\overline}
\def\bksl{\backslash}
\def\cW{\mathcal{W}}
\def\cZ{\mathcal{Z}}
\def\cB{\mathcal{B}}
\def\ur{\mathrm{ur}}
\def\La{\langle}
\def\Ra{\rangle}
\def\uf{\varpi}
\def\SymD{D_-}

\def\rmH{\mathrm{H}}
\def\Sel{\mathrm{Sel}}
\def\rmK{\mathrm{K}}
\def\Cl{\mathrm{Cl}}
\def\CM{\mathrm{CM}}
\def\Heeg{\mathrm{Heeg}}
\def\wh{\widehat}

\def\trs{\,^t\!}
\def\tri{\,^\iot\!}
\def\iu{\sqrt{-1}}
\def\oo{\hbox{\bf 0}}
\def\ono{\hbox{\bf 1}}
\def\smallcirc{\lower .3em \hbox{\rm\char'27}\!}
\def\thalf{\tfrac{1}{2}}
\def\bsl{\backslash}
\def\wtl{\widetilde}
\def\til{\tilde}
\def\Ind{\operatorname{Ind}}
\def\ind{\operatorname{ind}}
\def\cind{\operatorname{c-ind}}
\def\ord{\operatorname{ord}}
\def\beq{\begin{equation}}
\def\eeq{\end{equation}}
\def\d{\mathrm{d}}
\def\ot{\otimes}
\def\SymR{\breve R_-}

\newcounter{one}
\setcounter{one}{1}
\newcounter{two}
\setcounter{two}{2}
\newcounter{thr}
\setcounter{thr}{3}
\newcounter{fou}
\setcounter{fou}{4}
\newcounter{fiv}
\setcounter{fiv}{5}
\newcounter{six}
\setcounter{six}{6}
\newcounter{sev}
\setcounter{sev}{7}
\newcounter{eig}
\setcounter{eig}{8}
\newcounter{nine}
\setcounter{nine}{9}
\newcounter{ten}
\setcounter{ten}{10}
\newcounter{eleven}
\setcounter{eleven}{11}

\newcommand{\shp}{\rm\char'43}

\newcommand{\powerseries}[1]{\llbracket{#1}\rrbracket}
\def\lddots{\mathinner{\mskip1mu\raise1pt\vbox{\kern7pt\hbox{.}}\mskip2mu\raise4pt\hbox{.}\mskip2mu\raise7pt\hbox{.}\mskip1mu}}
\newcommand{\1}{1\hspace{-0.25em}{\rm{l}}}
\newcommand{\pMX}[4]{\begin{pmatrix}
{#1}& {#2}\\
{#3}&{#4}\end{pmatrix} }
 \newcommand{\pDII}[2]{\begin{pmatrix}{#1}&0
 \\0&{#2}\end{pmatrix}}

\makeatletter
\def\varddots{\mathinner{\mkern1mu
    \raise\p@\hbox{.}\mkern2mu\raise4\p@\hbox{.}\mkern2mu
    \raise7\p@\vbox{\kern7\p@\hbox{.}}\mkern1mu}}
\makeatother

\def\today{\ifcase\month\or
 January\or February\or March\or April\or May\or June\or
 July\or August\or September\or October\or November\or December\fi
 \space\number\day, \number\year}

\makeatletter
\def\varddots{\mathinner{\mkern1mu
    \raise\p@\hbox{.}\mkern2mu\raise4\p@\hbox{.}\mkern2mu
    \raise7\p@\vbox{\kern7\p@\hbox{.}}\mkern1mu}}
\makeatother

\def\today{\ifcase\month\or
 January\or February\or March\or April\or May\or June\or
 July\or August\or September\or October\or November\or December\fi
 \space\number\day, \number\year}

\makeatletter
\@addtoreset{equation}{section}
\def\theequation{\thesection.\arabic{equation}}

\theoremstyle{plain}
\newtheorem{theorem}{Theorem}[section]
\newtheorem*{main_proposition}{Proposition \ref{prop:51}}
\newtheorem*{main_theorem}{Theorem \ref{thm:51}}
\newtheorem{lemma}[theorem]{Lemma}
\newtheorem{proposition}[theorem]{Proposition}
\theoremstyle{definition}
\newtheorem{definition}[theorem]{Definition}
\newtheorem{conjecture}[theorem]{Conjecture}
\theoremstyle{remark}
\newtheorem{remark}[theorem]{Remark}
\newtheorem*{main_remark}{Remark}
\newtheorem{corollary}[theorem]{Corollary}

\renewcommand{\thepart}{\Roman{part}}
\setcounter{tocdepth}{1}
\setcounter{section}{0} 




\title[]{\bf Bessel periods and anticyclotomic $p$-adic spinor $L$-functions}  
\author{Ming-Lun Hsieh}
\address{ 
Institute of Mathematics, Academia Sinica~\\ Taipei 10617, Taiwan\and National Center for Theoretic Sciences~
}
\email{mlhsieh@math.sinica.edu.tw}
\author{Shunsuke Yamana}
\address{Osaka Metropolitan University, Department of Mathematics, Graduate School of Science, 
3-3-138 Sugimoto, Sumiyoshi-ku, Osaka 558-8585 JAPAN}
\email{yamana@omu.ac.jp}
\date{\today, \currenttime}
\begin{abstract}
We construct the anticyclotomic $p$-adic $L$-function that interpolates a square root of central values of twisted spinor $L$-functions of a quadratic base change of a Siegel cusp form of genus $2$ with respect to a paramodular group of square-free level. 
\end{abstract}
\maketitle
\tableofcontents


\section{Introduction}\label{sec:1}

The purpose of this article is to carry out the first step towards the analytic side of anticyclotomic Iwasawa theory for Siegel cusp forms by generalizing the works \cite{BD,CM,H1} for  elliptic cusp forms. Namely,
we construct anticyclotomic $p$-adic $L$-functions for scalar valued Siegel cusp forms of genus two and weight greater than one with respect to paramodular groups of square-free level and establish the explicit interpolation formulae. 

\subsection{Anticyclotomic Iwasawa main conjecture}

Let $\pi\simeq\otimes_v'\pi_v$ be a unitary irreducible cuspidal automorphic representation of $\PGSp_4(\AA)$ generated by a scalar valued degree two Siegel cuspidal Hecke eigenform $f$ of weight $\kap\geq 2$ and paramodular level $N$, where $\AA$ denotes the rational ad\`{e}le ring. 
Let $\frkS$ be the set of ramified places of the representation $\pi$. 
Fix a prime number $p\notin\frkS$ and embeddings $\iot_\infty:\bar\QQ\hookrightarrow\CC$ and $\iot_p:\bar\QQ\hookrightarrow\CC_p$.  Let $E/\QQ_p$ be a finite extension containing Hecke eigenvalues of $f$. Thanks to the work of many people (Chai-Faltings, Laumon, Shimura, Taylor and Weissauer), there exists a geometric $p$-adic Galois representation $\rho_{f,p}:\Gal(\bar\QQ/\QQ)\to\GSp_4(E)$ such that $\rho_{f,p}$ is unramified outside $\frkS\cup\{p\}$ and 
\[\det(\ono_4-\iot_\infty\iota_p^{-1}\rho_{f,p}(\mathrm{Frob}_\ell)\ell^{-s})^{-1}=L\left(s-\kap+\frac{3}{2},{\rm Spn}(\pi_\ell)\right) \]
for $\ell\notin\frkS\cup\{p\}$ at least if $\kap>2$, where $\mathrm{Frob}_\ell$ is the geometric Frobeninus and the right-hand side is the spinor $L$-factor of $\pi_\ell$. 
See \cite{Laumon} and \cite{W2} for the complete result. 
Denote by $\vep_{\rm cyc}$ the $p$-adic cyclotomic character. 
We are interested in the central critical twist
\[\rho_{f,p}^*:=\rho_{f,p}^{}\ot\vep_{\rm cyc}^{\kap-1}:\Gal(\bar\QQ/\QQ)\to\GSp_4(E). \]
Let $V_f=E^4$ be the representation space of $\rho_{f,p}^*$.  Then $V_f$ is self-dual in the sense that $V_f^\vee(1)\simeq V_f^{}$. We further assume the $V_f$ satisfies the following \emph{Panchishkin condition}: there exists a rank two $\Gal(\bar\QQ_p/\QQ_p)$-invariant subspace ${\rm Fil}_p^+V_f$ of $V_f$ such that ${\rm Fil}_p^+V_f$ has positive Hodge-Tate weights $(\kap-1,1)$ while the quotient $V_f/{\rm Fil}_p^+V_f$ has non-positive Hodge-Tate weights $(0,2-\kap)$\footnote{Here $\QQ_p(1)$ has Hodge-Tate weight $1$ in our convention.}. 
Let $\frko_E$ be the ring of integers of $E$. We shall fix a $\Gal(\bar\QQ/\QQ)$-stable $\frko_E$-lattice $T_f\subset V_f$ once and for all. Let $A_f=V_f/T_f$ and let ${\rm Fil}^+_pA_f$ be the image of ${\rm Fil}^+_pV_f$ in $A_f$. 
For any algebraic extension $L$ over $\QQ$, we consider the (minimal) Selmer group defined by 
\[\Sel(A_f/L):=\ker\biggl\{\rmH^1(L,A_f)\to \prod_{v\nmid p}\rmH^1(L_v,A_f)\times \prod_{\frkp\mid p}\rmH^1(L_\frkp,A_f/{\rm Fil}_p^+A_f)\biggl\}. \]

Let $K$ be an imaginary quadratic field of discriminant $-\Delta_K<0$ with integer ring $\frko_K$ and ad\`{e}le ring $\AA_K$. 
Denote by $K^\mathrm{ab}$ the maximal abelian extension over $K$ and by $\frkK_\infty$ the composition of all the $\ZZ_p$-extensions of $K$. 
Take the decomposition $\Gal(\frkK_\infty/K)\simeq\Gam^+\oplus\Gam^-$ so that the non-trivial element of $\Gal(K/\QQ)$ acts on $\Gam^\pm\simeq\ZZ_p$ by $\pm1$. 
Let $K_\infty^\pm$ be the subfield of $\frkK_\infty$ with $\Gal(K_\infty^\pm/K)=\Gam^\pm$. 
The $\ZZ_p$-extension $K_\infty^-/K$ is called anticyclotomic. We consider Iwasawa theory for $f$ over $K_\infty^-$. On the algebraic side, one considers the Pontryagin dual $\Sel(A_f/K^-_\infty)^\vee$ of the Selmer group $\Sel(A_f/K^-_\infty)$, which is known to be a finitely generated $\frko_E\powerseries{\Gamma^-}$-module. On the analytic side, one expects the existence of the $p$-adic $L$-function $L_p(f/K_\infty^-)\in\frko_E\powerseries{\Gamma^-}$ attached to $f$ which interpolates the central values of $L$-functions associated with $\rho_{f,p}$ twisted by characters of $\Gamma^-$, and then one could make the following anticyclotomic Iwasawa main conjecture for Siegel cusp forms.

\begin{conjecture} 
The characteristic ideal ${\rm char}_{\frko_E\powerseries{\Gamma^-}}\Sel(A_f/K^-_\infty)^\vee$ is generated by $L_p(f/K_\infty^-)$.
\end{conjecture} 

The main result of this paper is the construction of $L_p(f/K_\infty)$ when $f$ is a paramodular newform of square-free level.  Actually, we will construct a square root $\Tht_f$ of the anticyclotomic $p$-adic $L$-function. 

\subsection{Paramodular Siegel cusp forms}

The paramodular group of level $N$ is defined by 
\[{\rm K}(N)=\Sp_4(\QQ)\cap\begin{pmatrix} 
\ZZ & \ZZ & N^{-1}\ZZ & \ZZ \\
N\ZZ & \ZZ & \ZZ & \ZZ \\
N\ZZ & N\ZZ & \ZZ & N\ZZ \\ 
N\ZZ & \ZZ & \ZZ & \ZZ \end{pmatrix}. \] 
This subgroup is a good analogue of the congruence subgroup $\Gam_0(N)$ underlying the newform theory for $\GL_2$ (cf. \cite{RS}).  
Put 
\begin{align*}
\Sym_g&=\{z\in\Mat_g\;|\trs z=z\}, & 
\frkH_g&=\{Z\in\Sym_g(\CC)\;|\;\Im Z>0\}. 
\end{align*}
Throughout this paper we require $N$ to be square-free. 
Let $\pi$ be an irreducible cuspidal automorphic representation of $\PGSp_4(\AA)$ generated by a paramodular Siegel cuspidal Hecke eigenform 
\begin{align*}
f(Z)&=\sum_B\bfc_B(f)e^{2\pi\iu\tr(BZ)}, & 
Z&\in\frkH_2. 
\end{align*}
of genus $2$ and weight $\kap$ with respect to ${\rm K}(N)$. 
For each prime $\ell\nmid N$ we write $t_{\ell,1}$ and $t_{\ell,2}$ for the respective eigenvalues of the Hecke operators 
\begin{align*}
&\ell^{\kap-3}[\mathrm{K}(N)\diag[1,1,\ell,\ell]\mathrm{K}(N)], & 
&\ell^{2(\kap-3)}[\mathrm{K}(N)\diag[1,\ell,\ell^2,\ell]\mathrm{K}(N)]
\end{align*} 
acting on $f$. Let \[Q_\ell(X)=1-t_{1,\ell}X+(\ell t_{2,\ell}+(\ell^3+\ell)\ell^{2\kap-6})X^2-\ell^{2\kap-3}t_{1,\ell}X^3+\ell^{4\kap-6} X^4\] be the Hecke polynomial of $f$ at $\ell$. Then we have
\[Q_\ell(\ell^{-s})=L\left(s-\kap+\frac{3}{2},{\rm Spn}(\pi_\ell)\right).\]

We write $\Spn(\pi)$ for the strong lift of $\pi$ to an automorphic representation of $\GL_4(\AA)$ and $\Spn(\pi)_K$ for the base change of $\Spn(\pi)$ to $\GL_4(\AA_K)$. 
We consider its $L$-function twisted by Hecke characters $\nu$
\[L(s,\Spn(\pi)_K\otimes\nu)=\prod_\ell L(s,\Spn(\pi_\ell)_{K_\ell}\otimes\nu_\ell). \]

When $\ell$ does not divide $N$ and the conductor of $\nu$, the local $L$-factor is given as follows:  If $\ell=\frkl$ is inert in $K$, then \[L(s,\Spn(\pi_\ell)_{K_\ell}\otimes\nu_\ell)=Q_\ell(\lam_\frkl\ell^{3/2-\kap-s})Q_\ell(-\lam_\frkl\ell^{3/2-\kap-s}),\]
where $\lam_\frkl^2=\nu_\ell(\frkl)$, if $\ell=\frkl^2$ is ramified in $K$, then
\[L(s,\Spn(\pi_\ell)_{K_\ell}\otimes\nu_\ell)=Q_\ell(\lam_\frkl\ell^{3/2-\kap-s}),\]
where $\lam_\frkl=\nu_\ell(\frkl)$, and if $\ell=\frkl_1\frkl_2$ is split in $K$, then \[L(s,\Spn(\pi_\ell)_{K_\ell}\otimes\nu_\ell)=Q_\ell(\lam_{\frkl_1}\ell^{3/2-\kap-s})Q_\ell(\lam_{\frkl_2}\ell^{3/2-\kap-s}).\]
where $\lam_{\frkl_i}=\nu_\ell(\frkl_i)$ for $i=1,2$. 
The $L$-factors at prime factors of $N$ are given in (\ref{tag:L-factor}).  
It is conjectured that there is a bijection between isogeny classes of abelian surfaces $A/\QQ$ of conductor $N$ with $\End_\QQ A=\ZZ$ and such cusp forms $f$ with rational eigenvalues, up to scalar multiplication (see \cite{Yo1,BK}). 
In this case $T_f$ is the Tate module $\displaystyle{\lim_{\longleftarrow}}\,A[p^n]$ and the generalized BSD conjecture predicts that the vanishing order of  $L(s,\Spn(\pi)_K\otimes\nu)$ at the center $s=\frac{1}{2}$ coincides the dimension $\dim_\CC(A(K^-_\infty)\otimes\CC)^\nu$ of the $\nu$-eigenspace of the Mordell-Weil group of $A$ for an anticyclotomic character $\nu$ of $\Gam^-$. 

\subsection{The B\"{o}cherer conjecture}

We construct the anticyclotomic $p$-adic $L$-function attached to $\pi$ over $K$ with explicit evaluation formula for anticyclotomic characters of finite order. 
The key ingredient of our construction is the \emph{B\"{o}cherer conjecture} \cite{B1}, which is a special case of the refined Gross-Prasad conjecture formulated by Yifeng Liu \cite{L1} in full generality. 

Let $S$ be a positive definite half-integral symmetric matrix of size $2$ with determinant $\frac{\Delta_K}{4}$. 
The B\"{o}cherer conjecture relates the central $L$-values of $\pi$ to a square of the Bessel period defined by 
\[B_S^\nu(\phi)=\int_{K^\times\AA^\times\bsl\AA_K^\times}\int\phi\left(\begin{pmatrix} t & tz \\ 0 & (\det t)\trs t^{-1}\end{pmatrix}\right)\overline{\bfe(\tr(Sz))\nu(t)}\,\d z\d t, \]
where $z$ is integrated over symmetric matrices of size $2$ over $\AA/\QQ$, $\bfe$ denotes the standard additive character on $\AA/\QQ$ and $K^\times$ is identified with the subgroup $\{t\in\GL_2(\QQ)\;|\;\trs tSt=(\det t)S\}$. 
Let $\alp_{S,\nu_v}^\natural$ be a normalized local Bessel integral. 
Furusawa and Morimoto \cite[Theorem 1.2]{FM3} have recently proved the B\"{o}cherer conjecture. 

\begin{theorem}[Furusawa-Morimoto]
Assume that $\pi$ is tempered. 
For every anticyclotomic Hecke character $\nu:K^\times\CC^\times\AA^\times\bsl\AA_K^\times\to\CC^\times$ of $p$-power conductor and every nonzero cusp form $\phi=\otimes_v\phi_v\in\pi^D$
\[\frac{|B_S^\nu(\phi)|^2}{(\phi,\phi)}=\xi_\QQ(2)\xi_\QQ(4)\frac{\varLambda\bigl(\frac{1}{2},\Spn(\pi)_K\otimes\nu\bigl)}{2^{s_\pi}\varLambda(1,\pi,\ad)}\prod_v\frac{\alp^\natural_{S,\nu_v}(\phi_v,\phi_v)}{(\phi_v,\phi_v)}, \]
where $s_\pi=2$ or $1$ according as $\pi$ is endoscopic or not. 
\end{theorem}

\subsection{The ordinary hypothesis}


The Hecke polynomial $Q_p(X)$ of $f$ at $p$ can be factorized into 
\[Q_p(X)=(1-\alp_\calp X)(1-\beta_\calp X)(1-p^{2\kap-3}\alp^{-1}_\calp X)(1-p^{2\kap-3}\bet^{-1}_\calp X)\]
such that
\[0\leq\ord_p(\iot_p(\alp_\calp))\leq \ord_p(\iot_p(\bet_\calp))\leq \kap-\frac{3}{2}.\]
Let $\alp_\calq:=p^{2-\kap}\alp_\calp\bet_\calp$. In view of \cite[Th\'eor\`eme 1]{Urban2}, the Panchishkin hypothesis is equivalent to the the following \emph{Klingen $p$-ordinary hypothesis}: 
\begin{align}
 \ord_p(\iot_p(\alp_\calq))&=0. \tag{$\calq$-ord}
\end{align}

\begin{remark}
The Siegel $p$-ordinary hypothesis is 
\beq
\ord_p\iot_p(\alp_\calp)=0. \tag{$\calp$-ord}
\eeq 
The condition $(\calq)+(\calp)$ is referred to as the $p$-ordinary assumption relative to the Borel subgroup (cf. Remark \ref{rem:51}). 
\end{remark}

\subsection{Main theorem}
Let $\mathrm{rec}_K:K^\times\bsl\AA_K^\times\to\Gal(K^\mathrm{ab}/K)$ denote the geometrically normalized reciprocity law map.
Put $w_K=\sharp\frko_K^\times$. 
We view a character $\hat\nu:\Gam^-\to\bar\QQ^\times_p$ as a character of the Galois group of $K$ and associate a Hecke character 
\[\nu=\iot_\infty\circ\iot_p^{-1}\circ\hat\nu\circ\mathrm{rec}^{}_K:\AA_K^\times/K^\times\AA^\times\to\CC^\times. \] 
A character $\nu$ as above is usually referred to be anticyclotomic in the sense that $\nu$ is trivial on $\AA^\times$. 
We write $c(\nu)$ for the smallest non-negative integer $n$ such that $\nu_p$ is trivial on $\frko_{K_p}^\times\cap(1+p^n\frko_{K_p}^{})$. 


\begin{theorem}\label{thm:main_thm}
Assume that $N$ is square-free and that $K$ is split at each prime factor of $N$. 
Fix a decomposition $N\frko_K=\frkN_0^+\overline{\frkN_0^+}$.   
If $\pi$ is tempered and ($\calq$-ord) holds, then there is an element $\Tht_f\in\frko_E\powerseries{\Gam^-}$ such that for every finite-order character $\hat\nu:\Gam^-\to\bar\QQ^\times_p$ we have the following interpolation formula:
\begin{align*}
\hat\nu(\Tht_f)^2
=&\frac{\varLambda\bigl(\frac{1}{2},\Spn(\pi)_K\otimes\nu\bigl)}{\Omega_{\pi,1}}\cdot e(\pi_p,\nu_p)^2\\
&\times \nu^{-1}(\frkN_0^+)\cdot\alp_\calp^6\cdot 2^{2\kap-4}\cdot w_K^2\Delta_K^{\kap-1}\cdot N^{-1}, 
\end{align*}
where $e(\pi_p,\nu_p)$ is the $p$-adic multiplier defined by 
\begin{align}
e(\pi_p,\nu_p)&=(p^{\kap-1}\alp_\calq^{-1})^{c(\nu)}, \tag{$c(\nu)>0$}\\
e(\pi_p,\nu_p)&=\prod_{i=1}^2(1-\alp_\calp^{-1}\lam_{\frkp_i}p^{\kap-2})(1-\bet_\calp^{-1}\lam_{\frkp_i}p^{\kap-2}),  \tag{$c(\nu)=0$, $p=\frkp_1\frkp_2$}\\
e(\pi_p,\nu_p)&=(1-\alp_\calp^{-2}p^{2\kap-4})(1-\bet_\calp^{-2}p^{2\kap-4}), \tag{$c(\nu)=0$, $p=\frkp$}\\
e(\pi_p,\nu_p)&=(1-\alp_\calp^{-1}\lam_\frkp p^{\kap-2})(1-\bet_\calp^{-1}\lam_\frkp p^{\kap-2}) \tag{$c(\nu)=0$, $p=\frkp^2$}
\end{align}
and the complex number $\Omega_{\pi,1}$ is defined by 
\begin{align*}
\Omega_{\pi,1}&=\frac{\varLambda(1,\pi,\ad)}{\La f,f\Ra_{\rmK(N)}}, & 
\La f,f\Ra_{\rmK(N)}&=\int_{\rmK(N)\bsl\frkH_2}|f(Z)|^2(\det Y)^{\kap-3}\d Z. 
\end{align*}
Here $f$ is normalized so that all the Fourier coefficients of $f$ are real and contained in $E$ via $\iota_p:\bar\QQ_p\simeq\CC$ and some Fourier coefficient of $f$ is non-vanishing modulo the maximal ideal of $\frko_E$.
\end{theorem}

\begin{remark}\label{rem:12}
\begin{enumerate}
\renewcommand\labelenumi{(\theenumi)}
\item The imaginary quadratic field $K$ uniquely determines a factorization $N=N^+N^-$ with $N^+$ divisible only by primes that are split in $K$ and $N^-$ divisible only by primes that are inert or ramified in $K$. 
In Theorem \ref{thm:11}, we actually constructs the theta element $\Tht_f\in\frko_E\powerseries{\Gam^-}$ under the following \emph{Heegner hypothesis}: 
\begin{equation*}
\tag{Heeg}
\text{$N^-$ is the product of an even number of primes}.
\end{equation*}
If (Heeg) is not true, then $L\bigl(\frac{1}{2},\Spn(\pi)_K\otimes\nu\bigl)=0$ (see Remark \ref{rem:12}(\ref{rem:122})). 
Theorem \ref{thm:main_thm} is its special case where $N=N^+$ and $N^-=1$. 
\item\label{rem:122} Since endoscopic Siegel cusp forms of degree 2 are never paramodular (cf. the proof of Proposition 12.3 of \cite{RW3}), $s_\pi=1$. 
\item\label{rem:112} If $\kap>2$ and $\pi$ is not a Saito-Kurokawa lift, then $\pi$ is tempered by Proposition 8.1 of \cite{FM3}. 
\item The modified Euler factor $e(\pi_p,\nu_p)$ is compatible with the conjectural shape of $p$-adic $L$-functions due to Coates and Perrin-Riou. Indeed, let $M$ be the motive over $\QQ$ associated with $\rho_{f,p}|_{\Gal(\bar K/K)}\ot\nu$. Then $e(\pi_p,\nu_p)^2$ is the ratio between $\call_p^{(\rho)}(M)$ defined in \cite[(18), p.~109]{CP} and $L_p(M)=L\bigl(\frac{1}{2},\mathrm{Spn}(\pi_p)_{K_p}\ot\nu_p\bigl)$.
\end{enumerate} 
\end{remark}


\subsection{The construction of $\Tht_f$}
We sketch the construction of the theta element $\Theta_f$. 
Define ${\rm K}(N,p)$ to be the subgroup which consists of matrices $(b_{ij})\in{\rm K}(N)$ such that $b_{21},b_{31},b_{32},b_{34}, b_{41},b_{42}$ are divisible by $p$. 
We define Hecke operators on $S_\kap({\rm K}(N,p))$ by 
\begin{align*}
[\bfU_p^\calp h](Z)&=\sum_B\bfc_{pB}(h)e^{2\pi\iu p\tr(BZ)}, \\
[\bfU_p^\calq h](Z)&=\sum_{x=1}^p\sum_B\bfc_{\trs u_p(x)Bu_p(x)}(h)e^{2\pi\iu\tr(BZ)},    
\end{align*}
where $B$ runs over positive definite symmetric half-integral matrices of size $2$ and $u_p(x)=\begin{pmatrix} p & x \\ 0 & 1 \end{pmatrix}$. 
We define the $p$-stabilization $f^\ddagger\in S_\kap({\rm K}(N,p))$ of $f$ with respect to $\alp_\calq$ and $\alp_\calp$ by 
\[f^\ddagger=\alp_\calp^{-3}\alp_\calq^{-1}(\bfU_p^\calq-p^{\kap-1}\alp_\calp\bet_\calp^{-1})(\bfU^\calp_p-p^{2\kap-3}\alp_\calp^{-1})(\bfU_p^\calp-p^{2\kap-3}\bet_\calp^{-1})(\bfU^\calp_p-\bet_\calp)f. \]
This form $f^\ddagger$ is an eigenform of the operators $\bfU_p^\calq$ and $\bfU_p^\calp$ with eigenvalues $\alp_\calq$ and $\alp_\calp$, respectively. 
 
We begin with some notation. 
Let $\calr=\pMX{\ZZ}{\ZZ}{pN\ZZ}{\ZZ}$ be the standard Eichler order of level $pN$ in $D:=\Mat_2(\QQ)$. The group $\calr^\x$ acts on the set $\Psi\in \Hom(K,D)$ by $\Psi \cdot r\mapsto r^{-1}\Psi r$. Let \[\CM(K,D):=\Hom(K,D)/\calr^\x\]
be the set of $\calr^\x$-conjugacy classes of field homomorphisms from $K$ to $D$. 
For $\Psi\in\Hom(K,D)$, denote by $[\Psi]$ the $\calr^\x$-conjugacy class of $\Psi$ in $\CM(K,D)$. For a positive integer $c$, let $\calo_c$ be the order of $K$ of conductor $c$ and $K_c$ the associated ring class field. 
The conductor of a homomorphism $\Psi\in\Hom(K,D)$ is the unique positive integer $c$ such that $\Psi^{-1}(\calr)=\calo_c$. Let $\CM(\calo_c,R)$ be the set of $\calr^\x$-conjugacy classes of homomorphisms $\Psi\in\Hom(K,D)$ of conductor $c$, which we call the set of CM points of conductor $c$. 

The Galois group $\calg_c:=\Gal(K_c/K)$ acts on $\CM(\calo_c,\calr)$ in the following manner (\cite[p.133]{Gross}): for $\sigma\in \calg_c$ and $[\Psi]\in\CM(\calo_c,\calr)$, write $\sigma=\mathrm{rec}_K(a)$ for some $a\in \wh K^\x$ and decompose $\Psi(a)=\gamma\cdot u$ for some $u\in \wh \calr^\x$ and $\gamma\in \GL_2(\QQ)$ with $\det\gam>0$ by strong approximation. The action $[\Psi]^\sigma$ is defined by 
\[[\Psi]^\sigma:=[\gamma^{-1} \Psi\gamma].\]
To each $\Psi\in\Hom(K,D)$ of conductor $c$, we associate a unique half-integral symmetric positive definite matrix $S_\Psi$ defined by 
\[\pm S_\Psi=\pMX{0}{1}{-1}{0}\cdot \Psi(c\sqrt{-\Delta_K}/2).\]
Thus $[\Psi]\mapsto [S_\Psi]$ gives a map from $\CM(\calo_c,\calr)$ to $\calr^\x$-conjugacy classes of primitive half-integral symmetric positive definite matrices. For each non-negative integer $n$, we will choose special CM points $[\Psi_n]\in \CM(\calo_{p^n},R)$ of conductor $p^n$, and define the $n$-th theta element $\widetilde\Theta_n\in\frko_E[\calg_{p^n}]$ by
\[\widetilde\Theta_n=\alp_\calq^{-n}\sum_{\sigma\in \calg_{p^n}}\bfc_{S_{\Psi_n^\sigma}}(f^\ddagger)\cdot \sigma.\]
The fact that $f^\ddagger$ is an $\bfU_p^\calq$-eigenform allows us to make a good choice of CM points $\{[\Psi_n]\}_{n=1}^\infty$ such that $\widetilde\Theta_n$ is norm-compatible, i.e.
\[\mathit{\Pi}^{n+1}_n(\widetilde\Theta_{n+1})=\widetilde\Theta_n\] under the quotient map $\mathit{\Pi}^{n+1}_n:\calg_{p^{n+1}}\to \calg_{p^n}$, and hence we obtain the element $\widetilde\Theta_\infty:=\displaystyle{\lim_{\stackrel{\longleftarrow}{n}}}\,\widetilde\Theta_n\in \frko_{E}\powerseries{\Gal(K_{p^\infty}/K)}$. The theta element $\Theta_f$ is defined by $\Theta_f=\mathit{\Pi}^{K_{p^\infty}}_{K_\infty^-}(\widetilde\Theta_\infty)$ via the quotient map $\mathit{\Pi}^{K_{p^\infty}}_{K_\infty^-}:\Gal(K_{p^\infty}/K)\to\Gamma^-$. 
For each anticyclotomic character $\nu$ of conductor $p^n$,  the interpolation $\hat\nu(\Tht_f)$ is essentially the global Bessel period of $f^\ddagger$ with respect to $S_n$ and $\nu$, and the square of the global Bessel period is a product of the central $L$-value and local Bessel integrals by the B\"{o}cherer conjecture. 

\subsection{The interpolation formula}
To obtain the precise interpolation formula of $\Theta_f$, we have to evaluate the local Bessel integrals explicitly. To this end, we will construct a $\GSp_4(\QQ_v)$-equivariant isomorphism $M_v:\pi^{}_v\simeq\pi_v^\vee$. 
Choose an element $\bfJ\in\GL_2(\QQ)$ which satisfies $\trs\bfJ S\bfJ=S$ and $\det \bfJ=-1$. 
Then 
\[[\phi_v\otimes\phi'_v\mapsto \alp_{S,\nu_v}^\natural(\phi_v,M_v(\pi_v(\bft(\bfJ))\phi'_v))]\in\Hom_{R_S\times R_S}(\pi_v\boxtimes\pi_v,\nu_v^S\boxtimes\nu_v^S), \]
where we put $\bft(\bfJ)=\diag[\bfJ,-\trs\bfJ^{-1}]$ and define the character $\nu_v^S$ of $R_S=K_v^\times\rtimes\Sym_2(\QQ_v)$ by $\nu^S_v(t,z)=\nu_v(t)\bfe_v(\tr(Sz))$. 
The B\"{o}cherer conjecture relates the square $B^\nu_S(\phi)^2$ to the product of the central $L$-value and  
\[\alp_{S,\nu_v}^\natural(\phi^0_v,M_v(\pi_v(\bft(\bfJ))\phi^0_v)). \]
Our main task is to compute this local factor for a nice test vector $\phi^0_v\in\pi_v^{}$. 
If $v\neq p$ and both $\pi_v$ and $\nu_v$ are unramified, then $\phi^0_v$ is an unramified vector and the local Bessel integral has been calculated in \cite{L1,DPSS}. 
If $v$ divides $N$, then $\phi^0_v$ is a paramodular new vector of $\pi_v$ and its Bessel period will be computed in \S \ref{ssec:param}. 
The quaternion case is discussed in \S \ref{ssec:quaternion}. 
If $v=\infty$, then $\phi^0_\infty$ is a lowest weight vector and the computation is done in \S \ref{ssec:archimedean}. 
When $\kap>2$, the archimedean Bessel integral has been computed in \cite{DPSS} by a method suggested by Kazuki Morimoto. 
Our computation is different and includes the case $\kap=2$. 

 The calculation of the local integral at $v=p$ is one of the main novelties in this paper. We first construct an ordinary projector $e^0_{\mathrm{ord},p}$ in Section \ref{sec:7} and compute the Bessel integral of $e^0_{\mathrm{ord},p}\phi^0_p$ in Section \ref{sec:expIII}. 
To that end, we will construct a local Bessel period $\bfB^{\calw,\pi_p}_{S,\nu_p}\in\Hom_{R_S}(\pi_p,\nu_p^S)$ so that $\bfB^{\calw,\pi_p}_{S,\nu_p}(e^0_{\mathrm{ord},p}\phi^0_p)$ is computable. 
By uniqueness we are led to a functional equation and a factorization 
\begin{align*}
\bfB_{S,\nu_p}^{\calw,\pi_p^\vee}\circ M_p&=c(\pi_p,\nu_p)\bfB^{\calw,\pi_p}_{S,\nu_p}, & 
\alp^\natural_{S,\nu_p}&=c'(\pi_p,\nu_p)\bfB_{S,\nu_p}^{\calw,\pi_p}\otimes\bfB_{-S,\nu_p^{-1}}^{\calw,\pi_p^\vee}.
\end{align*}
We determine the proportionality constants in Propositions \ref{P:factorB.3}, \ref{P:factorB.4} and \ref{prop:fq}. 

\subsection*{Acknowledgement}
Yamana is partially supported by JSPS Grant-in-Aid for Scientific Research (C) 18K03210 and (B) 19H01778. 
He would like to thank Michael Harris for very stimulating discussions and Kazuki Morimoto and Hiraku Atobe for helpful discussions. 
Hsieh is partially supported by  a MOST grant 103-2115-M-002-012-MY5. 
This work was done when Yamana stayed in the Max Planck Institute for Mathematics, and he would like to thank the institute for the excellent working environment. 
This work was partially supported by Osaka City University Advanced Mathematical Institute (MEXT Joint Usage/Research Center on Mathematics and Theoretical Physics). 

\section{The basic setting}\label{sec:notation}


\subsection{Notation}

Besides the standard symbols $\ZZ$, $\QQ$, $\RR$, $\CC$, $\ZZ_p$, $\QQ_p$ we denote by $\NN$ the set of positive integers, by $\RR^\times_+$ the group of strictly positive real numbers and by $\CC_p=\widehat{\overline{\QQ}}_p$ the completion of an algebraic closure of the $p$-adic field $\QQ_p$.  
If $x$ is a real number, then we put $[x]=\max\{i\in\ZZ\;|\;i\leq x\}$.
For any finite set $A$ we denote by $\sharp A$ the number of elements in $A$. 
For any set $X$ we denote by $\1_X$ the characteristic function of $X$. 
When $G$ is a topological group, we write $G^\circ$ for its connected component of the identity. 
When $G$ is locally compact and abelian, we denote the group of quasi-characters of $G$ by $\Ome(G)$ and the subgroup of unitary characters of $G$ by $\Ome^1(G)$. 

Let $F$ be a local field of characteristic zero with normalized absolute value $\ome_F=|\cdot|_F$. 
We often simply write $|x|=|x|_F$ and $\ome(x)=\ome_F(x)$ for $x\in F^\times$ if its meaning is clear from the context without possible confusion. 
The group $\Ome(F^\times)^\circ$ (resp. $\Ome^1(F^\times)^\circ$) consists of homomorphisms of the form $\ome^s_F$ with $s\in\CC$ (resp. $s\in\iu\RR$). 
Let $\sig\in\Ome(F^\times)$. 
Define $\Re\sig$ as the unique real number such that $\sig\ome_F^{-\Re\sig}\in\Ome^1(F^\times)$. 

Let $F$ be nonarchimedean. 
We denote the integer ring of $F$ by $\frko_F$, the maximal ideal of $\frko_F$ by $\frkp$ and the order of the residue field $\frko_F/\frkp$ by $q$ and the different of $F$ by $\frkd_F$.  
Fix a prime element $\vpi$ of $\frko_F$. 
When $\sig\in\Ome(F^\times)^\circ$, we put $L(s,\sig)=\frac{1}{1-\sig(\vpi)q^{-s}}$. 
Otherwise we put $L(s,\sig)=1$. 
We extend $|\cdot|_F$ to fractional ideals of $\frko_F$ by $|\vpi^i|_F=q^{-i}$. 
Fix a generator $d_F$ of $\frkd_F$ and a nontrivial additive character $\psi$ on $F$. 
Put $\zet(s)=\zet_F(s)=\frac{1}{1-q^{-s}}$. 
In our later discussion we mostly let $\psi$ be trivial on $\frko_F$ but not trivial on $\frkp^{-1}$. When the residual characteristic of $F$ is $p$, we define the character $\psi^F$ of $F$ by $\psi^F(x)=e^{-2\pi\iu y}$ with $y\in\QQ$ such that $\Tr_{F/\QQ_p}(x)-y\in\ZZ_p$. 
Let $\d x$ be the self-dual Haar measure on $F$ with respect to the pairing $(x,y)\mapsto\psi^F(xy)$. 
This measure gives $\frko_F$ the volume $|\frkd_F|^{1/2}$. 
The Haar measure $\d^\times x$ of $F^\times$ is normalized by $\d^\times x=\zet(1)\frac{\d x}{|x|_F}$. 
When $K$ is a quadratic \'{e}tale algebra over $F$, let $\d t$ be the quotient measure of the Haar measures of $K^\times$ and $F^\times$. 

For an admissible representation $(\pi,V)$ of a reductive group $G$ over $F$ we will write $\pi^\vee$ for its contragredient representation. 
We occasionally identify the space $V$ with $\pi$ itself when there is no danger of confusion. 
When $G=\GL_n(F)$ and $\mu\in\Ome(F^\times)$, we define a representation $\pi$ on the same space $V_\pi$ by $(\pi\otimes\mu)(g)=\mu(\det g)\pi(g)$. 
When $\vPi$ is an irreducible admissible representation of $\GL_n(F)$, we write $\vPi^K$ for its base change to $\GL_n(K)$ and write $L(s,\vPi)$ for its Godement-Jacquet $L$-factor. 
Given an irreducible admissible representation $\pi$ of $\GSp_4(F)$, we denote its transfer to $\GL_4(F)$ by $\Spn(\pi)$ and its adjoint $L$-factor by $L(s,\pi,\mathrm{ad})$. 
When $\pi$ is not supercuspidal, these $L$-parameter and degree $10$ $L$-factor are explicitly computed in Table A.7 of \cite{RS} and \cite{AS2}, respectively. 


\subsection{Quaternion unitary groups}

For any ring $R$ we denote by $\Mat_{i,j}(R)$ the set of $i\times j$-matrices with entries in $R$ and write $\Mat_g(R)$ in place of $\Mat_{g,g}(R)$. 
The group of all invertible elements of $\Mat_g(R)$ and the set of symmetric matrices of size $g$ with entries in $R$ are denoted by $\GL_g(R)$ and $\Sym_g(R)$, respectively. 
We sometimes write $R^\times=\GL_1(R)$. 
The subgroup $B_g(R)$ consists of upper triangular matrices in $\GL_g(R)$. 
For matrices $B\in\Sym_g(R)$ and $G\in\Mat_{g,m}(R)$ we use the abbreviation $B[G]=\trs GBG$, where $\trs G$ is the transpose of $G$.  
If $A_1, \dots, A_r$ are square matrices, then $\diag[A_1, \dots, A_r]$ denotes the matrix with $A_1, \dots, A_r$ in the diagonal blocks and $0$ in all other blocks.
Let $\ono_g$ be the identity matrix of degree $g$. 
When $G$ is a reductive algebraic group and $Z$ is its center, we write $\mathrm{P}G$ for the adjoint group $G/Z$.

Let $D$ be a quaternion algebra over a field $F$. 
We denote by $x\mapsto\bar x$ the main involution of $D$ and by $\trs\bar A$ the conjugate transpose of a matrix $A\in\Mat_n(D)$. 
Let $\Nr^D_F(x)=x\bar x$ and $\Tr^D_F(x)=x+\bar x$ denote the reduced norm and the reduced trace of $x\in D$. 
Put $D_-=\{z\in D\;|\;\bar z=-z\}$. 
We frequently regard $D$ as an algebraic variety over $F$ and consider the algebraic group $\GU_2^D$ which associates to any $F$-algebra $R$ the group
\[\GU^D_2(R)=\left\{h\in \GL_2(D\ot_F R)\;\biggl|\; h\pMX{0}{1}{1}{0}\!{^t\!\bar h}=\lambda(h)\pMX{0}{1}{1}{0},\,\lambda(h)\in R^\x\right\}, \]
where $\lam$ is called the similitude character of $\GU_2^D$. 
We define homomorphisms $\bfm,\bft:D^\times\to\GU_2^D$, $\bfn:D_-\to\GU_2^D$ and $\bfd:F^\times\to\GU_2^D$ by 
\begin{align*}
\bfm(A)&=\begin{pmatrix} A & 0 \\ 0 & \bar A^{-1}\end{pmatrix}, &
\bft(A)&=\begin{pmatrix} A & 0 \\ 0 & A\end{pmatrix}, & 
\bfn(z)&=\begin{pmatrix} 1 & z \\ 0 & 1\end{pmatrix}, &
\bfd(\lam)&=\begin{pmatrix} 1 & 0 \\ 0 & \lam\end{pmatrix} 
\end{align*}
and denote the parabolic subgroup of $\GU_2^D$ with a Levi factor $\bfd(F^\times)\bfm(D^\times)$ and the unipotent radical $\bfn(D_-)$ by $\calp$. 

Fix $S\in D_-$ with $d_0=S^2\neq 0$. 
Put $K=F+FS\subset D$. 
We choose an element $\bfJ\in D_-$ such that $\bfJ t\bfJ^{-1}=\bar t$ for $t\in K$. 
Then $K\simeq F(\sqrt{d_0})$ and $D=K+K\bfJ$. 
Let $R_S=\bft(K^\times)\bfn(D_-)\simeq K^\times\ltimes D_-$ be a subgroup of $\calp$. 

Let $\GSp_{2g}$ be the symplectic similitude group of rank $g$ defined by
\begin{align*}
\GSp_{2g}&=\{h\in\GL_{2g}\;|\;hJ_g\trs h=\lam_g(h)J_g,\;\lam_g(h)\in\GL_1\}, & 
J_g&=\begin{pmatrix} 0 & \ono_g \\ -\ono_g & 0\end{pmatrix}.
\end{align*}
Put $\U_2^D=\ker\lam$ and $\Sp_g=\ker\lam_g$. 
We define the homomorphisms 
\begin{align*}
\bfm&:\GL_g\times\GL_1\to\GSp_{2g}, & 
\bfn&:\Sym_g\to\GSp_{2g}
\end{align*}
similarly by 
\begin{align*}
\bfm(A,\lam)&=\begin{pmatrix} A & 0 \\ 0 & \lam\trs A^{-1}\end{pmatrix}, & 
\bfn(z)&=\begin{pmatrix} \ono_g & z \\ 0 & \ono_g\end{pmatrix}. 
\end{align*}
We write
\begin{align*}
\bfm(A)&=\bfm(A,1), & 
\bft(A)&=\bfm(A,\det A), & 
\bfd(\lam)&=\bfm(\ono_g,\lam). 
\end{align*}
Define a maximal parabolic subgroup $\calp_g=\calm_gN_g$ of $\GSp_{2g}$ by 
\begin{align*}
\calm_g&=\{\bfm(A,\lam)\;|\;A\in\GL_g,\;\lam\in\GL_1\}, &
N_g&=\{\bfn(z)\;|\;z\in\Sym_g\}
\end{align*} 
and a Borel subgroup of $\GSp_{2g}$ by  
\[\calb_g=\{\bfm(A,\lam)\bfn(z)\;|\;A\in B_g,\;\lam\in\GL_1,\; z\in\Sym_g\}.\]
Note that 
\begin{align*}
\PGSp_2&\simeq\SO(2,1), &
\PGSp_4&\simeq\SO(3,2), &
\PGU_2^D&\simeq\SO(4,1), &  
\calb_1&\simeq B_2. 
\end{align*}
We write $U_g$ and $\calu_g$ for the unipotent radicals of $B_g$ and $\calb_g$, respectively. 

We include the case in which $D$ is the matrix algebra $\Mat_2(F)$. 
In this case 
\begin{align*}
\GU_2^D&\simeq\GSp_4, & D_-&\simeq\Sym_2, & \calp&\simeq\calp_2.  
\end{align*}


\subsection{Abstract Bessel integrals for $\GU_2^D$} 

Let $F$ be a local field and $(\pi,V)$ an irreducible admissible representation of $\PGU_2^D(F)$. 
Since $\pi\simeq\pi^\vee$, we have a $\GU_2^D(F)$-invariant bilinear perfect pairing $b:V\times V\to\CC$. 
Given a pair $\phi_1,\phi_2\in V$, we define the matrix coefficient $\Phi_{\phi_1,\phi_2}:\GU_2^D(F)\to\CC$ by 
\begin{align*}
\Phi_{\phi_1,\phi_2}(g)&=\Phi^\bfJ_{\phi_1,\phi_2}(g)=b(\pi(g)\phi_1,\pi(\bfm(\bfJ,-1))\phi_2). 
\end{align*}

\begin{definition}\label{def:21}
Let $\calu$ be a unipotent algebraic group over a $\frkp$-adic field $F$ and $f$ a smooth function on $\calu(F)$. 
We say that $f$ has a stable integral over $\calu(F)$ if there is a compact open subgroup $U$ of $\calu(F)$ such that for any open compact subgroup $U'$ containing $U$ 
\[\int_{U'}f(z)\,\d z=\int_Uf(z)\, \d z. \]
In this case we write $\int^\st_{\calu(F)}f(z)\,\d z=\int_Uf(z)\, \d z$. 
\end{definition}
   
We associate to $S\in D_-$ and $\Lambda$ the character $\Lambda^S$ of $R_S$ by $\Lambda^S(\bft(t)\bfn(z))=\Lambda(t)\psi(\Tr^D_F(Sz))$. 
By \cite{LM,L1,FG} the following stable integral of a matrix coefficient exists for each $t\in K^\times$: 
\[B^\psi_S(\phi_1,\phi_2,t)=\int^{\st}_{D_-}\Phi_{\phi_1,\phi_2}(\bfn(z)\bft(t))\overline{\psi(\Tr^D_F(Sz))}\,\d z. \] 

\begin{definition}[abstract Bessel integrals relative to $S$ and $\Lambda$]\label{def:22}
We define  
\[B_S^\Lambda(\phi_1,\phi_2)=\int_{F^\times\bsl K^\times}\int^{\st}_{D_-}\Phi_{\phi_1,\phi_2}(\bfn(z)\bft(t))\overline{\psi(\Tr^D_F(Sz))}\Lambda(t)^{-1}\,\d z\d t \]
for $\Lambda\in\Ome(F^\times\bsl K^\times)$ and $\phi_1,\phi_2\in V$ whenever the integral above converges. 
\end{definition}

When $\pi$ is tempered and $\Lambda$ is unitary, the iterated integral on the right-hand side converges. 
We give a direct proof for representations of our interest in Lemma \ref{lem:conv} below. 
We know that  
\[\dim_\CC\Hom_{R_S}(\pi,\Lambda^S)\leq 1 \]
by Corollary 15.3 of \cite{GGP}. 
It is important to note that 
\begin{align*}
\Hom_{R_S\times R_S}(\pi\boxtimes\pi,\Lambda^S\boxtimes\Lambda^S)&=\CC B_S^\Lambda,   
\end{align*}
i.e., $B_S^\Lambda$ is a basis vector of this zero or one-dimensional spaces. 

\begin{remark}
When $\pi$ is not square-integrable, the Bessel integral may diverge and is defined via regularization.
Yifeng Liu \cite{L1} has constructed a regularization of the archimedean Bessel integral in general. 
In \S \ref{ssec:archimedean} we will regularize the Bessel integrals of matrix coefficients of lowest weight representations of $\GSp_4(\RR)$ of scalar weight via a different way. 
\end{remark}

When $D=\Mat_2(F)$, we associate to $S\in\Sym_2(F)$ with $\det S\neq 0$ the Bessel integral $B_S^\Lam$ in a similar manner. 

\begin{remark}\label{rem:21}
Let $A\in\GL_2(F)$, $\lam\in F^\times$ and $t\in T$. 
Put 
\begin{align*}
S'&=\lam^{-1}S[A], & 
T'&=A^{-1}TA, & 
t'&=A^{-1}tA, & 
\bfJ'&=A^{-1}\bfJ A.  
\end{align*}
We define $\Lambda'\in\Ome(T')$ by $\Lambda'(t')=\Lambda(At'A^{-1})$. 
Since 
\[\Phi^\bfJ_{\pi(\bfm(A,\lam))\phi_1,\pi(\bfm(A,\lam))\phi_2}(\bfn(z)\bft(t))=\Phi^{\bfJ'}_{\phi_1,\phi_2}(\bfn(\lam A^{-1}z\trs A^{-1})\bft(t')), \]
it follows that   \[B_S^\Lambda(\pi(\bfm(A,\lam))\phi_1,\pi(\bfm(A,\lam))\phi_2,t)=|\lam|^{-3}|\det A|^3B_{S'}^\Lambda(\phi_1,\phi_2,t').\]
In particular, if we define the additive character $\psi_\lam$ by $\psi_\lam(x)=\psi(\lam x)$, then  
\begin{align*} B^\psi_S(\pi(\bfd(\lam))\phi_1,\pi(\bfd(\lam))\phi_2,t)&=|\lam|^{-3}B^{\psi_{\lam^{-1}}}_S(\phi_1,\phi_2,t), \\
B_S^\Lambda(\pi(\bfm(A,\lam))\phi_1,\pi(\bfm(A,\lam))\phi_2)&=|\lam|^{-3}|\det A|^3B_{S'}^{\Lambda'}(\phi_1,\phi_2).  \end{align*}
\end{remark}

\subsection{Hypotheses}

We switch to the global setting. 
Let $\pi\simeq\otimes_v'\pi_v$ be a unitary irreducible cuspidal tempered automorphic representation of $\PGSp_4(\AA)$ generated by a scalar valued degree two Siegel cuspidal Hecke eigenform $f$ of weight $\kap\geq 2$ and square-free paramodular level $N$. 

Fix an imaginary quadratic field $K$. 
We decompose $N$ as $N=N^+N^-$, where each prime factor of $N^+$ is split in $K$ and each prime factor of $N^-$ is inert or ramified in $K$. 
We assume the following \emph{Heegner hypothesis}: 
\begin{equation*}\label{Heeg}
\tag{Heeg}
\text{$N^-$ is the product of an even number of primes}.
\end{equation*}

Then there is an indefinite quaternion algebra $D$ that is ramified precisely at the prime factors of $N^-$. 
We consider the following inner form of $\GSp_4$:  
\[\GU_2^D=\biggl\{g\in\GL_2(D)\;\biggl|\;g\begin{pmatrix} 0 & 1 \\ 1 & 0\end{pmatrix}\trs\bar g=\lam(g)\begin{pmatrix} 0 & 1 \\ 1 & 0\end{pmatrix}\biggl\}, \]
where $\bar\cdot$ denotes the main involution of $D$. 
Furusawa and Morimoto \cite{FM3} have established the B\"{o}cherer conjecture more generally for $\GU_2^D$ (see Theorem \ref{coj:41}). 

When $\rho$ is a discrete series representation of $\GL_2(\QQ_q)$, we write $\rho^D$ for its Jacquet-Langlands lift to $D^\times(\QQ_q)$. 
We define the representation $\pi^D\simeq\otimes_v'\pi^D_v$ of $\PGU_2^D(\AA)$ by $\pi^D_v\simeq\pi^{}_v$ for $v\nmid N^-$ and by $\pi^D_q\simeq\rho^D_q\rtimes\sig^{}_q$ for prime factors $q$ of $N^-$, where we write $\pi_q\simeq\rho_q\rtimes\sig_q$ (cf. Remark \ref{rem:paramodular}). 
When $N^-\neq 1$, we will make use of the inner form transfer on $\PGSp_4$ established in \cite[Theorem 11.4]{RW3}, namely \emph{Jacquet-Langlands correspondence} between $\PGSp_4$ and $\mathrm{PGU}_2^D$ (cf. \cite[Proposition 12.3]{RW3} and Remarks \ref{rem:paramodular} and \ref{rem:IIaG}): 

\medskip

\noindent (JL) The representation $\pi^D$ occurs in the space of cusp forms on $\PGU_2^D(\AA)$ with multiplicity one. 

\medskip 


\subsection{Main theorem}
Put $w_K=\sharp\frko_K^\times$. 
To each character $\hat\nu:\Gam^-\to\bar\QQ^\times_p$ we associate a Hecke character $\nu=\iot_\infty\circ\iot_p^{-1}\circ\hat\nu\circ\mathrm{rec}^{}_K$ of $K$, where $\mathrm{rec}_K:K^\times\bsl\AA_K^\times\to\Gal(K^\mathrm{ab}/K)$ is the geometrically normalized reciprocity law map. 
We write $c(\nu)$ for the smallest non-negative integer $n$ such that $\nu_p$ is trivial on $\frko_{K_p}^\times\cap(1+p^n\frko_{K_p}^{})$. 
Fix a decomposition $N^+\frko_K=\frkN_0^+\overline{\frkN_0^+}$. 
For each prime factor $\ell$ of $N$ we write $\eps_\ell(f)=\vep\bigl(\frac{1}{2},\Spn(\pi_\ell)\bigl)$ for the eigenvalue of the Atkin-Lehner involution at $\ell$. 
Put $\eps_{N^-}(f)=\prod_{\ell|N^-}\eps_\ell(f)$. 

\begin{theorem}\label{thm:11}
Assume the hypotheses (Heeg) and ($\calq$-ord) are true for $\pi$, $K$ and $p$. 
Then there exist an explicitly given complex number $\Omega_{\pi,N^-}\in\CC^\times$ and an element $\Tht_f\in\frko_E\powerseries{\Gam^-}$ with the following interpolation formula
\begin{align*}
\hat\nu(\Tht_f)^2
=&\frac{\varLambda\bigl(\frac{1}{2},\Spn(\pi)_K\otimes\nu\bigl)}{\Omega_{\pi,N^-}}\cdot e(\pi_p,\nu_p)^2\cdot \nu^{-1}(\frkN_0^+)\cdot\alp_\calp^6\\
&\times 2^{2\kap-3-\ell(\pi)}\cdot w_K^2\Delta_K^{\kap-1}\eps_{N^-}(f)\cdot N^{-1}\prod_{\ell|(N^-,\Del_K)}(1-\eps_\ell(f)) 
\end{align*}
for every finite-order character $\hat\nu:\Gam^-\to\bar\QQ^\times_p$, where $e(\pi_p,\nu_p)$ is the $p$-adic multiplier defined by 
\begin{align}
e(\pi_p,\nu_p)&=(p^{\kap-1}\alp_\calq^{-1})^{c(\nu)}, \tag{$c(\nu)>0$}\\
e(\pi_p,\nu_p)&=\prod_{i=1}^2(1-\alp_\calp^{-1}\lam_{\frkp_i}p^{\kap-2})(1-\bet_\calp^{-1}\lam_{\frkp_i}p^{\kap-2}),  \tag{$c(\nu)=0$, $p=\frkp_1\frkp_2$}\\
e(\pi_p,\nu_p)&=(1-\alp_\calp^{-2}p^{2\kap-4})(1-\bet_\calp^{-2}p^{2\kap-4}), \tag{$c(\nu)=0$, $p=\frkp$}\\
e(\pi_p,\nu_p)&=(1-\alp_\calp^{-1}\lam_\frkp p^{\kap-2})(1-\bet_\calp^{-1}\lam_\frkp p^{\kap-2}). \tag{$c(\nu)=0$, $p=\frkp^2$}
\end{align}
\end{theorem}

\begin{remark}\label{rem:12}
\begin{enumerate}
\renewcommand\labelenumi{(\theenumi)}
\item\label{rem:121} The complex number $\Omega_{\pi,N^-}$ is given in Definition \ref{def:period}. 
\item\label{rem:122} Assume that $\eps_\ell(f)=-1$ for every prime factor $\ell$ of $(D_K,N^-)$. 
Then $\vep\bigl(\frac{1}{2},\Spn(\pi)_K\otimes\nu\bigl)=(-1)^{t(N^-)}$, where $t(N^-)$ is the number of prime factors of $N^-$. 
In particular, if (Heeg) is not true, then $L\bigl(\frac{1}{2},\Spn(\pi)_K\otimes\nu\bigl)=0$. 
\item\label{rem:111} 
We will construct the element $\Tht_f\in E\powerseries{\Gam^-}$ with the interpolation property without ($\calq$-ord) more generally for Hilbert-Siegel cusp forms. 
\item If $\nu$ has infinite order, then so does $\nu_\infty$ and so by Theorem 3.10 of \cite{PS} $B_S^\nu(\phi)=\alp_{S,\nu_\infty}^\natural(\vph,\vph')=0$ for all $\phi\in\pi^D$ and $\vph,\vph'\in\pi_\infty^D$. 
\item If Conjectures 9.4.2 and 9.5.4 of \cite{A} hold for $\GU_2^D$, then $2^{\ell(\pi)}$ coincides with the order of $S$-group of the Arthur parameter of $\pi^D$. 
\end{enumerate} 
\end{remark}


\section{Local Bessel integrals for $\GSp_4$}\label{sec:bessel1}


\subsection{Explicit Bessel integrals} 

Fix $S=\begin{pmatrix} a_0 & \frac{b_0}{2} \\ \frac{b_0}{2} & c_0 \end{pmatrix}\in\Sym_2(F)$. 
Assume that $\det S\neq 0$. 
Put  
\begin{align*}
d_0&=-4\det S=b_0^2-4a_0c_0, \\
K&=K_S=F(\sqrt{d_0}), \\
T&=T_S=\{A\in\GL_2(F)\;|\;\trs ASA=(\det A)S\}. 
\end{align*}

We denote the nontrivial automorphism of $K$ over $F$ by $t\mapsto\bar t$. 
Put $\Tr(t)=\Tr^K_F(t)=t+\bar t$ and $\Nr(t)=\Nr^K_F(t)=t\bar t$ for $t\in K$. 
We write $\frkr$ for the maximal order of $K$ and $\tau_{K/F}:F^\times\to\{\pm 1\}$ for the character of $F^\times$ whose kernel is $\Nr(K^\times)$. One can verify that 
\[T=\left\{\begin{pmatrix} x-y\frac{b_0}{2} & -yc_0 \\ ya_0 & x+y\frac{b_0}{2}\end{pmatrix} \biggl|\; x,y\in F,\;x^2-\frac{d_0}{4}y^2\neq 0\right\}. \]
We identify $T$ with $K^\times$ via the map
\[x+\frac{y}{2}\sqrt{d_0}\mapsto \begin{pmatrix} x-y\frac{b_0}{2} & -yc_0 \\ ya_0 & x+y\frac{b_0}{2}\end{pmatrix}. \] 
We regard characters of $K^\times$ as those of $T$. 

Let $\pi_0$ be an irreducible admissible unitary infinite dimensional representation of $\PGL_2(F)$. 
It possesses a Whittaker model, i.e., there is a functional 
\[0\neq \calw\in\Hom_{U_2}(\pi_0,\psi)=\Hom_{\GL_2(F)}(\pi_0,\Ind^{\GL_2(F)}_{U_2}\psi), \]
which is unique up to scalar multiple. 
Let $\calw(\pi_0,\psi)$ be the space of functions of the form $\calw_f(g)=\calw(\pi_0(g)f)$ with $f\in\pi_0$. 
We sometimes identify $V_\pi$ with $\calw(\pi_0,\psi)$. 
Since $\pi_0^{}\simeq\pi_0^\vee$, we can define the $\GL_2(F)$-invariant pairing 
\begin{align*}
b_\calw&:\pi_0\times\pi_0\to\mathbb{C}, & 
b_\calw(f,f')&=\int_{F^\times}\calw_f(\bft(a))\calw_{f'}(\bft(-a))\,\d^\times a. 
\end{align*}
For $\sig\in\Ome(F^\times)$ we consider the induced representation 
\[\pi=I(\pi_0,\sig):=\Ind_{\calp_2}^{\GSp_4(F)}(\pi_0\otimes\sig^{-1})\boxtimes\sig=\pi_0\otimes\sig^{-1}\rtimes\sig. \]
It is noteworthy that $\pi$ has trivial central character and can be viewed as a representation of $\PGSp_4(F)\simeq\SO(2,3)$. 

We normalize the Haar measure $\d z$ on $\Sym_2(F)$ so that $\Sym_2(\frko_F)$ has volume $1$. 
Choose a right invariant measure $\d g$ on $\calp_2\backslash\GSp_4(F)$ such that 
\beq\label{E:measure}
\int_{\calp_2\backslash\GSp_4(F)}f(g)\,\d g=\int_{\Sym_2(F)}f(w_s\bfn(z))\,\d z
\eeq
 for all $f\in\Ind_{\calp_2}^{\GSp_4(F)}\delta_{\calp_2}^{1/2}$. 
We associate to any $\GL_2(F)$-invariant pairing 
\[\calb:\pi_0\times\pi_0\to\mathbb{C}\]
the $\GSp_4(F)$-invariant pairing $\calb^\sharp:I(\pi_0,\sig)\times I(\pi_0,\sig^{-1})\to\mathbb{C}$ by
\beq
\calb^\sharp(\phi,\phi')=\int_{\calp_2\backslash\GSp_4(F)}\calb(\phi(g),\phi'(g))\,\d g. \label{E:2}
\eeq
We use $b_\calw^\sharp$ to identify $I(\pi_0,\sig^{-1})$ with the contragredient representation $\pi^\vee$. 

We shall study the Bessel integral for the representations of the form $I(\pi_0,\sig)$, and  in addition, we will explicitly factorize it into a product of two appropriate local Bessel periods when $K$ is split or $\pi_0$ is a principal series representation. 
In what follows, we fix $D'\in \frko_F$ and put $\tht=\frac{D'+\sqrt{d_0}}{2}\in\frkr$. 
Fix an element $\bfJ\in\GL_2(F)$ such that $t\bfJ=\bfJ\bar t$ for $t\in T$. 
There is no loss of generality by letting 
\begin{align}
S&=\begin{pmatrix} 1 & -\frac{\Tr(\tht)}{2} \\ -\frac{\Tr(\tht)}{2} & \Nr(\tht) \end{pmatrix}, & 
\bfJ&=\begin{pmatrix} -1 & \Tr(\tht) \\ 0 & 1 \end{pmatrix} \label{tag:22}
\end{align}
thanks to Remark \ref{rem:21}. 
The embedding $\iota:K\hookrightarrow\Mat_2(F)$ attached to this $S$ is 
\begin{align}
t=a\tht+b&\mapsto\iota(t)=\begin{pmatrix}
b+a\Tr(\tht) & -a\Nr(\tht) \\ a & b
\end{pmatrix} & 
(a,b&\in F). \label{tag:23}
\end{align}

We introduce the intertwining operator 
\[M(\pi_0,\sig): I(\pi_0,\sig)\to I(\pi_0,\sig^{-1}), \] 
defined for $\Re\sig\ll 0$ by the convergent integral
\[[M(\pi_0,\sig)\phi](g)=\int_{\Sym_2(F)}\phi(J_2\bfn(z)g)\,\d z, \]
and by meromorphic continuation otherwise. 
A normalized intertwining operator is defined by setting 
\[M^*(\pi_0,\sig)=\gamma(0,\pi_0\otimes\sig^{-1},\psi)\gamma(0,\sig^{-2},\psi)M(\pi_0,\sig). \]

The character $\psi_S:\Sym_2(F)\to \mathbb{C}^\times$ is defined by $\psi_S(\bfn(z)):=\psi(\tr(Sz))$.

\begin{definition}[explicit Bessel integrals relative to $S$ and $\Lambda$]\label{D:Bintegral}
We define 
\[J_{S,\Lambda}^\calw\in\Hom_{R_S\times R_S}(I(\pi_0,\sig)\boxtimes I(\pi_0,\sig^{-1}),\Lambda^S\boxtimes\Lambda^S)\] 
as in Definition \ref{def:22} by the integral
\[J_{S,\Lambda}^\calw(\phi,\phi')=\int_{F^\times\backslash K^\times}\int_{\Sym_2(F)}^\st b_\calw^\sharp(\pi(\bft(t)\bfn(z))\phi,\pi^\vee(\bft(\bfJ))\phi')\overline{\psi_{S}(z)\Lambda(t)}\d z\d t \]
for $\phi\in I(\pi_0,\sig)$ and $\phi'\in I(\pi_0,\sig^{-1})$.  
Furthermore we define 
\[B_{S,\Lambda}^\calw\in\Hom_{R_S\times R_S}(I(\pi_0,\sig)\boxtimes I(\pi_0,\sig),\Lambda^S\boxtimes\Lambda^S)\] 
by 
\begin{align*}
B_{S,\Lambda}^\calw(\phi_1,\phi_2)&:=J_{S,\Lambda}^\calw(\phi_1,M^*(\pi_0,\sig)\phi_2), & 
\phi_1,\phi_2&\in I(\pi_0,\sig). 
\end{align*} 
\end{definition}
Clearly, Definition \ref{D:Bintegral} is independent of the choice of $\bfJ$.  


\subsection{Bessel periods} 

We have introduced the symmetric matrix $S$ in \eqref{tag:22}.  
Define matrices $\varsigma\in\GL_2(F)$ and $S'\in\Sym_2(F)$ by 
\begin{align*}
\varsigma&=\begin{pmatrix} 1 & -\ol{\theta} \\ -1 & \theta \end{pmatrix}, & 
S'&=\begin{pmatrix} 0 & -\frac{1}{2} \\ -\frac{1}{2} & 0\end{pmatrix}
\end{align*}
Then 
\begin{align*}
\iota_\varsigma(t)&:=\varsigma t\varsigma^{-1}=\diag[t,\ol{t}], & 
\trs\varsigma S'\varsigma&=S. 
\end{align*}
If $K/F$ is not split, then we set 
\begin{align*}
\varsigma&=\ono_2, & 
S'&=S, & 
\iot_\varsigma&=\iot. 
\end{align*}
Fix a $\psi$-Whittaker functional $\calw$ on $\pi_0$. 
In order to investigate the Bessel integral $J_{S,\Lambda}^\calw$ we will explicitly construct toric and Bessel periods 
\begin{align*}
\bfT^\calw_\Lambda&\in\Hom_{\varsigma T\varsigma^{-1}}(\pi_0,\Lambda), & 
\bfB_{S',\Lambda}^{\calw,\sig}&\in\Hom_{R_{S'}}(I(\pi_0,\sig),\Lambda_{S'}). 
\end{align*} 

We define the toric period of $f\in \pi_0$ by 
\[\bfT^\calw_\Lambda(f)=\int_{F^\times\bsl K^\times}\calw(\pi_0(\iot_\varsigma(t))f)\Lambda(t)^{-1}\,\d t. \]
This integral is absolutely convergent and gives rise to a nonzero $K^\times$-invariant functional on $\pi_0$ for any unitary generic representation $\pi_0$.  
 
Let $w_s$ be the Weyl element given by 
\beq\label{E:weyl}
w_s=\left(\begin{array}{cc|cc}
   &     & 0 & 1   \\
   &     & -1 & 0  \\ \hline 
 0 & -1  &   &     \\
 1 &  0  &   &  
\end{array}\right)=\bfm\left(\begin{pmatrix} 1 & 0 \\ 0 & -1 \end{pmatrix}\right)s_2s_1s_2\in\GSp_4(F). 
\eeq
It is important to note that 
\[w_s\bft(A)=\bft(A)w_s. \]

\begin{definition}\label{def:BP}
We define the Bessel period of $\phi\in I(\pi_0,\sig)$ by 
\[\bfB^{\calw,\sig}_{S',\Lambda}(\phi)=\int_{F^\times\bsl K^\times}\int^\st_{\Sym_2(F)}\calw(\phi(w_s\bfn(z)\bft(\iot_\varsigma(t))))\psi_{S'}(-z)\Lambda(t)^{-1}\,\d z\d t. \]
\end{definition}

For $i\in\NN$ we put 
\[\Sym^i_2=\left\{\begin{pmatrix} x & y \\ y & w \end{pmatrix}\;\biggl|\;x,w\in F,\; y\in\frkp^{-i}\right\}. \]

\begin{lemma}\label{lem:conv}
Assume that $K/F$ is split. 
Write $\Lambda=(\Lambda_0^{},\Lambda_0^{-1})$. 
Let $\phi\in\pi$. 
Take $i\in\NN$ such that $\pi(k)\phi=\phi$ for elements $k\in\II$ which satisfy $k-\ono_4\in\Mat_4(\frkp^i)$. 
Suppose that $\psi$ has order $0$. 
If $\Re\sig<\frac{1}{2}$, then the double integral 
\[\int_{F^\times}\int_{\Sym_2^i}\calw(\phi(w_s\bfn(z)\bft(\diag[a,1])))\psi_{S'}(-z)\Lambda_0(a)^{-1}\,\d z\d^\times a\]
is absolutely convergent and equal to $\bfB^{\calw,\sig}_{S',\Lambda}(\phi)$. 
\end{lemma}

\begin{proof}
Conjugating $\phi$ by $\bfd(\lam)$ with $\lam\in 1+\frkp^i$ and making a change of variables, we see that $\bfB^{\calw,\sig}_{S',\Lambda}(\phi)$ is equal to 
\[\int_{F^\times}\int^\st_{F^3}\calw\left(\phi\left(w_s\bfn\left(\begin{pmatrix} x & y \\ y & w \end{pmatrix}\right)\bft(\diag[a,1])\right)\right)\psi(\lam y)\Lambda_0(a)^{-1}\,\d x\d y\d w\d a. \] 
Integrating both sides of this equality over $\lam\in 1+\frkp^i$, we get 
\[\int_{F^\times}\int_{\frkp^{-i}}\int^\st_{F^2}\calw\left(\phi\left(w_s\bfn\left(\begin{pmatrix} x & y \\ y & w \end{pmatrix}\right)\bft(\diag[a,1])\right)\right)\frac{\psi(y)}{\Lambda_0(a)}\,\d x\d w\d y\d^\times a. \]
The set $\Sym_2^i$ is stable under the action of elements $\diag[a,1]$. 
It suffices to check that the double integral 
\[\int_{F^2}\int_{F^\times}\biggl|\calw\left(\phi\left(\bft(\diag[a,1])w_s\bfn\left(\begin{pmatrix} x & y \\ y & w \end{pmatrix}\right)\right)\right)\biggl|\,\d^\times a\d x\d w\]
is convergent for every $y\in\frkp^{-i}$. 

Observe that 
\[\bfn(z)=\begin{pmatrix} \ono_2 & \\ z^{-1} & \ono_2 \end{pmatrix}
\begin{pmatrix} z &  \\ & z^{-1} \end{pmatrix}
\begin{pmatrix} & \ono_2 \\ -\ono_2 & \end{pmatrix}
\begin{pmatrix} \ono_2 & \\ z^{-1} & \ono_2 \end{pmatrix}. \]
Since $w_s\bft(A)w_s^{-1}=\bft(A)$, we get 
\beq
w_s\bfn(z)\in N_2Z_2\bfd(\det z)\bft(z)w_s
\begin{pmatrix} & \ono_2 \\ -\ono_2 & \end{pmatrix}
\begin{pmatrix} \ono_2 & \\ z^{-1} & \ono_2 \end{pmatrix}. \label{tag:Iwasawa}
\eeq
Let $z=\begin{pmatrix} x & y \\ y & w \end{pmatrix}$. 
The inner integral converges as $\pi_0$ is unitary and generic. 
Clearly, it depends only on $x+\frkp^i$ and $w+\frkp^i$. 
We may therefore assume that $x,w\notin\frkp^{i+1}$. 
If $x\notin\frkp^{-3i}$, then since $\ord x<-3i$, $\ord y\geq -i$ and $\ord w\leq i$, 
\begin{align*}
c_z&=-z^{-1}=-(\det z)^{-1}\begin{pmatrix} w & -y \\ -y & x \end{pmatrix}\in\begin{pmatrix} \frkp^{3i+1} & \frkp^{2i+1} \\ \frkp^{2i+1} & \frkp^{-i} \end{pmatrix}, \\ 
c_z'&=\diag[x^{-1},w^{-1}]z=\begin{pmatrix} 1 & \frac{y}{x} \\ \frac{y}{w} & 1 \end{pmatrix}\in\begin{pmatrix} 1 & \frkp^{2i+1} \\ \frkp^{-2i} & 1 \end{pmatrix} 
\end{align*} 
and hence 
\[|\calw(\phi(\bft(\diag[a,1])w_s\bfn(z)))|=|xw|_F^{\Re\sig-3/2}|\calw(\phi(\bft(\diag[ax,w]c'_z)w_s\bfn(c_z)J_2))| \]
where there is a compact set $K_i$ of $\GL_2(F)$ such that $c'_z\in K_i$. 
Put 
\[C_i=\sup_{c\in\Sym_2(\frkp^{-i}),\;c'\in K_i}\int_{F^\times}|\calw(\phi(\bft(\diag[a,1]c')w_s\bfn(c)J_2))|\,\d^\times a. \] 
We therefore conclude that  
\begin{align*}
&\int_{F\setminus\frkp^{-3i}}\int_F\int_{F^\times}\biggl|\calw\left(\phi\left(\bft(\diag[a,1])w_s\bfn\left(\begin{pmatrix} x & y \\ y & w \end{pmatrix}\right)\right)\right)\biggl|\,\d^\times a\d w\d x\\
\leq & C_i\sum_{n=3i}^\infty\sum_{m=-i}^\infty q^{(n+m)(2\Re\sig-1)/2}. 
\end{align*}
The last summation clearly converges. 
\end{proof}

It is easy to see that $\bfB^{\calw,\sig}_{S',\Lambda}$ is $R_{S'}$-invariant and 
\beq\label{E:dfnB}
\bfB^{\calw,\sig}_{S',\Lambda}(\phi)=\lim_{i\to\infty}\int_{\Sym^i_2}\bfT^\calw_\Lambda(\phi(w_s\bfn(z)))\psi_{S'}(-z)\,\d z. 
\eeq


\subsection{Factorization of $J_{S,\Lambda}^\calw$} 

\begin{proposition}\label{P:factorB.3} 
If $K/F$ is split, then  
\[J_{S,\Lambda}^\calw(\phi,\phi')=\Lambda_0(-1)\bfB^{\calw,\sig}_{S',\Lambda}(\pi(\bfm(\varsigma))\phi)\bfB^{\calw,\sig^{-1}}_{-S',\Lambda^{-1}}(\pi^\vee(\bfm(\varsigma)\bft(\bfJ))\phi') \]
for any $\phi\in \pi=I(\pi_0,\sig)$ and $\phi'\in \pi^\vee=I(\pi_0,\sig^{-1})$. 
\end{proposition}

\begin{proof}
For our choice of the measure $\d g$ on $\calp_2\bsl\GSp_4(F)$ we have  
\begin{align*}
&\int^{\rm st}_{\Sym_2(F)}\calb^\sharp(\pi(\bfn(z))\phi_1,\pi^\vee(\bft(\bfJ))\phi_2)\psi_S(-z)\,\d z\\
=&\int^{\rm st}_{\Sym_2(F)}\int_{\Sym_2(F)}\mathcal{B}(\phi_1(w_s\bfn(z'+z)),\phi_2(w_s\bfn(z')\bft(\bfJ)))\psi_S(-z)\,\d z'\d z  
\end{align*} 
(see (\ref{E:measure}) and (\ref{E:2})). 
Set $\phi_1'=\pi(\bfm(\varsigma))\phi_1^{}$ and $\phi_2'=\pi^\vee(\bfm(\varsigma)\bft(\bfJ))\phi_2^{}$. 
Take sufficiently large $i$. 
Then the right hand side is equal to 
\[\int_{\Sym^i_2}\int_{\Sym^i_2}\mathcal{B}(\phi_1'(w_s\bfn(z_1)),\phi_2'(w_s\bfn(-z_2)))\psi_{S'}(-z_1-z_2)\,\d z_1\d z_2 \]
(see the proof of Lemma \ref{lem:conv}). 
The triple integral 
\begin{align*}
&\int_{F^\times\bsl K^\times}\int_{\Sym^i_2}\int_{\Sym^i_2}|b_\calw(\pi_0(\iot_\varsigma(t))\phi_1'(w_s\bfn(z_1)),\pi_0(\bfJ)\phi_2'(w_s\bfn(z_2)))|\,\d z_1\d z_2\d t\\
=&\prod_{j=1}^2\int_{\Sym^i_2}\int_{F^\times}|\calw(\pi_0(\diag[a,1])\phi_j''(w_s\bfn(z_j)))|\,\d a\d z_j
\end{align*}
is convergent in view of Lemma \ref{lem:conv}, which justifies our formal manipulation. 
Here we put $\phi_1''=\phi_1'$ and $\phi''_2=\pi_0(\bfJ)\phi_2'$. 

Now we have the identity
\begin{align*}
&\int^{\rm st}_{\Sym_2(F)}b_\calw^\sharp(\pi(\bft(t)\bfn(z))\phi_1,\pi^\vee(\bft(\bfJ))\phi_2)\psi_S(-z)\,\d z\\
=&\int_{\Sym^i_2}\int_{\Sym^i_2}b_\calw(\pi_0(\iot_\varsigma(t))\phi'_1(w_s\bfn(z_1)),\phi'_2(w_s\bfn(z_2)))\psi_{S'}(-z_1+z_2)\,\d z_1\d z_2.  
\end{align*} 
Integrating over $t\in F^\times\bsl K^\times$ and changing the order of integration, we get   
\begin{multline*}
J_{S,\Lambda}^\calw(\phi_1,\phi_2)=\int_{\Sym_2^i}\int_{\Sym^i_2}\d z_1\d z_2\,\psi_{S'}(z_2-z_1)\\
\times\int_{F^\times\bsl K^\times}b_\calw(\pi_0(\iot_\varsigma(t))\phi'_1(w_s\bfn(z_1)),\phi'_2(w_s\bfn(z_2)))\,\d t.
\end{multline*}  
Put $f'_1=\phi'_1(w_s\bfn(z_1))$ and $f'_2=\phi'_2(w_s\bfn(z_2))$. 
Then the inner integral equals  
\[\int_{F^\x}\int_{F^\x}\cW_{f'_1}(\bft(ba))\cW_{f'_2}(\bft(-b))\Lambda_0(a)^{-1}\,\rmd^\x b\rmd^\x a=\Lambda_0(-1)\bfT^\calw_{\Lambda^{}}(f_1') \bfT^\calw_{\Lambda^{-1}}(f_2'). \]
We conclude that $\Lambda_0(-1)J_{S,\Lambda}^\calw(\phi_1,\phi_2)$ is equal to 
\[\int_{\Sym_2^i}\int_{\Sym_2^i}\bfT^\calw_{\Lambda^{}}(\phi'_1(w_s\bfn(z_1)))\bfT^\calw_{\Lambda^{-1}}(\phi'_2(w_s\bfn(z_2)))\psi_{S'}(-z_1+z_2)\,\d z_1\d z_2, \]
which completes our proof by (\ref{E:dfnB}). 
\end{proof}

\begin{remark}\label{rem:fq}
Let $K/F$ be split. 
Then $T^\calw_\Lambda(f)=\calz\bigl(\calw_f\otimes\Lambda_0^{-1},\frac{1}{2}\bigl)$ is the zeta integral for $\pi_0\otimes\Lambda_0^{-1}$.  
Since $\vsi\bfJ\vsi^{-1}=\begin{pmatrix} 0 & 1 \\ 1 & 0 \end{pmatrix}$, for every $f\in\pi_0$
\[\bfT^\calw_{\Lambda^{-1}}(\pi_0(\vsi\bfJ\vsi^{-1})f)=\Lambda_0(-1)\gam\left(\frac{1}{2},\pi_0\otimes\Lambda_0^{-1},\psi\right)\bfT^\calw_\Lambda(\pi_0(\diag[-1,1])f) \]
by the functional equation in Hecke's theory and hence  
\begin{align*}
\bfB^{\calw,\sig}_{S',\Lambda^{-1}}(\pi(\bfm(\vsi\bfJ\vsi^{-1}))\phi)&=\frac{\bfB^{\calw,\sig}_{S',\Lambda}(\pi(\bfm(\diag[-1,1]))\phi)}{\Lambda_0(-1)\gam\left(\frac{1}{2},\pi_0\otimes\Lambda_0,\psi\right)}, & 
\phi\in \pi&=I(\pi_0,\sig). 
\end{align*}
\end{remark}


\section{Explicit calculations of Bessel integrals \Roman{one}: new vectors}\label{sec:expI}
 
Let $\pi$ be an irreducible admissible representation of the form $I(\pi_0,\sig)$, where $\pi_0$ is an irreducible unramified unitary principal series representation of $\PGL_2(F)$ or the Steinberg representation twisted by an unramifiend quadratic character of $F^\times$ and $\sig\in \Ome(F^\times)^\circ$. 
Let $\phi_\sig\in\pi$ be a new vector, i.e., $\phi_\sig$ is a spherical vector in the former case and $\phi_\sig$ is a paramodular vector in the latter case. 
Our task in Sections \ref{sec:expI}, \ref{sec:expII} and \ref{sec:expIII} is to compute 
\[\BB_S^\Lambda(H):=\frac{B^S_\Lam(H\phi_\sig,H\phi_\sig)}{b(\phi_\sig,\phi_\sig)}=\frac{B_{S,\Lambda}^\calw(H\phi_\sig,H\phi_\sig)}{b_\calw^\sharp(\phi_\sig,M^*(\pi_0,\sig)\phi_\sig)}\]
for some Hecke operator $H$ on $\GSp_4(F)$, where $B_S^\Lam$ and $B_{S,\Lambda}^\calw$ are the Bessel integrals defined in Definitions \ref{def:22} and \ref{D:Bintegral}. 
Since there exists a constant $m(\pi_0,\sig)$ such that 
\[M^*(\pi_0,\sig)\phi_\sig=m(\pi_0,\sig)\phi_{\sig^{-1}}, \]
it follows from Proposition \ref{P:factorB.3} that 
\begin{align}
\BB_S^\Lambda(\pi(\bfm(\varsigma^{-1})))
=\Lambda_0(-1)\frac{\bfB^{\calw,\sig}_{S',\Lambda}(\phi_\sig)\bfB^{\calw,\sig^{-1}}_{-S',\Lambda^{-1}}(\pi^\vee(\bft(\varsigma\bfJ\varsigma^{-1}))\phi_{\sig^{-1}})}{b_\calw^\sharp(\phi_\sig,\phi_{\sig^{-1}})}.  \label{tag:newvec}
\end{align}

\begin{remark}
Proposition \ref{C:fq} below gives
\[m(\pi_0,\sig)=\gam\left(\frac{1}{2},\sig_K^{-1}\Lambda,\psi_K\right)\frac{\bfB_{S',\Lambda}^{\calw,\sig}(\phi_\sig)}{\bfB_{S',\Lambda}^{\calw,\sig^{-1}}(\phi_{\sig^{-1}})}. \]
\end{remark}


\subsection{The unramified case}

The ratio $\frac{J_{S,\Lambda}^\calw(\phi^{}_\sig,\phi_{\sig^{-1}})}{b_\calw^\sharp(\phi^{}_{\sig^{}},\phi_{\sig^{-1}})}$  has been computed by Liu in Theorem 2.2 of \cite{L1} for split or unramified extensions $K$ of $F$, and extended to ramified extensions in \cite{DPSS}. 
Since $\BB^\Lambda_S(\mathrm{Id})=\frac{J_{S,\Lambda}^\calw(\phi^{}_\sig,\phi_{\sig^{-1}})}{b_\calw^\sharp(\phi^{}_{\sig^{}},\phi_{\sig^{-1}})}$, we obtain the following result: 

\begin{theorem}[\cite{L1,DPSS}]\label{thm:21}
Assume the following conditions 
\begin{itemize}
\item both $\pi$ and $\Lambda$ are unramified, unitary and generic; 
\item $a_0,b_0\in\frko_F$; $c_0\in\frko_F^\times$; $-4\det S$ generates $\frkd_F$; $\psi$ has order $0$; 
\item When the residual characteristic of $F$ is $2$, we suppose that $F=\QQ_2$;   
\end{itemize}
If $\phi^0\in V$ is $\GSp_4(\frko_F)$-invariant, then
\[\BB^\Lambda_S(\mathrm{Id})=\frac{|\frkd_K|_K^{1/2}\zet(2)\zet(4)L\bigl(\frac{1}{2},\Spn(\pi)_K\otimes\Lambda\bigl)}{|\frkd_F|^{1/2}L(1,\tau_{K/F})L(1,\pi,\mathrm{ad})}. \]
\end{theorem}


\subsection{The paramodular case}\label{ssec:param}

The paramodular group $\rmK(\frkp)$ is the subgroup of $k\in\GSp_4(F)$ such that $\lam(k)$ is in $\frko_F^\times$ and 
\beq
k\in\begin{pmatrix} 
\frko_F & \frko_F & \frkp^{-1} & \frko_F \\
\frkp & \frko_F & \frko_F & \frko_F \\
\frkp & \frkp & \frko_F & \frkp \\ 
\frkp & \frko_F & \frko_F & \frko_F \end{pmatrix}. \label{tag:paramodular1}
\eeq
Proposition 5.1.2 of \cite{RS} gives the Iwasawa decomposition relative to $\rmK(\frkp)$ 
\[\GSp_4(F)=\calp_2(F)\rmK(\frkp). \]

Define the Iwahori subgroup of $\GL_2(\frko_F)$ by 
\[\frkI=\left\{A\in\GL_2(\frko_F)\;\biggl|\;A\equiv\begin{pmatrix} * & * \\ 0 & * \end{pmatrix}\pmod\frkp\right\}. \]
Given $\mu\in\Ome(F^\times)$, we denote by $I_1(\mu)=\mu\times\mu^{-1}$ the principal series representation of $\PGL_2(F)$. 
We write $\St\subset I_1(\ome_F^{1/2})$ for the Steinberg representation.
As is well-known, its subspace of $\frkI$-invariant vectors is one-dimensional. 
Take $\sig\in\Ome^1(F^\times)^\circ$ and an unramified quadratic character $\vep$ of $F^\times$. 
Let 
\begin{align*}
\pi_0&\simeq\St\otimes\vep, & 
\pi&\simeq I(\pi_0,\sig). 
\end{align*} 
Put $\mu=\vep\ome_F^{1/2}$. 
Then 
\begin{align*}
\pi_0&\subset I_1(\mu), & 
\pi&\subset I(I_1(\mu),\sig). 
\end{align*} 

\begin{remark}\label{rem:paramodular}
The representations $I(\St\otimes\vep,\sig)$ are called type \Roman{two}a in \cite{RS}. 
Their minimal paramodular level is $\frkp$ by Table A.12 of \cite{RS}, namely, the $\rmK(\frkp)$-invariant subspace $\pi^{\rmK(\frkp)}$ is one-dimensional. 
Representations of type \Roman{two}a (with unramified $\vep$ and $\sig$) are only the tempered representations of paramodular level $\frkp$ by Tables A.12, A.13 in loc. cit. 
\end{remark}

Let $f_{\pi_0}\in\pi_0^\frkI$ be such that 
\[\calw_{f_{\pi_0}}(\bft(a))=\vep(a)|a|_F\1_{\frko_F}(a). \]
The Iwasawa decomposition allows us to define $\phi_\sig\in\pi^{\rmK(\frkp)}$ by  
\[\phi_\sig(\bfn(z)\bfd(\lam)\bft(A)k)=\sig(\lam)|\lam|_F^{-3/2}f_{\pi_0}(A)\]
for $z\in\Sym_2(F)$, $\lam\in F^\times$, $A\in\GL_2(F)$ and $k\in \rmK(\frkp)$. 

Since $\pi$ has no Bessel model relative to the trivial character of $K^\times$ if $K/F$ is the unramified quadratic extension, we assume that $K=Fe_1\oplus Fe_2$ is split and let $S'=-\begin{pmatrix} 0 & \frac{1}{2} \\ \frac{1}{2} & 0 \end{pmatrix}$ throughout this subsection.
Let $\Lambda=(\Lambda_0^{},\Lambda_0^{-1})$ be an unramified character of $K^\times=F^\times\times F^\times$. 
Put
\begin{align*}
\gam&=\sig(\vpi), &
\eps&=-\vep(\vpi), & 
\del&=\Lambda_0(\vpi).  
\end{align*}

Define a function $\bfT':\PGL_2(F)\to\CC$ by 
\[\bfT'(A)=\bfT_\Lambda(\pi_0(\bft(A))f_{\pi_0}). \]
Observe that for $A\in\GL_2(F)$, $a,b\in F^\times$ and $k\in\frkI$ 
\begin{align}
\bfT'(A\eta_0)&=\eps\bfT'(A), &
\bfT'(\diag[a,b]Ak)&=\Lambda_0(ab^{-1})\bfT'(A). \label{tag:11}
\end{align}
In particular, the value $\bfT'(\bfn(x))$ depends only on $\frko_F^\times x+\frko_F$. 
We will write 
\[T'(m)=\bfT'(\bfn(\vpi^{-m}))\] 
for non-negative integers $m$. 
For $x,y,w\in F$ we put
\[\bfT(x,y,w)=\bfT_\Lambda\left(\phi_\sig\left(w_s\bfn\left(\begin{pmatrix} x & y \\ y & w \end{pmatrix}\right)\right)\right). \]

Recall the Bessel period 
\[\bfB^\sig_{S',\Lambda}(\phi)=\int_{\Sym_2(F)}\bfT_\Lambda(\phi(w_s\bfn(z)))\psi_{S'}(-z)\,\d z. \] 
Conjugating $\phi_\sig$ by $\bfd(\lam)$ with $\lam\in\frko_F^\times$, we get 
\[\bfB^\sig_{S',\Lambda}(\phi_\sig)=\int_{F^3}\bfT(x,y,w)\psi(-\lam y)\,\d x\d y\d z \]
by a change of variables.  
Choosing $\lam=1+u$ with $u\in\frkp$ and integrating both sides of this equality over $u\in\frkp$, we see that
\[\bfB^\sig_{S',\Lambda}(\phi_\sig)
=\int_{\frkp^{-1}}\int_{F^2}\bfT(x,y,w)\psi(-y)\,\d x\d w\d y. \] 

We define the function $\bfv:F\to\NN\cup\{0\}$ via
\[q^{\bfv(x)}=[x\frko_F+\frko_F:\frko_F]. \]

\begin{lemma}\label{lem:11}
Let $z=\begin{pmatrix} x & y \\ y & w \end{pmatrix}\in\Sym_2(F)$. 
Assume that the following conditions are satisfied: 
\begin{itemize}
\item $\bfv(y)\leq 1$. 
\item When $\bfv(y)=0$, we suppose that $x\notin\frko_F$ and $w\notin\frkp$.   
\item When $\bfv(y)=1$ and $\bfv(w)=0$, we suppose that if $\bfv(x)\geq 3$, then $w\in\frko_F^\times$, while if $\bfv(x)\leq 2$, then $w=0$. 
\item When $\bfv(y)=1$ and $\bfv(x)\leq 1$, we suppose that if $\bfv(w)\geq 2$, then $\bfv(x)=1$, while if $\bfv(w)\leq 1$, then $x=0$. 
\end{itemize}
Then $z\in\GL_2(F)$ and  
\[w_s\bfn(z)\in N_2Z_2\bfd(\det z)\bft(zJ_1)\rmK(\frkp). \]
\end{lemma}

\begin{remark}\label{rem:11}
Let $a,b,c\in F$. 
Since 
\beq
\bfT(a+\frkp^{-1},b+\frko_F,c+\frko_F)=\bfT(a,b,c), \label{tag:12}
\eeq
if $\bfv(b)\leq 1$, then we can find a triplet $x,y,w\in F$ which satisfies the conditions above and such that $\bfT(a,b,c)=\bfT(x,y,w)$.   
\end{remark}

\begin{proof}
Note that 
\[z^{-1}=(\det z)^{-1}\begin{pmatrix} w & -y \\ -y & x \end{pmatrix}\in\begin{pmatrix} \frkp & \frkp \\ \frkp & \frko_F\end{pmatrix} \]
by assumption. 
Now the lemma follows from (\ref{tag:Iwasawa}).  
\end{proof}

\begin{lemma}\label{lem:21}
If $\bfv(y)\leq 1$, then the value $\bfT(x,y,w)$ depends only on $\bfv(x)$, $\bfv(y)$ and $\bfv(w)$. 
We may therefore write  
\[\bfT(x,y,w)=\bfT_{\bfv(y)}(\bfv(x),\bfv(w)). \]
\begin{enumerate}
\renewcommand\labelenumi{(\theenumi)}
\item\label{lem:211} If $i\geq 1$, then 
\[\bfT_0(i,j)=\eps\gam^{-i-j}q^{-3(i+j)/2}\del^{j-i+1}T'(0). \]
\item\label{lem:212} If $i\geq 2$ and $j\geq 1$, then 
\[\bfT_1(i,j)=\eps\gam^{-i-j}q^{-3(i+j)/2}\del^{j-i+1}T'(0). \]
\item\label{lem:213} If $i\geq 2$, then 
\begin{align*}
\bfT_1(i,0)&=\gam^{-i}q^{-3i/2}\del^{2-i}T'(1), & 
\bfT_1(1,0)&=\gam^{-2}q^{-3}T'(0). 
\end{align*}
\item\label{lem:214} If $j\geq 1$, then 
\[\bfT_1(1,j)=\eps\gam^{-j-1}q^{-3(j+1)/2}\del^jT'(1). \]
\end{enumerate}
\end{lemma}

\begin{proof}
In view of Remark \ref{rem:11} we may assume that $x,y,w$ satisfy the assumptions of Lemma \ref{lem:11}. 
Then
\begin{align*}
\bfT(x,y,w)&=\sig(\det z)|\det z|^{-3/2}\bfT'(zJ_1)\\
&=\eps\sig(\det z)|\det z|^{-3/2}\bfT'(z\diag[\vpi,1])    
\end{align*}
by (\ref{tag:11}). 
If $y=0$ or if $\bfv(x)\geq 2$ and $\bfv(w)\geq 1$, then since 
\[z\diag[\vpi,1]=\begin{pmatrix} \vpi x & y \\ \vpi y & w \end{pmatrix}=\diag[\vpi x,w]\begin{pmatrix} 1 & \frac{y}{\vpi x} \\ \frac{\vpi y}{w} & 1 \end{pmatrix}\in \diag[\vpi x,w]\frkI, \]
we get 
\[\bfT(x,y,w)=\eps\sig(\det z)|\det z|^{-3/2}\Lambda_0(\vpi xw^{-1})\bfT'(\ono_2). \]
If $\bfv(x)\leq 1$ or $\bfv(w)=0$, then we can use (\ref{tag:12}) to verify that $\bfT(x,y,w)$ depends only on $\bfv(x)$, $\bfv(y)$ and $\bfv(w)$ by conjugating $\phi_\sig$ by $\bft(\diag[u,1])$, $\bft(\diag[1,v])$ and $\bfd(\lam)$ with $u,v,\lam\in\frko_F^\times$. 
We have proved (\ref{lem:211}) and (\ref{lem:212}). 
  
Next we shall prove (\ref{lem:213}). 
Let $j=0$. 
If $i\geq 3$, then since
\[\begin{pmatrix} \vpi^{1-i} & \vpi^{-1} \\ 1 & 1 \end{pmatrix}=\begin{pmatrix} \vpi^{-1}-\vpi^{1-i} & \vpi^{1-i} \\ 0 & 1 \end{pmatrix}\begin{pmatrix} 0 & 1 \\ 1 & 0 \end{pmatrix}\begin{pmatrix} 1 & 1 \\ 0 & 1 \end{pmatrix}, \]
we get 
\begin{align*}
\bfT(\vpi^{-i},\vpi^{-1}, 1)
&=\frac{\eps\del^{1-i}}{\gam^iq^{3i/2}}\bfT'\left(\begin{pmatrix} 1 & 1 \\ 0 & 1 \end{pmatrix}\begin{pmatrix} 0 & 1 \\ 1 & 0 \end{pmatrix}\right)=\frac{\del^{1-i}}{\gam^iq^{3i/2}}\bfT'\left(\begin{pmatrix} \vpi & 1 \\ 0 & 1 \end{pmatrix}\right).  
\end{align*}
We can easily prove the case $i=1$ as $\bfT(\vpi^{-1},\vpi^{-1},0)=\bfT(0,\vpi^{-1},0)$. 
Since
\[\begin{pmatrix} \vpi^{-2} & \vpi^{-1} \\ \vpi^{-1} & 0 \end{pmatrix}\begin{pmatrix} 0 & 1 \\ 1 & 0 \end{pmatrix}=\begin{pmatrix} \vpi^{-1} & \vpi^{-2} \\ 0 & \vpi^{-1} \end{pmatrix}, \]
we get $\bfT(\vpi^{-2},\vpi^{-1},0)=\gam^{-2}q^{-3}\bfT'(\bfn(\vpi^{-1}))$. 

If $j\geq 2$, then we can prove (\ref{lem:214}) by observing that 
\[\begin{pmatrix} 1 & \vpi^{-1} \\ 1 & \vpi^{-j} \end{pmatrix}
=\begin{pmatrix} 1 & \vpi^{-1} \\ 0 & \vpi^{-j} \end{pmatrix}\begin{pmatrix} 1-\vpi^{j-1} & 0 \\ \vpi^j & 1 \end{pmatrix}. \] 
From the computation
\[\begin{pmatrix} 0 & \vpi^{-1} \\ 1 & \vpi^{-1} \end{pmatrix}=\begin{pmatrix} 1 & \vpi^{-1} \\ 0 & \vpi^{-1} \end{pmatrix}\begin{pmatrix} -1 & 0 \\ \vpi & 1 \end{pmatrix} \]
we can deduce the remaining case $i=j=1$. 
\end{proof}

Now we are led to  
\begin{align*}
\bfB^\sig_{S',\Lambda}(\phi_\sig)=&
\int_{\frkp^{-1}}\int_{F^2}\bfT_{\bfv(y)}(\bfv(x),\bfv(w))\psi(y)\,\d x\d w\d y \\
=&\sum_{i=1}^\infty\sum_{j=0}^\infty q^{i+j}(1-q^{-1})^{\min\{1,i-1\}+\min\{1,j\}}(\bfT_0(i,j)-\bfT_1(i,j)).  
\end{align*}
If $i\geq 2$ and $j\geq 1$, then $\bfT_0(i,j)=\bfT_1(i,j)$ by Lemma \ref{lem:21}(\ref{lem:211}), (\ref{lem:212}). 
Hence 
\begin{align*}
\bfB^\sig_{S',\Lambda}(\phi_\sig)
=&q(\bfT_0(1,0)-\bfT_1(1,0)) \\
&+(1-q^{-1})\sum_{i=2}^\infty q^i(\bfT_0(i,0)-\bfT_1(i,0)) \\
&+(1-q^{-1})\sum_{j=1}^\infty q^{j+1}(\bfT_0(1,j)-\bfT^1(1,j)) \\  
=&q(\bfT_0(1,0)-\bfT_1(1,0))+(1-q^{-1})(I_0+J_0)+I_1+J_1, 
\end{align*}
where 
\begin{align*}
I_0&=\sum_{i=2}^\infty q^i\bfT_0(i,0), &
I_1&=-(1-q^{-1})\sum_{i=2}^\infty q^i\bfT_1(i,0), \\ 
J_0&=\sum_{j=1}^\infty q^{j+1}\bfT_0(1,j), & 
J_1&=-(1-q^{-1})\sum_{j=1}^\infty q^{j+1}\bfT_1(1,j).  
\end{align*}
Lemma \ref{lem:21}(\ref{lem:211}) gives 
\begin{align*}
I_0&=\sum_{i=2}^\infty q^i\eps\gam^{-i}q^{-3i/2}\del^{-i+1}T'(0)=\frac{\eps\gam^{-2}q^{-1}\del^{-1}}{1-\gam^{-1}\del^{-1}q^{-1/2}}T'(0), \\
J_0&=\sum_{j=1}^\infty q^{j+1}\eps\gam^{-1-j}q^{-3(1+j)/2}\del^jT'(0)=\frac{\eps\gam^{-2}q^{-1}\del}{1-\gam^{-1}\del q^{-1/2}}T'(0). 
\end{align*}
Lemma \ref{lem:21}(\ref{lem:213}) gives 
\[I_1=-(1-q^{-1})\sum_{i=2}^\infty q^i\gam^{-i}q^{-3i/2}\del^{2-i}T'(1)
=-\frac{\gam^{-2}q^{-1}(1-q^{-1})}{1-\gam^{-1}\del^{-1}q^{-1/2}}T'(1). \]
Lemma \ref{lem:21}(\ref{lem:214}) gives 
\[J_1=-(1-q^{-1})\sum_{j=1}^\infty q^{j+1}\eps\frac{\del^jT'(1)}{\gam^{j+1}q^{3(j+1)/2}}=-\frac{\eps\gam^{-2}q^{-1}\del(1-q^{-1})}{1-\gam^{-1}\del q^{-1/2}}T'(1). \]

\begin{proposition}\label{prop:11}
If $m\geq 1$, then 
\begin{align*}
T'(m)&=(1+\eps\del)\frac{(-\eps\del^{-1})^m}{q^m(1-q^{-1})}T'(0), & 
T'(0)&=L(1,\vep\Lambda_0^{-1}). 
\end{align*}
\end{proposition}

\begin{proof}
Since $\pi_0$ has no $\GL_2(\frko_F)$-invariant vector, 
\[\sum_{x\in\frko_F/\frkp}\bfT'\left(A\begin{pmatrix} 1 & 0 \\ x & 1\end{pmatrix}\right)=-\bfT'(AJ_1)=-\eps\bfT'(A\diag[\vpi,1]) \]
by (\ref{tag:11}). 
Observe that 
\[\bfn(y)\begin{pmatrix} 1 & 0 \\ x & 1\end{pmatrix}
=\begin{pmatrix} 1+yx & y \\ x & 1 \end{pmatrix}
=\begin{pmatrix} -x^{-1} & 1+yx \\ 0 & x \end{pmatrix}\begin{pmatrix} 0 & 1 \\ 1 & 0 \end{pmatrix}\begin{pmatrix} 1 & x^{-1} \\ 0 & 1 \end{pmatrix} \] 
for $x\in\frko_F^\times$. 
If $m\geq 2$, then we get 
\[\bfT'(\bfn(\vpi^{1-m}))+\eps(q-1)\bfT'(\bfn(\vpi^{1-m})\diag[\vpi,1])=-\eps\bfT'(\bfn(\vpi^{1-m})\diag[\vpi,1]), \]
letting $A=\bfn(y)$ and $y=\vpi^{1-m}$. 
It follows that 
\[T'(m)=-q^{-1}\eps\del^{-1} T'(m-1)=\cdots=(-q^{-1}\eps\del^{-1})^{m-1}T'(1). \]
Letting $y=0$, we get 
\[\bfT'\left(\begin{pmatrix} 1 & 0 \\ x & 1 \end{pmatrix}\right)=\eps\bfT'(\bfn(-x)\diag[\vpi,1])=\eps\del\bfT'(\bfn(\vpi^{-1}))=\eps\del T'(1)\]
for every $x\in\frko_F^\times$. 
We obtain 
\[T'(1)=-(1+\eps\del^{-1})(q-1)^{-1}T'(0). \]

Since $\calw_{f_{\pi_0}}(\diag[a,1])=\vep(a)|a|_F\1_{\frko_F}(a)$, one can easily compute $T'(0)$.      
\end{proof}

\begin{proposition}\label{prop:12}
\[\bfB^\sig_{S',\Lambda}(\phi_\sig)
=\eps\gam^{-1}q^{-1/2}\frac{L\left(\frac{1}{2},\sig^{-1}_K\Lambda\right)}{L\left(\frac{3}{2},\vep\sig\right)L(1,\sig^{-2})}L(1,\vep\Lambda_0^{-1}). \]
\end{proposition}

\begin{proof}
Proposition \ref{prop:11} gives
\begin{align*}
I_1&=-\frac{\gam^{-2}q^{-1}(1-q^{-1})}{1-\gam^{-1}\del^{-1}q^{-1/2}}(1+\eps\del)\frac{-\eps\del^{-1}}{q(1-q^{-1})}T'(0)
=\frac{\gam^{-2}q^{-2}(1+\eps\del^{-1})}{1-\gam^{-1}\del^{-1}q^{-1/2}}T'(0), \\ 
J_1&=-\frac{\eps\gam^{-2}q^{-1}\del(1-q^{-1})}{1-\gam^{-1}\del q^{-1/2}}(1+\eps\del)\frac{-\eps\del^{-1}}{q(1-q^{-1})}T'(0)
=\frac{\gam^{-2}q^{-2}(1+\eps\del)}{1-\gam^{-1}\del q^{-1/2}}T'(0). 
\end{align*}

Now we have
\begin{align*}
I:=\frac{(1-q^{-1})I_0+I_1}{T'(0)}
&=(1-q^{-1})\frac{\eps\gam^{-2}q^{-1}\del^{-1}}{1-\gam^{-1}\del^{-1}q^{-1/2}}
+\frac{\gam^{-2}q^{-2}(1+\eps\del^{-1})}{1-\gam^{-1}\del^{-1}q^{-1/2}} \\ 
&=\frac{\eps\gam^{-2}q^{-1}\del^{-1}+\gam^{-2}q^{-2}}{1-\gam^{-1}\del^{-1}q^{-1/2}}, \\
J:=\frac{(1-q^{-1})J_0+J_1}{T'(0)}
&=\frac{\eps\gam^{-2}q^{-1}\del+\gam^{-2}q^{-2}}{1-\gam^{-1}\del q^{-1/2}}. 
\end{align*}
By Lemma \ref{lem:21}(\ref{lem:211}), (\ref{lem:213}) we have 
\begin{align*}
q\frac{\bfT_0(1,0)-\bfT_1(1,0)}{T'(0)}+J
&=\eps\gam^{-1}q^{-1/2}-\gam^{-2}q^{-2}+\frac{\eps\gam^{-2}q^{-1}\del+\gam^{-2}q^{-2}}{1-\gam^{-1}\del q^{-1/2}}\\
&=\frac{\eps\gam^{-1}q^{-1/2}+\del\gam^{-3} q^{-5/2}}{1-\gam^{-1}\del q^{-1/2}}. 
\end{align*}
We conclude that 
\begin{align*}
\frac{\bfB^\sig_{S',\Lambda}(\phi_\sig)}{T'(0)}
&=\frac{\eps\gam^{-2}q^{-1}\del^{-1}+\gam^{-2}q^{-2}}{1-\gam^{-1}\del^{-1}q^{-1/2}}
+\frac{\eps\gam^{-1}q^{-1/2}+\del\gam^{-3} q^{-5/2}}{1-\gam^{-1}\del q^{-1/2}}\\
&=\eps\gam^{-1}q^{-1/2}\frac{(1+\eps\gam^{-1}q^{-3/2})(1-\gam^{-2}q^{-1})}{(1-\gam^{-1}\del^{-1}q^{-1/2})(1-\gam^{-1}\del q^{-1/2})}. 
\end{align*}
The proof is complete by Proposition \ref{prop:11}. 
\end{proof}

We use the Iwasawa decomposition to define $\phi^\mathrm{pa}\in I(\mathrm{Id},1)$ by 
\[\phi^\mathrm{pa}(\bfd(\lam)\bft(A)\bfn(z)k)=|\lam|^{-3/2}, \]
where $\lam\in F^\times$, $A\in\GL_2(F)$, $z\in\Sym_2(F)$ and $k\in \rmK(\frkp)$. 
The elements $\phi_w\in I(I_1(\ome_F^{-1/2}),1)$ are defined in Section \ref{sec:7} below. 
The pairing $\bfq:I(\mathrm{Id},1)\times I(\mathrm{Id},1)\to\CC$ will be defined in (\ref{tag:Qpairing}). 

\begin{lemma} 
We have 
\begin{align*}
&\phi^\mathrm{pa}=\phi_{\bf1_4}+\phi_{s_1}+\phi_{s_2}+\phi_{s_1s_2}+q^{-3/2}(\phi_{s_2s_1}+\phi_{s_1s_2s_1}+\phi^\dagger+\phi_{s_2s_1s_2}),  \\
&\bfq(\phi^\mathrm{pa},\phi^\mathrm{pa})
=q^{-2}\zet(1)^2\zet(2)^{-2}. 
\end{align*}
\end{lemma}

\begin{proof}
Since $s_2\in \rmK(\frkp)$, it is clear that 
\begin{align*}
\phi^\mathrm{pa}(\ono_4)&=\phi^\mathrm{pa}(s_1)=\phi^\mathrm{pa}(s_2)=\phi^\mathrm{pa}(s_1s_2)=1, \\
\phi^\mathrm{pa}(s_2s_1)&=\phi^\mathrm{pa}(s_1s_2s_1)=\phi^\mathrm{pa}(s_2s_1s_2)=\phi^\mathrm{pa}(s_1s_2s_1s_2). 
\end{align*}
Since $\bfm(\diag[\vpi^{-1},1])s_1s_2s_1s_2\in \rmK(\frkp)$, we get $\phi^\mathrm{pa}(s_1s_2s_1s_2)=q^{-3/2}$. 
The Iwasawa decomposition relative to $\GSp_4(\frko_F)$ gives
\[\bfq(\phi^\mathrm{pa},\phi^\mathrm{pa})
=\int_{\GSp_4(\frko_F)}\phi^\mathrm{pa}(k)^2\,\d k 
=\sum_{w\in W}\phi^\mathrm{pa}(w)^2\mathrm{vol}(\II w\II). \]
Clearly, $\mathrm{vol}(\II w\II)=q_w\mathrm{vol}(\II)$ (see Definition \ref{def:21}). 
Since $\mathrm{vol}(\II s_1s_2s_1s_2\II)+\mathrm{vol}(\II s_2s_1s_2\II)=1$ for our choice of the measure, $\mathrm{vol}(\II)=q^{-4}(1+q^{-1})^{-1}$ and 
\begin{align*}
\bfq(\phi^\mathrm{pa},\phi^\mathrm{pa})
&=(1+q^{-1})^{-1}\sum_{w\in W}\phi^\mathrm{pa}(w)^2q^{-4}q_w \\
&=(1+q^{-1})^{-1}\{q^{-4}+2q^{-3}+q^{-2}+q^{-3}(q^{-2}+2q^{-1}+1)\},
\end{align*}
which proves the second identity. 
\end{proof}

\begin{proposition}\label{prop:paramodular}
Let $\sig,\Lambda_0\in\Ome^1(F^\times)^\circ$.
Fix a quadratic unramified character $\vep:F^\times\to\{\pm 1\}$. 
Put $\pi=I(\St\otimes\vep,\sig)$. 
Then
\[\BB_S^\Lambda(\pi(\bfm(\varsigma^{-1})))=-\vep q(1+q^{-2})\Lambda_0(\vpi)^{-1}\frac{|\frkd_K|_K^{1/2}\zet(2)\zet(4)L\left(\frac{1}{2},\Spn(\pi)_K\otimes\Lambda\right)}{|\frkd_F|^{1/2}L(1,\tau_{K/F})L(1,\pi,\mathrm{ad})}. \]
\end{proposition}

\begin{proof}
We know that 
\begin{align}
L(s,\Spn(\pi)_K\otimes\Lambda)=L(s,\sig^{-1}_K\Lambda)L(s,\sig_K\Lambda)L(s,\ome_K^{1/2}\vep_K\Lambda), \label{tag:L-factor}\\
L(s,\pi,\mathrm{ad})=\zet(s)\zet(s+1)L(s,\sig^{-2})L(s,\sig^2)L(s,\ome_F^{1/2}\vep\sig)L(s,\ome_F^{1/2}\vep\sig^{-1}) \notag
\end{align} 
(see \cite[Table A.8]{RS} and \cite[(13)]{AS2}). 
We also remark that 
\begin{align*}
\vep(1/2,\St\otimes\vep\Lambda_0^{-1})&=\Lambda_0(\vpi)^{-1}(-\vep), &
\end{align*}
(cf. \cite[Lemma 3.1, Proposition 3.8]{CM}). 
By (\ref{tag:newvec}) and Remark \ref{rem:fq} 
\[b_\calw^\sharp(\phi_\sig,\phi_{\sig^{-1}})\BB_S^\Lambda(\pi(\bfm(\varsigma^{-1})))
=\gam\left(\frac{1}{2},\St\otimes\vep\Lambda_0^{-1},\psi\right)\bfB^\sig_{S',\Lambda}(\phi_\sig)\bfB^{\sig^{-1}}_{S',\Lambda}(\phi_{\sig^{-1}}). \]
By Proposition \ref{prop:12} the right hand side is equal to 
\begin{align*}
&\frac{L(1,\vep\Lambda_0)}{L(1,\vep\Lambda_0^{-1})}(-\vep)\Lambda_0(\vpi)^{-1}\frac{L\left(\frac{1}{2},\sig^{-1}_K\Lambda\right)L\left(\frac{1}{2},\sig_K\Lambda\right)L(1,\vep\Lambda_0^{-1})^2}{qL\left(\frac{3}{2},\vep\sig\right)L(1,\sig^{-2})L\left(\frac{3}{2},\vep\sig^{-1}\right)L(1,\sig^2)}\\
=&(-\vep)\Lambda_0(\vpi)^{-1}\zet(1)\zet(2)\frac{L\bigl(\frac{1}{2},\Spn(\pi)_K\otimes\Lambda\bigl)}{qL(1,\pi,\mathrm{ad})}. 
\end{align*}

Since $\calw_{\phi_\sig}(\bft(\diag[a,1]))=\vep(a)|a|_F\1_{\frko_F}(a)$ for $a\in F^\times$, we have 
\[b_\calw(\pi(g)\phi_\sig,\pi^\vee(g)\phi_{\sig^{-1}})=\zet(2)\phi^\mathrm{pa}(g)^2 \]
and hence $b_\calw^\sharp(\phi_\sig,\phi_{\sig^{-1}})=\zet(2)\bfq(\phi^\mathrm{pa},\phi^\mathrm{pa})=q^{-2}\zet(1)^2\zet(2)^{-1}$. 
\end{proof} 


\section{Explicit calculations of Bessel integrals \Roman{two}: degenerate vectors}\label{sec:expII}


\subsection{Degenerate principal series}\label{ssec:degenerate}

Let $D$ be a quaternion algebra over a local field $F$ of characteristic zero. 
We retain the notation in Section \ref{sec:notation}. 
Fix a quadratic character $\vep$ of $F^\times$. 
For $\sig\in\Ome(F^\times)$ we consider the normalized induced representation $\pi=I(\vep\circ\Nr^D_F,\sig)$, which is realized on the space of smooth functions $\phi:\GU_2^D(F)\to\CC$ satisfying  
\[\phi(\bfd(\lam)\bft(A)\bfn(z)g)=\sig(\lam)|\lam|^{-3/2}\vep(\Nr^D_F(A))\phi(g)\]
for $A\in D^\times$, $z\in D_-$, $\lam\in F^\times$ and $g\in\GU_2^D(F)$. 

In the $\frkp$-adic case we fix a maximal compact subring $\frko_D$ of $D$ and set $\calk=\GU_2^D(F)\cap\GL_2(\frko_D)$. 
In the archimedean case we fix a maximal compact subgroup $\calk$ of $\GU_2^D(F)$. 
For $g\in \GU_2^D(F)$ the quantity $|a(g)|$ is defined by setting $|a(g)|=|\lam|_F^{-1}$, where we write $g=pk$ with $p=\bfd(\lam)\bft(A)\bfn(z)\in \calp$ and $k\in\calk$. 
For $\phi\in \pi$ and $s\in\CC$ we define $\phi^{(s)}\in I(\vep\circ\Nr^D_F,\sig\ome_F^s)$ by $\phi^{(s)}(g)=\phi(g)|a(g)|^{-s}$.
We define the intertwining operator 
\[M(\vep\circ\Nr^D_F,\sig):I(\vep\circ\Nr^D_F,\sig)\to I(\vep\circ\Nr^D_F,\sig^{-1})\] 
by
\[[M(\vep\circ\Nr^D_F,\sig)\phi](g)=\lim_{s\to 0}\int_{D_-}\phi^{(s)}\left(\begin{pmatrix} 0 & 1 \\ 1 & 0\end{pmatrix}\bfn(z)g\right)\,\d z. \]
The integral converges for $\Re s+\Re\sig<-\frac{3}{2}$ and admits meromorphic continuation to whole $s$-plane. 


\subsection{Degenerate Whittaker model} 

Fix $S\in D_-$. 
Put 
\begin{align*}
d_0&=-\Nr^D_F(2S), & 
K&=F(\sqrt{d_0}). 
\end{align*}
We identify $K$ with the subalgebra $F+FS$ of $D$. 
Put 
\[R_S=\{\bft(t)\bfn(z)\;|\;z\in D_-,\; t\in K^\times\}. \]
We define an additive character $\psi_S$ on $D_-$ by $\psi_S(z)=\psi(\Tr^D_F(S z))$ and associate a character $\Lambda^S$ of $R_S$ by $\Lambda^S(\bft(t)\bfn(z))=\Lambda(t)\psi_S(z)$ to a character $\Lambda$ of $F^\times\backslash K^\times$.  
The integral 
\[\bfB^{\sig\ome_F^s}_S(\phi^{(s)})=\int_{D_-}\phi^{(s)}\left(\begin{pmatrix} 0 & 1 \\ 1 & 0\end{pmatrix}\bfn(z)\right)\overline{\psi_S(z)}\,\d z\]
is absolutely convergent for $\Re s+\Re\sig<-\frac{3}{2}$ and makes sense for all $s$ by an entire analytic continuation. 
In the nonarchimedean case the integral stabilizes. 
The reader who are interested in the archimedean case can consult \cite{W1}. 
One can easily see that 
\beq
\bfB^\sig_S\circ\pi(\bft(t))=\vep(\Nr^K_F(t))\bfB^\sig_S \label{tag:inv}
\eeq
for $t\in K^\times$. 
Thus $\bfB^\sig_S$ is a Bessel functional on $\pi$ relative to $S$ and $\vep_K$. 

We introduce the $\GU_2^D(F)$-invariant pairing 
\[\bfq:I(\vep\circ\Nr^D_F,\sig)\times I(\vep\circ\Nr^D_F,\sig^{-1})\to\mathbb{C}\] 
by
\begin{align}
\bfq(\phi_1,\phi_2)
&=\int_{\calp\backslash\GU_2^D(F)}\phi_1(g)\phi_2(g)\,\d g \label{tag:Qpairing}\\ 
&=\int_{D_-}\phi_1\left(\begin{pmatrix} 0 & 1 \\ 1 & 0\end{pmatrix}\bfn(z)\right)\phi_2\left(\begin{pmatrix} 0 & 1 \\ 1 & 0\end{pmatrix}\bfn(z)\right)\,\d z \notag
\end{align}
for $\phi_1\in I(\vep\circ\Nr^D_F,\sig)$ and $\phi_2\in I(\vep\circ\Nr^D_F,\sig^{-1})$. 
Define the Bessel integral by 
\begin{multline*}
J_S(\phi_1,\phi_2)
=\int_{D_-}\bfq(\pi(\bfn(z))\phi_1,\pi(\bfd(-1))\phi_2)\overline{\psi_S(z)}\,\d z \\
:=\lim_{s\to 0}\sig(-1)\int_{D_-}^2\phi^{(s)}_1\left(\begin{pmatrix} 0 & 1 \\ 1 & 0\end{pmatrix}\bfn(z-z')\right)\phi_2^{(s)}\left(\begin{pmatrix} 0 & 1 \\ 1 & 0\end{pmatrix}\bfn(z')\right)\d z'\overline{\psi_S(z)}\d z.  
\end{multline*}
The double integral absolutely converges for $\Re s\ll 0$ and can be continued as an entire function to the whole complex plane. 
We have the factorization
\beq
J_S(\phi,\phi')=\bfB^\sig_S(\phi)\bfB^{\sig^{-1}}_{-S}(\pi^\vee(\bfd(-1))\phi'). \label{Q:factorization}
\eeq
It follows that 
\[J_S\in\Hom_{R_S\times R_S}(I(\vep\circ\Nr^D_F,\sig)\boxtimes I(\vep\circ\Nr^D_F,\sig^{-1}),\vep_K^S\boxtimes\vep_K^S). \]

If $I(\vep\circ\Nr^D_F,\sig)$ has a new vector $\phi_\sig$, then we have a functional equation 
\[M(\vep\circ\Nr^D_F,\sig)\phi_\sig=c(\vep,\sig)\phi_{\sig^{-1}} \]
with factor $c(\vep,\sig)$ of proportionality. 
We set 
\[\bfB_S(\phi,\phi')=J_S(\phi,M(\vep\circ\Nr^D_F,\sig)\phi'). \]
Then 
\beq
\frac{\bfB_S(\phi_\sig,\phi_\sig)}{\bfq(\phi_\sig,M(\vep\circ\Nr^D_F,\sig)\phi_\sig)}=\frac{\bfB^\sig_S(\phi_\sig)\bfB^{\sig^{-1}}_{-S}(\pi^\vee(\bfd(-1))\phi_{\sig^{-1}})}{\bfq(\phi_\sig,\phi_{\sig^{-1}})}. \label{Q:newvec}
\eeq



\subsection{The quaternion case}\label{ssec:quaternion}

Let $D$ be a quaternion division algebra over a $\frkp$-adic field $F$ of characteristic $0$. 
We denote by $\frko_D$ the maximal compact subring of $D$ and by $\frkP$ the maximal proper two-sided ideal of $\frko_D$. 
Put $\frko_D^-=\frko_D\cap D_-$. 
Define a maximal compact subgroup $\rmK(\frkP)$ of $\GU_2^D(F)$ by 
\beq
\rmK(\frkP)=\biggl\{\begin{pmatrix} \alpha & \beta \\ \gamma & \delta \end{pmatrix}\in\GU_2^D(F) \;\biggl|\; \alpha,\delta\in\frko_D,\; \beta\in\frkP^{-1},\; \gamma\in\frkP\biggl\}. \label{tag:paramodular2}
\eeq
Then $\GU_2^D(F)=\calp \rmK(\frkP)$. 
Let $\sig=\ome_F^s$ be an unramified character of $F^\times$ and $\vep=\ome_F^t$ an unramified quadratic character of $F^\times$. 
Put $\pi=I(\vep\circ\Nr^D_F,\sig)$. 
We define its $L$-factors as those of $I(\vep\circ\St,\sig)$. 

\begin{remark}\label{rem:IIaG}
The representations $I(\vep\circ\Nr^D_F,\sig)$ are called type IIa$^G$ in \cite{RW3}. 
The subgroup $\rmK(\frkP)$ is considered in \cite{Hi1,N1}. 
Appendix of \cite{N1} shows that representations of type IIa$^G$ (with unramified $\vep$ and $\sig$) are only the tempered representations having a non-zero $\rmK(\frkP)$-fixed vector. 
\end{remark}

It follows from \cite{AS2} that 
\begin{align*}
\frac{L\bigl(\frac{1}{2},\Spn(\pi)_K\otimes\vep_K\bigl)}{L(1,\pi,\mathrm{ad})}
=&\frac{L\bigl(\frac{1}{2},(\sig^{-1}\vep)_K\bigl)L\bigl(\frac{1}{2},(\sig\vep)_K\bigl)L(1,\tau_{K/F})}{\zet(2)L(1,\sig^{-2})L(1,\sig^2)L\bigl(\frac{3}{2},\vep\sig^{-1}\bigl)L\bigl(\frac{3}{2},\vep\sig\bigl)}. 
\end{align*}
For $g\in\GU_2^D(F)$ the quantity $|\alp(g)|$ is defined via  
\[|\alp(\bfd(\lam)\bfm(A)\bfn(z)k)|=|\Nr^D_F(A)|_F, \]
where $\lam\in F^\times$, $A\in D^\times$, $z\in D_-$ and $k\in \rmK(\frkP)$. 
We define $\phi_\sig\in\pi$ by 
\[\phi_\sig(g)=|\lam(g)|_F^{(2s-3)/2}|\alp(g)|^{(-2s+2t+3)/2}. \]

\begin{proposition}[Hirai \cite{Hi1}]\label{prop:quatB}
Let $S\in\frko_D^-$ and $K=F+FS$. 
Take the Haar measure $\d z$ on $D_-$ which gives $D_-\cap\frkP$ the volume $1$. 
\begin{enumerate}
\renewcommand\labelenumi{(\theenumi)}
\item\label{prop:quatB1} If $\vpi^{-1}S\notin\frko_D^-$, then 
\[\bfB_S^\sig(\phi_\sig)=\frac{q^{3/2}(\vep\sig)(\vpi)^{-1}L\bigl(\frac{1}{2},(\vep\sig^{-1})_K\bigl)}{L(1,\sig^{-2})L\bigl(\frac{3}{2},\vep\sig^{-1}\bigl)}. \]
\item\label{prop:quatB2} $\bfq(\phi_\sig,\phi_{\sig^{-1}})=\zet(2)\zet(4)^{-1}$. 
\end{enumerate}
\end{proposition}

\begin{proof}
Put $\rmK(\frkP)'=\bfd(\vpi)^{-1}\rmK(\frkP)\bfd(\vpi)$. 
Observe that 
\[\rmK(\frkP)'=\biggl\{\begin{pmatrix} \alpha & \beta \\ \gamma & \delta \end{pmatrix}\in\GU_2^D(F) \;\biggl|\; \alpha,\delta\in\frko_D,\; \beta\in\frkP,\; \gamma\in\frkP^{-1}\biggl\}. \]
Define the function $A_\frkp:\GU_2^D(F)\to\CC$ by $A_\frkp(g)=|\alp(g\bfd(\vpi)^{-1})|$. 
Set 
\[\alpha_\frkp(\eta,x)=\int_{D_-}A_\frkp\left(\begin{pmatrix} 0 & 1 \\ 1 & 0\end{pmatrix}\bfn(z)\right)^{(2x+3)/2}\overline{\psi_\eta(z)}\,\d z \]
for $\eta\in\vpi^{-1}\frko_D^-$ and $x\in\CC$. 
Put 
\begin{align*}
\phi_\sig'&=\pi(\bfd(\vpi)^{-1})\phi_\sig, &
S'&=\vpi^{-1}S. 
\end{align*}
By a change of variables we have 
\begin{align*}
\bfB_S^\sig(\phi_\sig)
&=\int_{D_-}\phi'_\sig\left(\begin{pmatrix} 0 & 1 \\ 1 & 0\end{pmatrix}\bfn(z)\bfd(\vpi)\right)\overline{\psi_S(z)}\,\d z\\
&=q^{(2s-3)/2}\cdot q^3\bfB_{S'}^\sig(\phi'_\sig)\\
&=q^{(2s+3)/2}\cdot q^{(2s-3)/2} \alpha_\frkp(S',t-s). 
\end{align*}
Theorem 2.3 of \cite{Hi1} explicitly computes $\alpha_\frkp(\eta,s)$. 
By the assumption on $S'$
\[\alpha_\frkp(S',s)=q^{(2s+3)/2}\frac{\zet_K\bigl(s+\frac{1}{2}\bigl)}{\zet(2s+1)\zet(s+\frac{3}{2}\bigl)}. \] 
We remind the reader that the measure in \cite{Hi1} gives $\frkP\cap D_-$ volume $1$.   
The second part follows from the obvious equality $\bfq(\phi_\sig,\phi_{\sig^{-1}})=q^{-3}\alp_\frkp\bigl(0,\frac{3}{2}\bigl)$. 
\end{proof}

Now we have the following conclusion by (\ref{Q:newvec}) and Proposition \ref{prop:quatB}.  

\begin{corollary}\label{cor:quatB}
\[\frac{\bfB_S(\phi_\sig,\phi_\sig)}{\bfq(\phi_\sig,M(\vep\circ\Nr^D_F,\sig)\phi_\sig)}=q^3(1-q^{-2})\frac{\zet(4)L\bigl(\frac{1}{2},\Spn(\pi)_K\otimes\vep_K\bigl)}{\zet(2)L(1,\tau_{K/F})L(1,\pi,\mathrm{ad})}. \]
\end{corollary} 


\subsection{The archimedean case}\label{ssec:archimedean}

We discuss the case in which $F=\RR$ and $K=\CC$. 
Put $\zet_F(s)=\pi^{-s/2}\vGm\bigl(\frac{s}{2}\bigl)$. 
We define the character of $\CC$ by $\bfe(z)=e^{2\pi\iu z}$ for $z\in\CC$. 
Its restriction to $F$ is denoted by $\psi^F$. 
The measure $\d x$ on $\RR$ is the Lebesgue measure and $\d^\times x=\frac{\d x}{|x|_\RR}$. 
Let $\d z$ be the standard measure on $\Sym_2(\RR)$ defined by viewing $\Sym_2(\RR)$ as $\RR^3$ in an obvious fashion. 


Denote by $\Sym_g(\RR)^+$ the set of positive definite symmetric matrices of rank $g$ over $\RR$. 
Let
\[\GSp_{2g}(\RR)^\circ=\{h\in\GSp_{2g}\;|\;\lam(h)>0\}\] 
be the identity component of the real reductive group $\GSp_{2g}(\RR)$. 
We can define the action of the connected component $\GSp_{2g}(\RR)^\circ$ on the space 
\[\frkH_g=\{Z\in\Sym_g(\CC)\;|\;\Im Z\in\Sym_g(\RR)^+\} \] 
and the automorphy factor on $\GSp_{2g}(\RR)^\circ\times\frkH_g$ by 
\begin{align*}
hZ&=(A Z+B)(C Z+D)^{-1}, & 
j(h,Z)&=(\det h)^{-1/2}\det(C Z+D)
\end{align*}
for $Z\in\frkH_g$ and $h=\begin{pmatrix} A & B \\ C & D \end{pmatrix}\in\GSp_{2g}(\RR)^\circ$ with matrices $A,B,C,D$ of size $g$. 
Put $\bfi_g=\sqrt{-1}\ono_g$. 
Define the maximal compact subgroup of $\GSp_{2g}(\RR)^\circ$ by 
\[\U_n=\{h\in\GSp_{2g}(\RR)^\circ\;|\;h(\bfi_g)=\bfi_g\}. \]

\begin{definition}\label{def:25}
For each positive integer $\kap$ we denote the lowest weight representation of $\GSp_{2g}(\RR)^\circ$ with lowest $\U_g$-type $k\mapsto j(k,\bfi_g)^{-\kap}$ by $\frkD_\kap^{(g)}$ and the highest weight representation with highest $\U_g$-type $k\mapsto j(k,\bfi_g)^\kap$ by $\frkD_{-\kap}^{(g)}$. 
The direct sum $D_\kap^{(g)}=\frkD_\kap^{(g)}\oplus\frkD_{-\kap}^{(g)}$ extends to an irreducible representation of $\GSp_{2g}(\RR)$. 
\end{definition}

\begin{remark}
If $S\in\Sym_2(\RR)$ is positive or negative definite, then Theorem 3.10 of \cite{PS} says that $D_\kap^{(2)}$ admits a Bessel model relative to $(\Lambda,S)$ if and only if $\Lambda$ is trivial. 
This fact is compatible with our observations in the previous subsection. 
\end{remark}

To simplify notation, we set 
\begin{align*}
J(m)&=I(\sgn^m\circ\det,\ome_\RR^{(3-2m)/2}), &
N(m)&=M(\sgn^m\circ\det,\ome_\RR^{(3-2m)/2})
\end{align*}
for $m\in\ZZ$. 
As is well-known, $\frkD_{\pm\kap}^{(2)}$ are subrepresentations of the degenerate principal series $J(\kap)$ which is here viewed as a representation of $\GSp_4(\RR)^\circ$. 
The representation $\frkD_\kap^{(2)}$ (resp. $\frkD_{-\kap}^{(2)}$) is generated by the function $\phi_\kap(h)=j(h,\bfi_2)^{-\kap}$ (resp. $\phi_{-\kap}(h)=\overline{j(h,\bfi_2)}^{-\kap}$). 
These functions $\phi_{\pm\kap}$ are extended uniquely to elements of $J(\kap)$. 
Since $\phi_\kap(h\bfd(-1))=\phi_{-\kap}(h)$, we can view $D_\kap^{(2)}$ as a subrepresentation of $J(\kap)$. 

Put $\vph_\kap=\phi^{(2\kap-3)}_\kap\in J(3-\kap)$ and $\bfB_S^s=\bfB_S^{\ome_\RR^s}$. 
Observe that 
\begin{align*}
\bfB_S^{(2s+3-2\kap)/2}(\phi_\kap^{(s)})
&=\int_{\Sym_2(\RR)}\phi_\kap^{(s)}(J_2\bfn(z))\bfe(-\tr(Sz))\,\d z\\
&=\int_{\Sym_2(\RR)}|\det(z+\bfi_2)|^s\det(z+\bfi_2)^{-\kap}\bfe(-\tr(Sz))\,\d z\\
&=\xi\left(\ono_2,S;\kap-\frac{s}{2},-\frac{s}{2}\right).   
\end{align*}
The confluent hypergeometric function $\xi(Y,S;\alp,\bet)$ is extensively studied in \cite{Sh1}. 
By Lemma 3.1 of \cite{Y2} the operator 
\[\frac{M(\sgn^\kap\circ\det,\ome_\RR^{-s})}{L(-s-\frac{1}{2},\sgn^\kap\bigl)\varGamma(-s)}\] 
is entire. 
In particular, $M(\sgn^\kap\circ\det,\ome_\RR^{-s})$ is holomorphic at $s=\kap-\frac{3}{2}$. 
Letting $S=0$, $s=2\kap-3+t$ and $t\to 0$, we get 
\begin{align*}
N(3-\kap)\vph_\kap
&=\lim_{t\to 0}\xi\left(\ono_2,0;\frac{3-t}{2},\frac{3-t}{2}-\kap\right)\phi_\kap^{(-t)}\\
&=\lim_{t\to 0}(-1)^\kap 2^{-1}(2\pi)^3\frac{\varGamma_2\bigl(\frac{3}{2}-t-\kap\bigl)}{\varGamma_2\bigl(\frac{3-t}{2}\bigl)\varGamma_2\bigl(\frac{3-t}{2}-\kap\bigl)}2^{2\kap+2t-3}\phi_\kap \\
&=(-1)^\kap 4^{-1}(2\pi)^3\varGamma_2(3/2)^{-1}2^{2\kap-3}\phi_\kap
=2^{-1}(-4)^\kap \pi^2\phi_\kap 
\end{align*}
by (1.31) of \cite{Sh1}, where $\varGamma_2(s)=\sqrt{\pi}\varGamma(s)\varGamma\bigl(s-\frac{1}{2}\bigl)$.  
We write $\cald_\kap^{(2)}$ for the subrepresentation of $J(3-\kap)$ generated by $\vph_\kap$. 
It has the module $D_\kap^{(2)}$ as a quotient. 
The quotient map $\cald_\kap^{(2)}\twoheadrightarrow D_\kap^{(2)}$ is realized by the operator $N(3-\kap)$. 
Since $J(3-\kap)$ is multiplicity free even as a representation of $\U_2$, any $\GSp_4(\RR)$-invariant pairing $\cald_\kap^{(2)}\times D_\kap^{(2)}\to\CC$ factors through the quotient map. 
We construct a $\GSp_4(\RR)$-invariant pairing $\bfr:D_\kap^{(2)}\times D_\kap^{(2)}\to\CC$ in the following way: for $\phi,\phi'\in\frkD_\kap^{(2)}$ we set 
\[\bfr(\phi,\phi')=\bfq(\varphi,\phi'), \]
where we take $\varphi\in\cald_\kap^{(2)}$ so that $N(3-\kap)\varphi=\phi$. 

\begin{definition}[Bessel integrals for $D^{(2)}_\kap$]  
We define 
\[\bfA_S^\kap(\phi,\phi')=\int_{\Sym_2(\RR)}\bfr(D_\kap^{(2)}(\bfn(z))\phi,D_\kap^{(2)}(\bfd(-1))\phi')\overline{\bfe(\tr(Sz))}\,\d z \]
for $\kap\in\NN$, $S\in\Sym_2(\RR)^+$ and $\phi,\phi'\in D_\kap^{(2)}$. 
\end{definition}

\begin{remark}
We make the integral above meaningful by analytic continuation (see the previous subsection). 
If $\kap\geq 2$, then $D_\kap^{(2)}$ is a discrete series and this integral converges absolutely by Proposition 3.15 of \cite{L1}. 
\end{remark}

\begin{proposition}\label{realbessel}
For every positive integer $\kap$ and $S\in\Sym_2(\RR)^+$ we have 
\[\frac{\bfA_S^\kap(\phi_\kap,\phi_\kap)}{\bfr(\phi_\kap^{},D_\kap^{(2)}(\bfd(-1))\phi_\kap^{})}
=2^{4\kap-2}(2\pi)^{2\kap-1}\frac{(\det S)^{(2\kap-3)/2}}{\varGamma(2\kap-1)}e^{-4\pi\tr(S)}. \]
\end{proposition}

\begin{proof}
In view of (\ref{Q:factorization}) we arrive at 
\begin{align*}
\bfA_S^\kap(\phi,\phi')
&=\int_{\Sym_2(\RR)}\bfq(\cald_\kap^{(2)}(\bfn(z))\varphi,D_\kap^{(2)}(\bfd(-1))\phi')\overline{\bfe(\tr(Sz))}\,\d z\\
&=\bfB_S^{(2\kap-3)/2}(\varphi)\bfB_{-S}^{(3-2\kap)/2}(D_\kap^{(2)}(\bfd(-1))\phi'). 
\end{align*}

Let $\phi=\phi'=\phi_\kap$ and $\vph=2(-4)^{-\kap}\pi^{-2}\vph_\kap$. 
Then 
\begin{align*}
\bfr(\phi_\kap^{},D_\kap^{(2)}(\bfd(-1))\phi_\kap^{})
&=2(-4)^{-\kap} \pi^{-2}\bfq(\vph_\kap,D_\kap^{(2)}(\bfd(-1))\phi_\kap^{}) \\
&=2(-4)^{-\kap} \pi^{-2}\int_{\Sym_2(\RR)}\vph_\kap(J_2\bfn(z))\phi_{-\kap}(J_2\bfn(z))\,\d z\\
&=2(-4)^{-\kap} \pi^{-2}\xi\left(\ono_2,0;\frac{3}{2},\frac{3}{2}\right)
=2(-4)^{-\kap}. 
\end{align*}
From (4.34.K) and (4.35.K) of \cite{Sh1}
\begin{align*}
\bfB_S^{(3-2\kap)/2}(\phi_\kap)
&=\xi(\bfi_2,S;\kap,0)=(-1)^\kap 4\pi^{(4\kap-1)/2}\frac{(4\det S)^{(2\kap-3)/2}}{\varGamma(\kap)\varGamma\bigl(\kap-\frac{1}{2}\bigl)}e^{-2\pi\tr(S)}, \\
\bfB_S^{(2\kap-3)/2}(\vph_\kap)
&=\xi\left(\bfi_2,S;\frac{3}{2},\frac{3}{2}-\kap\right)
=(-4)^\kap \pi^2e^{-2\pi\tr(S)}. 
\end{align*}
We get $\bfB^{(2\kap-3)/2}_S(\varphi)=2e^{-2\pi\tr(S)}$ and conclude that 
\begin{align*}
\frac{\bfA_S^\kap(\phi_\kap,\phi_\kap)}{\bfr(\phi_\kap^{},D_\kap^{(2)}(\bfd(-1))\phi_\kap^{})}
&=2^{-1}(-4)^\kap \bfB_S^{(2\kap-3)/2}(\varphi)\bfB_{-S}^{(3-2\kap)/2}(D_\kap^{(2)}(\bfd(-1))\phi_\kap)\\
&=2^{-1}(-4)^\kap 2e^{-2\pi\tr(S)}\overline{\bfB_S^{(3-2\kap)/2}(\phi_\kap)}, 
\end{align*}
which completes our proof. 
\end{proof}


\section{Bessel periods on principal series representations}\label{sec:PS}


\subsection{Tate's local zeta integral}

Let us recall Tate's theory for local factors of quasi-characters of the multiplicative group of a local field $L$. 
Denote by $\mathcal{S}(L)$ the space of Bruhat-Schwartz functions on $L$. 
Fix a non-trivial additive character $\psi_L$ of $L$. 
Tate's local zeta integral is defined by
\[\mathcal{Z}(\Phi,\sig)=\int_{L^\times}\Phi(x)\sig(x)\,\d^\times x \]
for $\sig\in\Ome(L^\times)$ and $\Phi\in\mathcal{S}(L)$.  
The gamma factor 
\[\gamma(s,\sig,\psi_L)=\varepsilon(s,\sig,\psi_L)\frac{L(1-s,\sig^{-1})}{L(s,\sig)}\]
is defined as the proportionality constant of the functional equation
\[\mathcal{Z}(\widehat{\Phi},\sig^{-1}\ome_L^{1-s})=\gamma(s,\sig,\psi_L)\mathcal{Z}(\Phi,\sig\ome_L^s), \]
where 
\[\widehat{\Phi}(y)=\int_L\Phi(x)\psi_L(yx)\,\d x\] 
is the Fourier transform with respect to $\psi_L$. 
We repeatedly use the equation 
\beq
\gamma(s,\sig,\psi_L)\gamma(1-s,\sig^{-1},\psi_L)=\sig(-1), \label{tag:Tatefq}
\eeq
Define the additive character $\psi_K$ on $K$ by $\psi_K(x)=\psi(\Tr^K_F(\daleth^{-1}x))$, where 
\[\daleth=\begin{cases}
e_1-e_2&\text{ if $K=Fe_1\oplus Fe_2$ is split, }\\
\theta-\ol{\theta}&\text{ if $K$ is not split. } 
\end{cases}\]


\subsection{Principal series representations}

Let $\chi_1,\chi_2,\sig\in\Ome(F^\times)$ be such that $\chi_1\chi_2=\sig^{-2}$. 
We consider the principal series representation 
\[\pi=I_2(\chi)=\chi_1\times\chi_2\rtimes\sig:=\Ind_{\calb_2}^{\GSp_4(F)}\chi, \]
where the character $\chi$ of $\calb_2$ is defined by 
\begin{align*}
\chi(\bfm(\diag[a,d],\lam)u)&=\chi_1(a)\chi_2(d)\sig(\lam) &(a,d,\lam&\in F^\times,\;u\in\calu_2(F)). 
\end{align*}
The induction is always normalized, i.e., the space $V$ of $\pi$ consists of $\CC$-valued functions on $\GSp_4(F)$ with the transformation property 
\[\phi(\bfm(\diag[a,d],\lam)ug)=\chi_1(a)\chi_2(d)\sig(\lam)|a|^2|d||\lam|^{-3/2}\phi(g). \]
If $\chi_1$ and $\chi_2$ are unitary, then $\pi$ is irreducible by Lemma 3.2 of \cite{ST}.  

Then $\pi$ is equivalent to the induced representation $I(\pi_0,\sig)$, where we put $\pi_0=I_1(\chi_1\sig)$. 
A $\psi$-Whittaker functional $\bfW$ on $I_1(\chi_1\sig)$ is constructed by the Jacquet integral
\[\bfW_f(g)=\bfW(\pi(g)f):=\int_F^\st f(J_1\bfn(x)g)\psi(-x)\,\d x. \]
We define the $\GL_2(F)$-invariant pairing $b_\bfW:\pi_0\times\pi_0^\vee\to\CC$ by 
\[b_\bfW(f,f')=\int_{F^\times}\bfW_f(\diag[a,1])\bfW_{f'}(\diag[-a,1])\,\d a \]
and identify $\pi^\vee$ with $I(\pi_0^\vee,\sig^{-1})\simeq\chi_1^{-1}\times\chi_2^{-1}\rtimes\sig^{-1}$ via the pairing
\[b_\bfW^\sharp(\phi,\phi')=\int_{\Sym_2(F)}b_\bfW(\phi(w_s\bfn(z)),\phi'(w_s\bfn(z)))\,\d z. \]
 
For a Weyl element $w$ of $\GSp_4(F)$ we define $\chi^w\in\Ome(\calt_2)$ by $\chi^w(t)=\chi(w^{-1}tw)$ and define the intertwining operator $M_w(\chi)\colon I_2(\chi)\to I_2(\chi^w)$ by the integral
\[[M_w(\chi)\phi](g)=\int_{\mathcal{U}_2\cap w\mathcal{U}_2w^{-1}\backslash \mathcal{U}_2}\phi(w^{-1} ug)\d u. \]
This integral is absolutely convergent if $\chi$ lies in some open set, and can be meromorphically continued to all $\chi$. 
Let $\varSigma_+$ be the set of positive roots of $\GSp_4$. 
For each $\alpha\in\varSigma$, let $G_\alpha$ be the derived group of the centralizer in $\GSp_4$ of the kernel of $\alpha$. 
Then $G_\alpha$ has relative semi-simple rank one. 
Letting $\iota_\alpha:\SL_2\to G_\alpha$ be the relevant homomorphism, we define $\chi_\alpha\in\Ome(F^\times)$ by $\chi_\alpha(a)=\chi(\iota_\alpha(\diag[a,a^{-1}]))$ for $a\in F^\times$. 
Now we define the normalized intertwining operator
\[M^*_w(\chi)=\prod_{\alpha\in\varSigma_+,\,\alpha^w\notin\varSigma_+}\gamma(0,\chi_\alpha,\psi)\cdot M_w(\chi).\]
For example, 
\begin{align*}
M_{w^\dagger}^*(\chi)&=\gamma(0,\chi^{}_1\chi_2^{-1},\psi)\gamma(0,\chi_1\chi_2,\psi)\gamma(0,\chi_1,\psi)\gamma(0,\chi_2,\psi)M_{w^\dagger}(\chi), \\
M_{w_s}^*(\chi)&=\gamma(0,\chi_1\chi_2,\psi)\gamma(0,\chi_1,\psi)\gamma(0,\chi_2,\psi)M_{w_s}(\chi). 
\end{align*}


\subsection{Toric periods on principal series representations}
Let $\pi_0=I_1(\mu)$. 
We define the toric period of $f\in I_1(\mu)$ in the split case by 
\begin{align*}
\bfT^\mu_\Lambda(f)&=\int_{F^\times}f(\gimel\diag[a,1])\Lambda_0(a)^{-1}\d^\times a, & 
\gimel&=\begin{pmatrix} 0 & -1 \\ 1 & 1 \end{pmatrix},  
\end{align*} 
where we have written $\Lambda=(\Lambda^{}_0,\Lambda_0^{-1})$, and in the non-split case by 
\[\bfT^\mu_\Lambda(f)=\int_{F^\times\bsl K^\times}f(\iota(t))\Lambda(t)^{-1}\,\d t \]
otherwise. 
The former integral is convergent if $\Re\mu>-\frac{1}{2}$. 

\begin{lemma}\label{L:tw.3}
If $K/F$ is split, then for $f\in I_1(\mu)$ 
\[\bfT^\mu_\Lambda=\gamma\left(\frac{1}{2},\mu^{-1}\Lambda_0^{-1},\psi\right)\bfT^\bfW_\Lambda(f). \]
\end{lemma}

\begin{proof}
For each $\Phi=\Phi_1\otimes\Phi_2\in\mathcal{S}(F\oplus F)$ we define the Godement section $f=f^\Phi_\mu$ as in \eqref{E:Godement}. 
The left hand side equals 
\begin{multline*}
\int_{F^\times}\mu(a)|a|_F^{1/2}\int_{F^\times}\Phi\left((0,b)\begin{pmatrix} 0 & -1 \\ a & 1 \end{pmatrix}\right)\mu(b)^2|b|_F\Lambda_0(a)^{-1}\,\d^\times b\d^\times a\\
=\mathcal{Z}(\Phi_1,\mu\Lambda_0^{-1}\ome_F^{1/2})\mathcal{Z}(\Phi_2,\mu\Lambda_0\ome_F^{1/2}). 
\end{multline*}
The right hand side equals the product of $\gamma\left(\frac{1}{2},\mu^{-1}\Lambda_0^{-1},\psi\right)$ and 
\begin{multline*}
\int_{F^\times}\int^\st_F\mu(a)|a|_F^{1/2}\int_{F^\times}\Phi\left((0,b)\begin{pmatrix} 0 & -1 \\ a & x \end{pmatrix}\right)\mu(b)^2|b|_F\frac{\psi(-x)}{\Lambda_0(a)}\,\d^\times b\d x\d^\times a\\
=\mathcal{Z}(\Phi_1,\mu\Lambda_0^{-1}\ome_F^{1/2})\mathcal{Z}(\widehat{\Phi_2},\mu^{-1}\Lambda_0^{-1}\ome_F^{1/2}) (\mu\Lambda_0)(-1).
\end{multline*}The lemma follows from the functional equation for Tate's local integral.
\end{proof}

We associate to $\Lambda\in\Ome^1(F^\times\bsl K^\times)$ the toric integral 
\[P_{\Lambda}\in\Hom_{K^\x\times K^\times}(I_1(\mu)\boxtimes I_1(\mu^{-1}),\Lambda\boxtimes\Lambda)\] 
by the convergent integral 
\[P_{\Lambda}(f,f')=L(1,\tau_{K/F})\int_{F^\x\bksl K^\x}b_\bfW(\pi_0(t)f,\pi_0(\bfJ)f')\Lambda^{-1}(t)\, \rmd t. \]

The normalized intertwining operator $\mathcal{M}(\mu_1,\mu_2):I(\mu_1,\mu_2)\to I(\mu_2,\mu_1)$ is defined by the integral 
\[[\calm(\mu_1,\mu_2)f](g):=\gamma(0,\mu_1^{}\mu_2^{-1},\psi)\int_Ff(J_1\bfn(x)g)\,\d x \]
if $\Re(\mu_1\mu_2^{-1})>1$, and by meromorphic continuation otherwise. 
To simplify notation, we will write $\calm(\sig)=\calm(\sig,\sig^{-1})$. 

\begin{lemma}\label{L:fcnT}
\[\bfT^{\sig^{-1}}_\Lambda\circ \mathcal{M}(\sig)=\gamma\left(\frac{1}{2},\sigma_K\Lambda,\psi_K\right)\bfT^\sig_\Lambda. \]
\end{lemma}

\begin{proof}
Since $\bfW(\calm(\sig)f)=\bfW(f)$ for the choice of a normalization of the intertwining operator, if $K/F$ is split, then Lemma \ref{L:tw.3} gives 
\[\bfT^{\sig^{-1}}_\Lambda\circ \mathcal{M}(\sig)=\frac{\gamma\left(\frac{1}{2},\sig\Lambda_0^{-1},\psi\right)}{\gamma\left(\frac{1}{2},\sig^{-1}\Lambda_0^{-1},\psi\right)}\bfT^\sig_\Lambda=\gamma\left(\frac{1}{2},\sigma_K\Lambda,\psi_K\right)\bfT^\sig_\Lambda. \]
Let $K$ be a field. 
To each $\Phi\in\mathcal{S}(F\oplus F)$ we associate the Godement section
\beq\label{E:Godement}
f_\sig^\Phi(g)=\sig(\det g)|\det g|_F^{1/2}\int_{F^\times}\Phi((0,b)g)\sig(b)^2|b|_F\,\d^\times b\in I_1(\sig). 
\eeq
We shall identify $\Phi$ with a Bruhat-Schwartz function on $K$ in such a way that $\Phi(a\theta+b)=\Phi(a,b)$. 
Define the Fourier transforms of $\Phi$ by 
\begin{align*}
\widehat{\Phi}(z)&=\int_{K}\Phi(x)\psi_K(xz)\,\d x, &  
\widetilde{\Phi}(z)&=\int_{K}\Phi(x)\psi_K(x\overline{z})\,\d x.  
\end{align*}
The proof of Lemma 14.7.1 of \cite{J1} tells us that 
\beq\label{E:9}
\calm(\sig)f_\sig^{\Phi}=f_{\sig^{-1}}^{\widetilde{\Phi}}. 
\eeq
Notice that $\psi_K((a\theta+b)(x\bar\theta+y))=\psi(ay-bx)$.  

Observe that 
\[\bfT^\sig_\Lambda(f^\Phi_\sig)
=\int_{F^\x\bksl K^\x}\int_{F^\x}\Phi(bt)\frac{\sig(b^2 t\ol{t})|b^2 t\ol{t}|_F^{1/2}}{\Lambda(t)}\,\rmd^\x b\rmd^\x t
=\mathcal{Z}(\Phi,\sig_K\Lambda^{-1}\ome_K^{1/2}). \]
We define $\Phi^\tau\in\mathcal{S}(K)$ by $\Phi^\tau(x)=\Phi(\bar x)$. 
Since $\widetilde{\Phi}(z)=\widehat{\Phi^\tau}(-z)$ and $\Lambda(-\bar t)=\Lambda(t)^{-1}$ for $t\in K^\times$, the lemma follows again from \eqref{E:9} and the functional equation for $\GL_1(K)$.
\end{proof}
 

\subsection{Factorizations}

Given $\phi\in \pi=I_2(\chi)$ and $\phi'\in I_2(\chi^{-1})$, we define 
\[J_S^\Lam(\phi,\phi')=\int_{F^\times\bsl K^\times}\int_{\Sym_2(F)}^\st b_\bfW^\sharp(\pi(\bfn(z)\bft(t))\phi,\phi')\overline{\Lam^S(\bfn(z)\bft(t)))}\,\d t\d z. \]
Put $\chi_1=\mu\sig^{-1}$ and $\chi_2=\mu^{-1}\sig^{-1}$.  
Then $I(\pi_0,\sig)$ is equivalent to $I_2(\chi)$. 
Let $\pi_0'$ be a generic irreducible subrepresentation of $I_1(\mu)$. 
Since $\bfT^\mu_\Lambda$ is necessarily proportional to $\bfT^\calw_\Lambda$ on $\pi_0'$ by uniqueness, Lemma \ref{lem:conv} allows us to define the Bessel period $\bfB_{S',\Lambda}^\chi\in\Hom_{R_{S'}}(I(\pi_0',\sig),\Lambda_{S'})$ by 
\[\bfB^\chi_{S',\Lambda}(\phi)=\lim_{i\to\infty}\int_{\Sym^i_2}\bfT^\mu_\Lambda(\phi(w_s\bfn(z)))\psi_{S'}(-z)\,\d z.\]

If $K$ is a field, then the pairing 
\[I_1(\mu)\otimes I_1(\mu^{-1})\ni f\otimes f'\mapsto\int_{F^\times\bsl K^\times}f(c)f'(c)\,\d c\]
is also $\GL_2(F)$-invariant as $\GL_2(F)=\calb_1 K^\times$. 

\begin{lemma}\label{toricpair}
If $K/F$ is not split, then for $f\in I_1(\mu)$ and $f'\in I_1(\mu^{-1})$ 
\[b_\bfW(f,f')=\mu(-1)\frac{|\frkd_F|^{1/2}\zet(1)}{|\frkd_K|^{1/2}L(1,\tau_{K/F})}\int_{F^\times\bsl K^\times}f(c)f'(c)\,\d c. \]
\end{lemma}

\begin{proof}
Define $f_\mu^\dagger\in I_1(\mu)$ by $f_\mu^\dagger(\calb_1)=0$ and $f_\mu^\dagger(J_1\bfn(z))=\1_{\frko_F}(z)$. 
Then 
\[\bfW_{f_\mu^\dagger}(\diag[a,1])=\mu(a)^{-1}|a|^{1/2}\1_{\frko_F}(a) \]
and hence $b_\bfW(h_{\mu^{}},f_{\mu^{-1}}^\dagger)=\mu(-1)\zet(1)$. 
Since 
\beq
\iot(x)=\begin{pmatrix} x+\Tr(\tht) & -\Nr(\tht) \\ 1 & x \end{pmatrix}=\begin{pmatrix} \Nr(x+\tht) & x+\Tr(\tht) \\ 0 & 1 \end{pmatrix}J_1\bfn(x), \label{tag:Iwa}
\eeq
we have 
\[f_\mu^\dagger(\iot(x+\tht))=\mu_K(x+\tht)|x+\tht|_K^{1/2}\1_{\frko_F}(x) \]
for $x\in F$. 
Since $f_\mu^\dagger(F^\times)=0$, 
\begin{align*}
\int_{F^\times\bsl K^\times}f_\mu^\dagger(c)f_{\mu^{-1}}^\dagger(c)\,\d c
&=\frac{|\frkd_K|^{1/2}}{|\frkd_F|^{1/2}}\int_{F^\times}(f_\mu^\dagger f_{\mu^{-1}}^\dagger)(\iot(x+\tht))\,\frac{L(1,\tau_{K/F})}{|x+\tht|_K}\d x
\end{align*}
The identity therefore holds if $f=f_\mu^\dagger$ and $f'=f_{\mu^{-1}}^\dagger$. 
Since the $\GL_2(F)$-invariant pairings must be proportional, it holds in general.  
\end{proof}

\begin{proposition}\label{P:factorP.3}
For $f\in I_1(\mu)$ and $f'\in I_1(\mu^{-1})$ 
\[P_{\Lambda}(f,f')=\mu(-1)\zet(1)|\frkd_F|^{1/2}|\frkd_K|^{-1/2}\bfT^\mu_{\Lambda^{}}(\pi_0(\varsigma)f) \bfT^{\mu^{-1}}_{\Lambda^{-1}}(\pi_0(\varsigma\bfJ)f'). \]
\end{proposition}

\begin{proof}
Set $h=\pi_0(\varsigma)f^{}$ and $h'=\pi_0(\varsigma\bfJ)f'$. 
In the split case we have
\begin{align*}
\zeta(1)^{-1} P_\Lambda(f,f')
=&\int_{F^\x\bksl K^\x}b_\bfW(\pi_0(\iota_\varsigma(t))h,h')\Lambda(t)^{-1}\,\rmd^\x t\\
=&\int_{F^\x}\int_{F^\x}\bfW_h(\bft(ba))\bfW_{h'}(\bft(-b))\Lambda_0(a)^{-1}\,\rmd^\x b\rmd^\x a.
\end{align*} 
Lemma \ref{L:tw.3} now proves the declared identity.  

Next we shall prove the non-split case. 
Lemma \ref{toricpair} gives
\begin{align*}
P_\Lambda(f,f')
=&\mu(-1)\zet(1)\frac{|\frkd_F|^{1/2}}{|\frkd_K|^{1/2}}\int_{F^\x\bksl K^\x}\int_{F^\times\bsl K^\times} f(ct)f'(c\bfJ)\Lambda(t)^{-1}\,\d c\rmd t. 
\end{align*} 
The double integral above is clearly equal to $\bfT^\mu_{\Lambda^{}}(f) \bfT^{\mu^{-1}}_{\Lambda^{-1}}(\pi_0(\bfJ)f')$. 
\end{proof}

\begin{proposition}\label{P:factorB.4}
For $\phi\in I_2(\chi)$ and $\phi'\in I_2(\chi^{-1})$ 
\[J_S^{\Lambda}(\phi,\phi')=\frac{\mu(-1)\zet(1)|\frkd_F|^{1/2}}{L(1,\tau_{K/F})|\frkd_K|^{1/2}}\bfB^\chi_{S',\Lambda^{}}(\pi(\bfm(\varsigma))\phi) \bfB^{\chi^{-1}}_{-S',\Lambda^{-1}}(\pi(\bfm(\varsigma)\bft(\bfJ))\phi'). \]
\end{proposition}

\begin{proof}
One can prove Proposition \ref{P:factorB.4} in the same way as in the proof of Proposition \ref{P:factorB.3}, using Proposition \ref{P:factorP.3}.
\end{proof}


\subsection{Functional equations for $\bfB^\chi_{S,\Lambda}$}

Our goal is to prove the following functional equation:  

\begin{proposition}\label{prop:fq}
\[\bfB_{S',\Lambda}^{\chi^{-1}}\circ M^*_{w_\dagger}(\chi)=\gamma\left(\frac{1}{2},\mu_K\Lambda,\psi_K\right)\gamma\left(\frac{1}{2},\sig_K^{-1}\Lambda,\psi_K\right)\bfB_{S',\Lambda}^\chi. \]
\end{proposition}

By uniqueness we arrive at a functional equation 
\[\bfB_{S',\Lambda}^{\chi^w}\circ M^*_w(\chi)=c(w,\chi,\Lambda,\psi)\bfB_{S',\Lambda}^\chi. \]
When both $\chi$ and $\Lambda$ are unramified, the factor $c(w,\chi,\Lambda,\psi)$ and Lemma \ref{L:fcnT} were calculated in \cite{BFF}. 
We will generalize these results to ramified characters. 

\begin{proposition}\label{prop:fcnB}
\begin{align*}
c(s_1,\chi,\Lambda,\psi)=&\gamma\left(\frac{1}{2},(\chi_1\sigma)_K\Lambda,\psi_K\right), & 
c(s_2,\chi,\Lambda,\psi)=&1.
\end{align*}
\end{proposition}

\begin{proof}
Let $\phi\in I_2(\chi)$. 
From the expression
\[\bfB_{S',\Lambda}^{\chi^{s_1}}(M_{s_1}^*(\chi)\phi)
=\lim_{i\to\infty}\int_{\Sym_2^i}\bfT^{(\chi_1\sig)^{-1}}_\Lambda(\mathcal{M}(\chi_1\sig)\phi(w_s\bfn(z)))\psi_{S'}(-z)\, \rmd z. \]
We deduce the assertion for $s_1$ from Lemma \ref{L:fcnT}.

To prove the assertion for $s_2$, we consider the following embedding 
\begin{align*}
\iot&:\GL_2\to\GSp_4, & 
\iot\left(\begin{pmatrix} a & b \\ c & d \end{pmatrix}\right)
&=\begin{pmatrix}
ad-bc & & & \\
& a & & b \\
& & 1 & \\
& c & & d
\end{pmatrix} 
\end{align*}
and the subgroup $U_2''$ of $N_2$ given by
\begin{align*}
U_2''&=\left\{\bfn\left(\begin{pmatrix} 0 & y \\ y & w \end{pmatrix}\right)\biggl |\;y,w\in F\right\}
\end{align*}
Then $\iota(J_1)=s_2$, and we can write a unique expression
\begin{align*}
w_sn\bft(t)&=\iota(J_1\bfn(x))g', & 
g'&=s_2w_su''\bft(t), & u''&\in U_2''. 
\end{align*}
Put $\bfB_{S,\Lambda}^\chi=\bfB_{S',\Lambda}^\chi\circ\pi(\bfm(\varsigma))$ with $\pi=I_2(\chi)$. 
Recall that the upper left entry of $S$ is $1$. 
If  $\Re\sig>-\frac{1}{2}$, then $\bfB^{\chi^{s_2}}_{S,\Lambda}(M^*_{s_2}(\chi)\phi)$ is equal to 
\[\int_{F^\x\bksl K^\x}\int_{U_2''}\int_{F}[M^*_{s_2}(\chi)\phi](\iota(J_1\bfn(x))g')\psi(-x)\,\rmd x\,\ol{\psi_{S}(u'')}\Lambda^{-1}(t)\,\rmd u''\rmd t. \] 
Note that $\chi^{s_2}=(\chi^{}_1,\chi_2^{-1},\chi_2^{}\sig)$. 
We define the function $f:\GL_2(F)\to \mathbb{C}$ via $f(A)=\phi(\iota(A)g')$. 
Clearly, 
\begin{align*}
f&\in I( \sigma^{-1},\chi_1\sigma), & 
[M^*_{s_2}(\chi)\phi](\iota(A)g')&=[\mathcal{M}(\sigma^{-1},\chi_1\sigma)f](A). 
\end{align*}
Since $\bfW(\calm(\sigma^{-1},\chi_1\sigma)h)=\bfW(h)$ for any $h\in I(\sigma^{-1},\chi_1\sigma)$, we find that 
\begin{multline*}
\bfB^{\chi^{s_2}}_{S,\Lambda}(M^*_{s_2}(\chi)\phi)\\
=\int_{F^\x\bksl K^\x}\int_{U_2''}\int_{F}\phi(\iota(J_1\bfn(x))g')\psi(-x)\rmd x\,\ol{\psi_{S}(u'')\Lambda(t)}\,\rmd u''\rmd t
=\bfB^\sig_{S,\Lambda}(\phi), 
\end{multline*}
which proves the assertion for $s_2$.
\end{proof}

Now we will prove Proposition \ref{prop:fq}. 
Observe that  
\begin{align*}
\chi^{s_2}&=(\chi^{}_1,\chi_2^{-1},\chi_2^{}\sig), &
\chi^{s_1s_2}&=(\chi_2^{-1},\chi^{}_1,\chi_2^{}\sig), &
\chi^{s_2s_1s_2}&=(\chi_2^{-1},\chi_1^{-1},\sig^{-1}) 
\end{align*}
and $\chi^{s_1s_2s_1s_2}=\chi^{-1}$. 
Proposition \ref{prop:fcnB} gives 
\begin{align*}
&c(w_\dagger,\chi,\Lambda,\psi)\\
=&c(s_1,\chi^{s_2s_1s_2},\Lambda,\psi)c(s_2,\chi^{s_1s_2},\Lambda,\psi)c(s_1,\chi^{s_2},\Lambda,\psi)c(s_2,\chi,\Lambda,\psi)\\
=&\gamma\left(\frac{1}{2},(\chi_2\sigma)^{-1}_K\Lambda,\psi_K\right)\cdot 1\cdot \gamma\left(\frac{1}{2},(\chi_1\chi_2\sig)_K\Lambda,\psi_K\right)\cdot 1. \end{align*} 


\subsection{Local coefficients}

The factor $c(w,\chi,\Lambda,\psi)$ is an analogue of the local coefficients for Bessel models instead of Whittaker models and have been studied in \cite{FG} in more general situations.
We will discuss a functional equation for the Bessel periods $\bfB_{S',\Lambda}^{\calw,\sig}$ introduced in Definition \ref{def:BP}, which is of interest in its own right. 
This result is not used in our later discussion and the reader can skip the rest of this section and continue reading from the next section onwards. 

\begin{conjecture}\label{conj:fq}
Let $\pi_0$ be an irreducible admissible unitary generic representation of $\PGL_2(F)$ and $\sig\in\Ome(F^\times)$. 
Then 
\[\bfB_{S',\Lambda}^{\calw,\sig^{-1}}\circ M^*(\pi_0,\sig)=\gamma\left(\frac{1}{2},\sigma^{-1}_K\Lambda,\psi_K\right)\bfB_{S',\Lambda}^{\calw,\sig}. \]
\end{conjecture}

We will prove this conjecture, provided that $\pi_0$ is not supercuspidal. 
One will be able to prove the supercuspidal case by the global method. 

\begin{proposition}\label{C:fq} 
Conjecture \ref{conj:fq} is true if $\pi_0$ is not supercuspidal and $-\frac{1}{2}<\Re\sig<\frac{1}{2}$. 
\end{proposition}

\begin{proof}
By uniqueness the Bessel period admits a functional equation 
\[\bfB_{S',\Lambda}^{\calw,\sig^{-1}}\circ M^*(\pi_0,\sig)=c(\pi_0,\sig,\Lambda,\psi)\bfB_{S',\Lambda}^{\calw,\sig}. \]
Take $\mu\in\Ome(F^\times)$ with $\Re\mu>-\frac{1}{2}$ such that $\pi_0$ is equivalent to the (unique) irreducible subrepresentation of the principal series representation $I_1(\mu)$ of $\GL_2(F)$. 
Since $\chi^{w_s}=(\chi_2^{-1},\chi_1^{-1},\sig^{-1})$, we have $I_2(\chi^{w_s})=I(\pi_0,\sig^{-1})$. 
The restriction of $M_{w_s}^*(\chi)$ to $I(\pi_0,\sig)$ agrees with the normalized intertwining operator $M^*(\pi_0,\sig)$, and consequently 
\[c(\pi_0,\sig,\Lambda,\psi)=c(w_s,\chi,\Lambda,\psi). \] 

Since $c(\pi_0,\sig\ome_F^s,\psi)$ is a meromorphic function in $s$, it suffices to prove the equality for $\sig$ in a general position. 
We may therefore suppose that $\gam(s,\sig\Lambda^{-1}_0,\psi)$ and $\gam(s,\sig^{-1}\chi^{-1}_1\Lambda^{-1}_0,\psi)$ have no pole or zero at $s=\frac{1}{2}$. 
Then 
\[c(w_s,\chi,\Lambda,\psi)
=1\cdot c(s_1,\chi^{s_2},\Lambda,\psi)\cdot 1
=\gamma\left(\frac{1}{2},\sigma^{-1}_K\Lambda,\psi_K\right)\]
by Proposition \ref{prop:fcnB}. 
\end{proof}


\section{The Iwahori Hecke algebras and the ordinary projector $e_\mathrm{ord}^0$}\label{sec:7}

We introduce the ordinary projector on principal series representations of $\GSp_4(F)$. 
Define the Iwahori subgroup of $\GSp_4(\frko_F)$ by 
\[\II=\left\{g\in\GSp_4(\frko_F)\;\left|\;g\equiv 
\begin{pmatrix}
* & * & * &  *\\ 
0 & * & * & *\\
0 & 0 & * & 0 \\
0 & 0 & * & *  
\end{pmatrix}\pmod \frkp\right.\right\}. \]
We define elements of $\GSp_4(F)$ by  
\begin{gather*}
\del_1=\diag[\vpi,1,\vpi^{-1},1], \quad\quad
\del_2=\diag[-\vpi,-\vpi,1,1],\\
w_\dagger=\begin{pmatrix} 0 & \ono_2 \\ -\ono_2 & 0\end{pmatrix}, \quad\quad 
s_1=\bfm\left(\begin{pmatrix} 0 & 1 \\ 1 & 0 \end{pmatrix}\right), \\ 
s_2=\begin{pmatrix}
1 & & & \\ 
& & & 1 \\
& & 1 & \\
& -1 & &  
\end{pmatrix}, \; 
\eta=\begin{pmatrix} 0 & 0 & 0 & 1 \\ 0 & 0 & 1 & 0 \\ 0 & \vpi & 0 & 0 \\ \vpi & 0 & 0 & 0 \end{pmatrix}, \;
s_0=\begin{pmatrix}
& & -\frac{1}{\vpi} & \\ 
& 1 & & \\
\vpi & & & \\
& & & 1  
\end{pmatrix}.  
\end{gather*}
Observe that 
\begin{align*}
w_\dagger&=s_1s_2s_1s_2, & 
\eta&=s_2s_1s_2\del_2, &
s_0&=\eta s_2\eta^{-1}, &  
\eta s_1\eta^{-1}&=s_1. 
\end{align*}

Let $(\pi,V)$ be an admissible representation of $\GSp_4(F)$. 
For an open compact subgroup $\calk$ of $\GSp_4(F)$ the subspace 
\[V^\calk=\{\phi\in V\;|\;\pi(k)\phi=\phi\text{ for }k\in\calk\}\] 
consists of $\calk$-invariant vectors. 

\begin{definition}[Hecke operators]\label{def:23}
Given $g\in\GSp_4(F)$, we write $\II g\II=\bigsqcup_{u\in I_g}ug\II$ and define the operator $[\II g\II]$ on $V^\II$ by 
\[[\II g\II]v=\sum_{u\in I_g}\pi(ug)v. \]
Put $q_g=[\II:\II\cap g\II g^{-1}]=\sharp I_g$. 
Define the Hecke operators by  
\begin{align*}
U^\calq&=[\II\del_1\II], & 
U^\calp&=[\II\del_2\II]. 
\end{align*}
\end{definition}

Let $\calt_2'$ be the diagonal torus of $\Sp_4$ and $\til N$ the normalizer of $\calt_2'(F)$ in $\Sp_4(F)$. 
The Weyl group $W=\til N/\calt_2'(F)$ has $8$ elements and is generated by the images of $s_1,s_2$. 
We may view $W$ as a subgroup of $Sp_4(\frko_F)$ and will not distinguish in notation between the matrices $s_1,s_2$ and their images in $W$. 
The affine Weyl group $\til W=\til N/\calt_2'(\frko_F)$ is generated by the images of $s_0$, $s_1$ and $s_2$. 
The length $\ell(w)$ of $w\in\til W$ is defined as the minimum number of uses of $s_0$, $s_1$ and $s_2$ required to express $w$. 
If $\ell(ww')=\ell(w)\ell(w')$, then 
\begin{align}
q_{ww'}&=q_wq_{w'}, &
[\II ww'\II]&=[\II w\II][\II w'\II]. \label{tag:24}
\end{align}

Let $\chi_1,\chi_2,\sig\in\Ome^1(F^\times)^\circ$ be such that $\chi_1\chi_2=\sig^{-2}$. 
We consider the unramified principal series representation 
\[\pi=\chi_1\times\chi_2\rtimes\sig=\Ind_{\calb_2}^{\GSp_4(F)}\chi, \]
which is irreducible by Lemma 3.2 of \cite{ST}.  
Put 
\begin{align*}
\alp&=\chi_1(\vpi), & 
\bet&=\chi_2(\vpi), & 
\gam&=\sig(\vpi), & 
\alp_0&=\alp\gam.  
\end{align*}
The space $V^\II$ has the basis $\{\phi_w\}_{w\in W}$, where $\phi_w$ is the unique $\II$-invariant vector of $V$ such that $\phi_w(w) =1$ and $\phi_w(w')=0$ for $w\neq w'\in W$. 
We will primarily be interested in $\phi^\dagger=\phi_{w_\dagger}=\phi_{s_1s_2s_1s_2}$. 
It is convenient to order the basis as follows:
\begin{align*}
&\phi_{{\bf1}_4}, &
&\phi_{s_1}, &
&\phi_{s_2}, &
&\phi_{s_2s_1}, &
&\phi_{s_1s_2s_1}, &
&\phi_{s_1s_2}, &
&\phi_{s_1s_2s_1s_2}, &
&\phi_{s_2s_1s_2}.
\end{align*}
With respect to this basis the actions of $[\II s_i\II]$ and $[\II \eta\II]$ on $V^\II$ are given by 
{\tiny\begin{align*}
[\II s_1\II]&=\begin{pmatrix}
 0 & q & & & & & & \\
 1 & q-1 & & & & & & \\
 & & 0 & q & & & & \\
 & & 1 & q-1 & & & & \\
 & & & & q-1 & 1 & & \\
 & & & & q & 0 & & \\
 & & & & & & q-1 & 1 \\
 & & & & & & q & 0
\end{pmatrix}, \\
[\II s_2\II]&=\begin{pmatrix}
 0 & 0 & q & 0 & 0 & 0 & 0 & 0 \\
 0 & 0 & 0 & 0 & 0 & q & 0 & 0 \\
 1 & 0 & q-1 & 0 & 0 & 0 & 0 & 0 \\
 0 & 0 & 0 & 0 & 0 & 0 & 0 & q \\
 0 & 0 & 0 & 0 & 0 & 0 & q & 0 \\
 0 & 1 & 0 & 0 & 0 & q-1 & 0 & 0 \\
 0 & 0 & 0 & 0 & 1 & 0 & q-1 & 0 \\
 0 & 0 & 0 & 1 & 0 & 0 & 0 & q-1 
\end{pmatrix}, \\
[\II \eta\II]&=\begin{pmatrix}
 & & & & & & & \gam q^{3/2} \\
 & & & & & & \gam q^{3/2} & \\
 & & & & & \bet\gam q^{1/2} & & \\
 & & & & \bet\gam q^{1/2} & & & \\
 & & & \frac{\alp\gam}{q^{1/2}} & & & & \\
 & & \frac{\alp\gam}{q^{1/2}} & & & & & \\
 & \frac{\alp\bet\gam}{q^{3/2}} & & & & & & \\
 \frac{\alp\bet\gam}{q^{3/2}} & & & & & & &
\end{pmatrix}. 
\end{align*}}
thanks to Lemma 2.1.1\footnote{The matrix for $[\II s_1\II]$ in Lemma 2.1.1 of \cite{S1} contains a typo. } of \cite{S1}. 
We have $[\II s_0\II]=[\II\eta\II][\II s_2\II][\II\eta\II]^{-1}$. 
Let  
\[\phi^0_\chi=\phi_{{\bf1}_4}+\phi_{s_1}+\phi_{s_2}+\phi_{s_2s_1}+\phi_{s_1s_2s_1}+\phi_{s_1s_2}+\phi^\dagger+\phi_{s_2s_1s_2} \] 
be the unique element of $\pi$ that takes the value $1$ on $\GSp_4(\frko_F)$. 

\begin{definition}[$(\alp^{-1},\gam)$-stabilizations]\label{def:24}
Introduce the ordinary projector
\[e_\mathrm{ord}^0:=\frac{\alp}{\gam^3q^{13/2}}(U^\calq-q^2/\bet)(U^\calp-q^{3/2}/\gam)(U^\calp-q^{3/2}\gam\alp)(U^\calp-q^{3/2}\gam\bet)\]
(cf. \cite{MY}). 
Define stabilizations of $\phi^0_\chi$ by 
\begin{align*}
\phi^\ddagger&=e^0_\mathrm{ord}\phi^0_\chi, \\
\phi^\flat&=q^{-15/2}\gam^{-1}\alp^3(U^\calp-q^{3/2}\gam\bet)(U^\calq-q^2\alp)(U^\calq-q^2\bet)(U^\calq-q^2\bet^{-1})\phi^0_\chi.  
\end{align*}
\end{definition}

\begin{remark}\label{rem:22} 
The operators $U^\calq$ and $U^\calp$ are commutative. 
\end{remark}

\begin{proposition}\label{prop:21}
\begin{enumerate}
\renewcommand\labelenumi{(\theenumi)}
\item\label{prop:211} The support of $\phi^\dagger$ is contained in $\calb_2w_\dagger\calu_2(\frko_F)$. 
\item\label{prop:212} $\phi^\dagger$ is an eigenform for both $U^\calq$ and $U^\calp$, i.e.,  
\begin{align*}
U^\calq\phi^\dagger&=q^2\alp^{-1}\phi^\dagger, & 
U^\calp\phi^\dagger&=q^{3/2}\gam\phi^\dagger. 
\end{align*}
\item\label{prop:213} 
$\phi^\ddagger$ and $\phi^\flat$ are eigenforms for both $U^\calq$ and $U^\calp$. 
Moreover, 
\begin{align*} 
\phi^\ddagger&=(1-\alp q^{-1})(1-\bet q^{-1})(1-\gam^2\alp^2 q^{-1})(1-\gam^{-2}q^{-1})\phi^\dagger, &
\phi^\flat&=(\alp+1)\phi^\ddagger. 
\end{align*}
\end{enumerate}
\end{proposition}

\begin{remark}\label{rem:23}
One can partially deduce (\ref{prop:212}) from (\ref{prop:211}), namely,  
\[[\II\del_i\II]\phi^\dagger=([\II\del_i\II]\phi^\dagger)(w_\dagger)\phi^\dagger. \]
Let $g\in\GSp_4(F)$ be such that $([\II\del_i\II]\phi^\dagger)(g)\neq 0$. 
There exists $u\in\II$ such that $\phi^\dagger(gu\del_i)\neq 0$. 
We have $gu\del_i\in\calb_2w_\dagger\calu_2(\frko_F)$ in view of (\ref{prop:211}). 
Since $\del_i\calu_2(\frko_F)\del_i^{-1}\subset\calu_2(\frko_F)$, we get $g\in\calb_2w_\dagger\II=\calb_2w_\dagger\calu_2(\frko_F)$. 
\end{remark}

\begin{proof}
Put 
\begin{align*}
\bar\calu_2&=w_\dagger\calu_2 w_\dagger^{-1}, & \bar\calu_2(\frkp)&=\{u\in\bar\calu_2(\frko_F)\;|\;u\equiv\ono_4\pmod\frkp\}. 
\end{align*}
Thanks to the Iwahori factorization $\II=\bar\calu_2(\frkp)\calb_2(\frko_F)=\calb_2(\frko_F)\bar\calu_2(\frkp)$, we get 
\[\calb_2\II w_\dagger\II
=\calb_2\bar\calu_2(\frkp)w_\dagger\II
=\calb_2w_\dagger\II
=\calb_2w_\dagger\bar\calu_2(\frkp)\calb_2(\frko_F)
=\calb_2w_\dagger\calu_2(\frko_F). \]

To prove (\ref{prop:212}), one can show that
\begin{align*}
U^\calq&=[\II s_1\II][\II s_2\II][\II s_1\II][\II s_0\II], & 
U^\calp&=[\II s_2\II][\II s_1\II][\II s_2\II][\II\eta\II], 
\end{align*}
using (\ref{tag:24}). 
By direct computations the matrix representation of $U^\calq$ is 

{\tiny\[\begin{bmatrix} 
 q^2\alp & 0 & 0 & 0 & 0 & 0 & 0 & 0 \\
 \alp q(q-1) & q^2\bet & 0 & 0 & 0 & 0 & 0 & 0 \\
 0 & 0 & q^2\alp & 0 & 0 & 0 & 0 & 0 \\
 0 &  q(q-1)(\bet+1) & \alp q(q-1) & \frac{q^2}{\bet} & 0 & 0 & 0 & 0 \\
 (\alp+1)(q-1) & (q-1)^2 & \alp(q-1)^2 & \frac{q^2-q}{\bet} & \frac{q^2}{\alp} & \bet q(q-1) & 0 & 0 \\ 
 0 & 0 & \alp q(q-1) & 0 & 0 & q^2\bet & 0 & 0 \\
 \frac{\alp}{q}(q-1)^2 & \bet(q-1) & (q-1)\{\alp(q-1+\frac{1}{q})+1\} & 0 & 0 & (q-1)^2(\bet+1) & \frac{q^2}{\alp} & \frac{q^2-q}{\bet} \\
 \alp(q-1) & 0 & \alp(q-1)^2 & 0 & 0 &  q(q-1)(\bet+1) & 0 & \frac{q^2}{\bet}
\end{bmatrix}\]} 

\noindent and the matrix representation of $U^\calp$ is given by 
{\tiny\[\begin{bmatrix} 
 \frac{q^{3/2}}{\gam} & 0 & 0 & 0 & 0 & 0 & 0 & 0 \\
 0 & \frac{q^{3/2}}{\gam} & 0 & 0 & 0 & 0 & 0 & 0 \\
 \frac{q^{1/2}(q-1)}{\gam} & 0 & \gam\alp q^{3/2} & 0 & 0 & 0 & 0 & 0 \\
 0 & \frac{q^{1/2}(q-1)}{\gam} & 0 & \gam\alp q^{3/2} & 0 & 0 & 0 & 0 \\
 \frac{q-1}{\gam q^{1/2}} & \frac{(q-1)^2}{\gam q^{1/2}} & 0 & \gam\alp q^{1/2}(q-1) & \gam\bet q^{3/2} & 0 & 0 & 0 \\ 
 0 & \frac{q^{1/2}(q-1)}{\gam} & \frac{q^{1/2}(q-1)}{\gam\bet} & 0 & 0 & \gam\bet q^{3/2} & 0 & 0 \\
 \frac{(q-1)^2}{\gam q^{3/2}} & \frac{(q^3-2q^2+2q-1)}{\gam q^{3/2}} & \frac{\gam\alp(q-1)}{q^{1/2}} & \frac{\gam\alp(q-1)^2}{q^{1/2}} & \gam\bet q^{1/2}(q-1) & 0 & \gam q^{3/2} & 0 \\
 \frac{q-1}{\gam q^{1/2}} & \frac{(q-1)^2}{\gam q^{1/2}} & 0 & \gam\alp q^{1/2}(q-1) & 0 &  \frac{q^{1/2}(q-1)}{\gam\alp} & 0 & \gam q^{3/2}
\end{bmatrix}\]} 

\noindent From these we can prove (\ref{prop:212}) and observe that both $\phi^\ddagger$ and $\phi^\flat$ are multiples of $\phi^\dagger$. 
By a brute force calculation one can show that
\begin{align*}
&\alp\gam^{-3}q^{-13/2}(U^\calq-q^2\bet^{-1})(U^\calp-q^{3/2}\gam^{-1})(U^\calp-q^{3/2}\gam\alp)(U^\calp-q^{3/2}\gam\bet)e_0\\
=&(1-\alp q^{-1})(1-\bet q^{-1})(1-\gam^2\alp^2 q^{-1})(1-\gam^{-2}q^{-1})e_7,  
\end{align*}
where $e_0=\trs(1,1,1,1,1,1,1,1)$ and $e_7=\trs(0,0,0,0,0,0,1,0)$.  
\end{proof}


\section{Explicit calculations of Bessel integrals \Roman{thr}: ordinary vectors}\label{sec:expIII}

Let $\mu,\sig\in\Ome(F^\times)^\circ$. 
Put $\pi_0=I_1(\mu)$ and $\chi=(\mu\sig^{-1},\mu^{-1}\sig^{-1},\sig)$. 
Let $\pi=I(\pi_0,\sig)=I_2(\chi)$ be an irreducible unramified unitary principal series representation of $\PGSp_4(F)$. 
Recall that $\phi^\ddagger=e_\mathrm{ord}^0\phi^0_\chi$ is the $\frkp$-stabilization of the spherical section $\phi^0_\chi$ in $I_2(\chi)$ obtained by the ordinary projector $e_\mathrm{ord}^0$ in Definition \ref{def:24}. 
Let $\Lambda\in\Ome(F^\times\bsl K^\times)$. 
\begin{definition}
When $\Lambda$ is trivial on $\frkr^\times$, set $c(\Lambda)=0$. 
Otherwise we put 
\[c(\Lambda)=\max\{s\in\NN\;|\;\Lambda\text{ is trivial on }1+\frkp^s\frkr\}. \]
\end{definition}
For any given positive integer $n$ we put 
\begin{align*}
\xi^{(n)}&:=\bfm(\varsigma^{(n)}_\frkp), & 
\varsigma^{(n)}_\frkp&:=\begin{cases}
\begin{pmatrix} \theta & -1 \\ 1 & 0 \end{pmatrix}\begin{pmatrix} \vpi^n & 0 \\ 0 & 1 \end{pmatrix}&\text{if $K=F\oplus F$, }\\
\begin{pmatrix} 0 & 1 \\ \vpi^n & 0 \end{pmatrix}&\text{otherwise.}
\end{cases} 
\end{align*} 

Set $n=\max\{1,c(\Lambda)\}$. 
Our task in this section is to compute 
\[\BB_S^\Lambda(\pi(\xi^{(n)})e_\mathrm{ord}^0)=\frac{J_S^\Lambda(\pi(\xi^{(n)})e_\mathrm{ord}^0\phi^0_\chi,M^*_{w_\dagger}(\chi)\pi(\xi^{(n)})e_\mathrm{ord}^0\phi^0_\chi)}{b_\bfW^\sharp(\phi^0_\chi,M^*_{w_\dagger}(\chi)\phi^0_\chi)}, \]
where we use the pairing $b_\bfW^\sharp$ to define $J_S^\Lambda$. 
Recall the unique section $\phi^\dagger\in I_2(\chi)$ supported in $\mathcal{B}_2w_\dagger N_2(\frko_F)$ with $\phi^\dagger(w_\dagger)=1$. 
Since $(\varsigma^{(n)})^{-1}\bfJ\varsigma^{(n)}\in\frkI$, the following result readily follows upon combining Propositions \ref{P:factorB.4} and \ref{prop:fq}: 
\begin{align}
&L(1,\tau_{K/F})\frac{J_S^\Lambda(\pi(\xi^{(n)})\phi^\dagger,M^*_{w_\dagger}(\chi)\pi(\xi^{(n)})\phi^\dagger)}{\zet(1)\gamma\left(\frac{1}{2},\mu_K\Lambda,\psi_K\right)\gamma\left(\frac{1}{2},\sig_K^{-1}\Lambda,\psi_K\right)} \label{tag:ordvec}\\
=&|\frkd_F|^{1/2}|\frkd_K|^{-1/2}\bfB^\chi_{S',\Lambda^{}}(\pi(\bfm(\varsigma\varsigma^{(n)}))\phi^\dagger) \bfB^\chi_{-S',\Lambda^{-1}}(\pi(\bfm(\varsigma\varsigma^{(n)}))\phi^\dagger). \notag
\end{align}



We retain the notation in the previous section. 
Put 
\begin{align*}
\alp&=\mu(\vpi), & \gam&=\sig(\vpi). 
\end{align*}
In the split case we write $\Lambda=(\Lambda_0^{},\Lambda_0^{-1})$. 
Set 
\begin{align*}
\pi_0&\simeq I_1(\mu), & 
\chi&=(\mu\sig^{-1},\mu^{-1}\sig^{-1},\sig), & 
\pi&=I_2(\chi)\simeq I(\pi_0,\sig). 
\end{align*}

\begin{definition}
Define the modified $\frkp$-Euler factor
\[e(\pi,\Lambda)=
 \frac{\alpha^{c(\Lambda)}}{L\left(\frac{1}{2},(\chi_1\sig)_K\Lambda\right)L\left(\frac{1}{2},\sigma_K^{-1}\Lambda\right)}. \]
\end{definition}

Recall the element $f_\mu^\dagger\in I_1(\mu)$ defined in the proof of Lemma \ref{toricpair}. 

\begin{lemma}\label{L:explicitT}
We have 
\[\bfT^\mu_\Lambda(\pi_0(\varsigma\varsigma^{(n)}_\frkp)f_\mu^\dagger)=\frac{|\frkd_K|_K^{1/2}L(1,\tau_{K/F})}{|\frkd_F|^{1/2}(\alp\gam)^nq^{n/2}}\times\begin{cases}
\Lambda_0(-1) &\text{if $K=F\oplus F$, } \\
1 &\text{otherwise. }
\end{cases}\]
\end{lemma}

\begin{proof}
Note that 
\[\varsigma\varsigma^{(n)}_\frkp=\begin{cases}\pMX{\daleth \uf^n}{-1}{0}{1}&\text{ if $K/F$ is split},\\\pMX{0}{1}{\uf^n}{0}& \text{ if $K/F$ is non-split}.\end{cases}\]
Since $\cW_{f_\mu^\dagger}(\diag[a,1])=\mu(a)^{-1}|a|^{1/2}\1_{\frko_F}(a)$, if $K/F$ is split, then 
\begin{align*}
\frac{\bfT_\Lambda(\pi_0(\varsigma\varsigma^{(n)}_\frkp)f_\mu^\dagger)}{\gam\bigl(\frac{1}{2},\mu^{-1}\Lambda_0^{-1},\psi\bigl)}
=&\int_{F^\x}\cW^\dagger_{\mu^{-1}}\left(\pDII{a}{1}\pMX{\daleth\uf^n}{-1}{0}{1}\right)\Lambda_0(a)^{-1}\rmd^\x a\\
=&\int_{F^\x}\psi(-a)|a\vpi^n|^{1/2}\mu(a\vpi^n)^{-1}\1_{\frko_F}(a\vpi^n)\Lambda_0(a)^{-1}\,\d^\x a \\
=&\cZ(\Phi,\Lambda_0^{-1}\mu^{-1}\ome_F^{1/2})\Lambda_0(\vpi)^n,  
\end{align*}
where $\Phi(x)=\psi\left(-\frac{x}{\vpi^n}\right)\1_{\frko_F}(x)$. 
Since $\widehat{\Phi}(x)=\1_{\vpi^{-n}+\frko_F}$, we obtain 
\[\cZ(\Phi,\Lambda_0^{-1}\mu^{-1}\ome_F^{1/2})
=\frac{\cZ(\hat\Phi,\Lambda_0\mu\ome_F^{1/2})}{\gam\left(\frac{1}{2},\Lambda_0^{-1}\mu^{-1},\psi\right)}
=\frac{(\alp\gam)^{-n}\Lambda_0(\vpi)^{-n}q^{n/2}}{\gam\left(\frac{1}{2},\Lambda_0^{-1}\mu^{-1},\psi\right)}\vol(\vpi^{-n}+\frko_F). \]
Observe that $\vol(\vpi^{-n}+\frko_F)=\vol(1+\frkp^n)=q^{-n}\zet(1)$. 

Next we assume $K$ to be a field. 
Since 
\[K^\times=F^\times(1+\frko_F\tht)\bigsqcup F^\times(\frkp+\tht), \]
we use the formula 
\[\int_{F^\times\bsl K^\times}f(t)\d t=\int_{\frko_F} f(\iot(1+y\tht))\,\d' y+\int_\frkp f(\iot(y+\tht))\,\frac{\d' y}{|y+\tht|_K}, \]
where $\d'y$ is the Haar measure on $F$ giving $\frko_F$ the volume $L(1,\tau_{K/F})\frac{|\frkd_K|_K^{1/2}}{|\frkd_F|^{1/2}}$. 
Since $f_\mu^\dagger(\iot(y+\tht)\varsigma_\frkp^{(n)})=0$ by (\ref{tag:Iwa}), 
\begin{align*}
\bfT^\mu_\Lambda(\pi_0(\varsigma^{(n)}_\frkp f_\mu^\dagger)
&=\int_{\frko_F} f_\mu^\dagger\left(\begin{pmatrix} 1+y\Tr(\tht) & -y\Nr(\tht) \\ y & 1 \end{pmatrix}\begin{pmatrix} 0 & 1 \\ \vpi^n & 0 \end{pmatrix}\right)\Lambda(1+y\tht)\,\d'y\\
&=\int_{\frko_F} f_\mu^\dagger\left(\begin{pmatrix} \Nr(1+y\tht) & * \\ 0 & \vpi^n \end{pmatrix}\begin{pmatrix} 0 & 1 \\ 1 & \vpi^{-n}y \end{pmatrix}\right)\Lambda(1+y\tht)\,\d'y\\
&=\mu(\vpi)^{-n}q^{n/2}\int_{\frko_F} \1_{\frko_F}(\vpi^{-n}y)\,\d' y. 
\end{align*}
This finishes the proof. 
\end{proof}

\begin{lemma}\label{L:explicitB}
We have 
\[\bfB^\chi_{S,\Lambda}(\pi(\bfm(\varsigma)\xi^{(n)})\phi^\dagger)=\gam^nq^{-3n/2}\bfT^\mu_{\Lambda}(\pi_0(\varsigma\varsigma_\frkp^{(n)})f^\dagger_\mu). \]
\end{lemma}

\begin{proof}
For any $\phi\in\pi$ we have 
\[[\pi(\bfm(\varsigma)\xi^{(n)})\phi](w_s\bfn(z))=\gam^n q^{3n/2} \pi_0(\varsigma\varsigma^{(n)}_\frkp)\phi(w_s\bfn(\varsigma_\frkp^{(n)-1}\varsigma^{-1}z\trs\varsigma^{-1}\trs\varsigma_\frkp^{(n)-1})).\]
Since $S[\varsigma\varsigma_\frkp^{(n)}]\in\Sym_2(\frko_F)$ and since $\phi^\dagger(\bft(A)w_s\bfn(z))=f^\dagger_\mu(A)\1_{\Sym_2(\frko_F)}(z)$ by definition, we find the first identity by (\ref{E:dfnB}). 
Remark \ref{rem:fq} and Lemma \ref{L:explicitT} give the second identity. 
\end{proof}

The main result of this section is the following explicit formula for the Bessel integral of ordinary vectors.

\begin{proposition}\label{P:padiclocal}
Let $n=\max\{1,c(\Lambda)\}$ and $\alpha=\chi_1(\uf)$. 
Then 
\[\frac{\BB_S^\Lambda(\pi(\bfm(\xi^{(n)})e_\mathrm{ord}^0)}{L(1,\tau_{K/F})\zeta(2)\zeta(4)}=
 \frac{|\frkd_K|_K^{1/2}L(\frac{1}{2},{\rm Spn}(\pi)_K\ot\Lambda)}{|\frkd_F|^{1/2}L(1,{\rm ad},\pi)}e(\pi,\Lambda)^2\alpha^{-2n}q^{-4n}. \]
\end{proposition}

\begin{proof}
It is proved in Proposition \ref{prop:21} (\ref{prop:212}) that
\begin{align*}
e_\mathrm{ord}^0\phi^0_\chi=&\bfd(\chi)^{-1}\phi^\dagger, &\bfd(\chi)=&L(1,\chi_1)L(1,\chi_2)L(1,(\chi_1\sig)^2)L(1,\sig^{-2}). 
\end{align*}
Put $\phi'=\pi(\xi^{(n)})\phi^\dagger$. 
By Lemmas \ref{L:explicitT}, \ref{L:explicitB} and (\ref{tag:ordvec}) 
\[\frac{L(1,\tau_{K/F})J_S^\Lambda(\phi',M^*_{w_\dagger}(\chi)\phi')}{\zet(1)\gamma\left(\frac{1}{2},\mu_K\Lambda,\psi_K\right)\gamma\left(\frac{1}{2},\sig_K^{-1}\Lambda,\psi_K\right)} 
=\gam^{2n}q^{-3n}\frac{|\frkd_K|_K^{1/2}L(1,\tau_{K/F})^2}{|\frkd_F|^{1/2}(\alp\gam)^{2n}q^n}. \]
Since 
\begin{align*}
\gam\left(\frac{1}{2},\sig^{-1}_K\Lambda,\psi_K\right)&=\gam^{-2c(\Lambda)}\frac{L\left(\frac{1}{2},\sig_K\Lambda^{-1}\right)}{L\left(\frac{1}{2},\sig^{-1}_K\Lambda\right)}, \\
\gam\left(\frac{1}{2},\mu_K\Lambda,\psi_K\right)&=(\alp\gam)^{2c(\Lambda)}\frac{L\left(\frac{1}{2},\mu^{-1}_K\Lambda^{-1}\right)}{L\left(\frac{1}{2},\mu_K\Lambda\right)} 
\end{align*}
by the definition of the gamma factors, we get 
\begin{align*}
\frac{J_S^\Lambda(\phi',M^*_{w_\dagger}(\chi)\phi')}{\zet(1)L(1,\tau_{K/F})}
&=\frac{|\frkd_K|_K^{1/2}}{|\frkd_F|^{1/2}}\alp^{2c(\Lambda)-2n}q^{-4n}\frac{L\left(\frac{1}{2},\sig_K\Lambda^{-1}\right)L\left(\frac{1}{2},\mu_K^{-1}\Lambda^{-1}\right)}{L\left(\frac{1}{2},\sig^{-1}_K\Lambda\right)L\left(\frac{1}{2},\mu_K\Lambda\right)}\\
&=|\frkd_F|^{-1/2}|\frkd_K|_K^{1/2}L(1/2,{\rm Spn}(\pi)_K\ot\Lambda)e(\pi,\Lambda)^2\alpha^{-2n}q^{-4n}. 
\end{align*}

In view of (9) of \cite{AS} we have 
\[M^*_{w_\dagger}(\chi)\phi^0_\chi=\frac{L(1,\pi,\mathrm{ad})}{\zet(1)^2\bfd(\chi)^2}\phi_{\chi^{-1}}^0. \] 
Let $f_\mu^0$ be a unique section of $I_1(\mu)$ such that $f_\mu^0(\GL_2(\frko_F))=1$. 
Then 
\begin{align*}
\bfW(f_\mu^0)&=L(1,\mu^2)^{-1}, & 
b_\bfW(f_\mu^0,f_{\mu^{-1}}^0)&=\frac{\zet(1)L(1,\mathrm{ad},\pi_0)}{\zet(2)L(1,\mu^2)L(1,\mu^{-2})}=\frac{\zet(1)^2}{\zet(2)}. 
\end{align*}
Since $\phi_\chi^0(\bft(A)w_s)=f_\mu^0(A)$, we have 
\begin{align*}
b_\bfW^\sharp(\phi^0_\chi,\phi_{\chi^{-1}}^0)
=\frac{\zet(1)^2}{\zet(2)}\int_{\Sym_2(F)}|a(w_s\bfn(z))|^{-3}\,\d z
=\frac{\zet(1)^3}{\zet(2)\zet(4)}. 
\end{align*}
We conclude that 
\[\bfd(\chi)^2b_\bfW(\phi^0_\chi,M^*_{w_\dagger}(\chi)\phi^0_\chi)=L(1,\pi,\mathrm{ad})\zet(1)\zet(2)^{-1}\zet(4)^{-1}. \]
From these our proof is complete. 
\end{proof}


\section{Global Bessel periods for $\GU_2^D$}\label{sec:4}


\subsection{Notation}

If $L$ is a number field, then $\frko_L$ is the ring of integers of $L$, $\AA_L$ is the ad\`{e}le ring of $L$ and $L_\infty=L\otimes_\QQ\RR$ is the infinite part of $\AA_L$. 
When $L=\QQ$, we suppress the subscript $_L$. 
Let $\bar\ZZ$ be the ring of algebraic integers of $\bar\QQ$, $\bar\ZZ_\ell$ the $\ell$-adic completion of $\bar\ZZ$ in $\CC_\ell=\widehat{\overline{\QQ}}_\ell$ and $\widehat{\ZZ}=\prod_\ell\ZZ_\ell^{}$ the finite completion of $\ZZ$. 
Given an abelian group $M$, we put $\widehat{M}=M\otimes_\ZZ\widehat{\ZZ}$. 
In particular, $\AA_L=L_\infty\oplus\widehat{L}$. 
Let $\bfe=\prod_v\bfe_v$ denote the standard additive character of $\AA/\QQ$ such that $\bfe_v(x)=e^{2\pi\iu x}$ for $x\in\RR$ and $v\in\vSi$. 
Set $\psi^L=\bfe\circ\Tr^L_\QQ$. 
When $G$ is a reductive algebraic group over $L$, we denote by $\scra_\cusp(G)$ the space of cusp forms on $G(\AA_L)$. 
For an ad\`{e}le point $g\in G(\AA_L)$ we denote its projections to $G(\widehat{L})$, $G(L_\infty)$ and $G(L_v)$ by $g_\bff$, $g_\infty$ and $g_v$, respectively. 
We fix once and for all an embedding $\iot_\infty:\bar\QQ\hookrightarrow\CC$ and an isomorphism $\j_\ell:\CC\simeq\CC_\ell$ for each rational prime $\ell$. 
Let $\iot_\ell=\j_\ell\circ\iot_\infty$ be their composition. 
We regard $L$ as a subfield of $\CC$ (resp. $\CC_\ell$) via $\iot_\infty$ (resp. $\iot_\ell$) and $\Hom(L,\bar\QQ)=\Hom(L,\CC_\ell)$. 

Let $F$ be a totally real number field of degree $d$ and $K$ a totally imaginary quadratic extension of $F$. 
We denote by $\Delta_F$ (resp. $\Delta_K$) the discriminant of $F$ (resp. $K$), by $\frkd_F=\prod_\frkl\frkd_{F_\frkl}$ the different of $F$, by $\frkD^K_F$ the relative different of $K/F$ and by $\tau_{K/F}$ the quadratic Hecke character of $\AA^\times$ corresponding to $K/F$. 
Fix a square free ideal $\frkN=\frkN^+\frkN^-$ of $\frko_F$ such that every prime factor of $\frakN^+$ (resp. $\frakN^-$) is split (resp. not split) in $\frko_K$. 
Fix a decomposition \[\frakN^+\frko_K=\frkN_0^+\ol{\frkN^+_0}.\]
Suppose that the number of prime factors of $\frkN^-$ is even. 
Then there exists a totally indefinite quaternion algebra $D$ over $F$ of discriminant $\frkN^-$, i.e., $D$ is a central simple algebra of dimension $4$ over $F$ such that $D_v:=D\otimes_vF_v$ is a division algebra if and only if $v$ divides $\frkN^-$. 
Put 
\begin{align*}
&\Nr^D_\QQ=\Nr^F_\QQ\circ\Nr^D_F, & 
&\Tr^D_\QQ=\Tr^F_\QQ\circ\Tr^D_F, &
&N=\sharp(\frko_F/\frkN), & 
N^-&=\sharp(\frko_F/\frkN^-).  
\end{align*}

Once and for all we fix a prime $\frkp$ of $F$, which does not divide $\frkN$, CM type $\vSi$ of $K$ and a finite id\`{e}le $d_F=(d_{F_\frkl})\in\widehat{F}^\times$ such that $d_{F_\frkl}$ is a generator of the local different $\frkd_{F_\frkl}$ for each finite prime $\frkl$. 
We identify $\vSi$ with the set of real places of $F$. 
Fix a maximal order $\frko_D$ of $D$. 
For any finite prime $\frkp$ we set $\frko_{D_\frkp}=\frko_D\otimes_{\frko_F}\frko_{F_\frkp}$. 
If $\frkp$ divides $\frkN^-$, then we write $\frkP_\frkp$ for the maximal ideal of $\frko_{D_\frkp}$. 
We choose an element $\tht\in K$ such that \begin{itemize} \item $\Im\tau(\tht)>0$ for every $\tau\in\vSi$; \item $\{1,\tht\}$ is an $\frko_{F_\frkl}$-basis of $\frko_{K_\frkl}$ for every prime $\frkl$ dividing $\frkp\frkD^K_F\frkN$; \item $\tht$ is a uniformizer of $\frko_{K_\frkl}$ for every prime $\frkl$ ramified in $K$. \end{itemize} 
We regard $K$ as a subalgebra of $D$.  
Put $S=S_\tht:=\frac{1}{2}(\tht-\bar\tht)\in D_-(F)$. 

Recall that $J_1=\pMX{0}{1}{-1}{0}$. 
Put $J_\star=\diag[\ono_2,J_1]$. 
For $v\nmid \frkN^-$ we fix an isomorphism $i_v:\Mat_2(F_v)\simeq D_v$ by which we identity $\Mat_{2m}(F_v)$ with $\Mat_m(D_v)$.  
Since $i_v^{-1}(\bar x)=J_1^{-1}{^ti_v^{-1}}(x)J_1$ for $x\in D_v$, we arrive at 
\begin{align*}
J_\star^{}\GU_2^D(F_v)J_\star^{-1}&=\GSp_4(F_v), & 
J_1 D_-(F_v)&=\Sym_2(F_v). 
\end{align*}
We identify $\begin{pmatrix} \Tr^K_F(\tht) & -\Nr^K_F(\tht) \\ 1 & 0 \end{pmatrix}$ with $\tht$. 
Then 
\[J_1S_\tht=\begin{pmatrix} 1 & -\frac{\Tr^K_F(\tht)}{2} \\ -\frac{\Tr^K_F(\tht)}{2} & \Nr^K_F(\tht) \end{pmatrix}. \] 

We always take the ad\`{e}lic measure $\d g$ on $\PGU_2^D(\AA_F)$ to be the Tamagawa measure. 
We define the bilinear pairing by 
\[\La\phi,\phi'\Ra=\int_{\PGU_2^D(F)\bsl\PGU_2^D(\AA_F)}\phi(g)\phi'(g\tau_\infty)\,\d g, \]
where $\tau_\infty=\prod_{v\in\vSi}\tau_v$ with $\tau_v=\bfd(-1)\in\GU_2^D(F_v)$. 

Let $\pi\simeq\otimes_v'\pi^{}_v$ be an irreducible admissible representation of $\PGU_2^D(\AA_F)$ which is realized on a subspace $V$ of $\scra_\cusp(\PGU_2^D)$.  
The space $\scra_\cusp(\PGSp_4)$ satisfies multiplicity one thanks to the work of Arthur. 
It is conjectured in general that $\scra_\cusp(\PGU_2^D)$ satisfies multiplicity one, which we assume. 
Then since $\pi^{}_v\simeq\pi_v^\vee$ for every $v$, we have $V=\bar V:=\{\bar\phi\;|\;\phi\in V\}$.    
Thus the restriction of the pairing $\La\;,\;\Ra$ to $V\times V$ is nondegenerate. 

Let $\d z$ denote the Tamagawa measure on $D_-(\AA_F)$. 
When $\frkl$ and $\frkN^-$ are coprime, we take the Haar measure $\d z_\frkl$ on $D_-(F_\frkl)$ so that the measure of $D_-(F_\frkl)\cap\frko_{D_\frkl}$ is $1$. 
For each prime factor $\frkq$ of $\frkN^-$ we take the Haar measure $\d z_\frkq$ on $D_-(F_\frkq)$ so that the measure of $D_-(F_\frkq)\cap\frkP_\frkq$ is $1$.  
For $v\in\vSi$ we define the Haar measure $\d z_v$ on $D_-(F_v)$ by identifying $D_-(F_v)\simeq\Sym_2(F_v)$ with $F_v^3$ with respect to the standard basis. 
Then by Lemma 2.1 of \cite{Y1}
\beq
\d z=\Del_F^{-3/2}(N^-)^{-2}\prod_v\d z_v. \label{tag:measure}
\eeq
Fix a Hecke character $\Lambda\in\Ome^1(K^\times\AA^\times_F\bsl\AA_K^\times)$. 
Let $\d t$ be the invariant measure on $K^\times\AA_F^\times\bsl\AA_K^\times$ normalized to have total volume $2\varLambda(1,\tau_{K/F})$, where 
\[\varLambda(s,\tau_{K/F})=\pi^{-d(s+1)/2}\vGm((s+1)/2)^dL(s,\tau_{K/F})\] 
is the complete Hecke $L$-function of $\tau_{K/F}$. 

We define the $S$th Fourier coefficient and the Bessel period relative to $S$ and $\Lam$ of a cusp form $\phi\in\scra_\cusp(\PGU_2^D)$ by 
\begin{align*}
\bfW_S(\phi,g)&=\int_{D_-(F)\bsl D_-(\AA_F)}\phi(\bfn(z)g)\overline{\psi^F(\tau(Sz))}\,\d z, \\
B_S^\Lam(\phi,g)&=\int_{K^\times\AA_F^\times\bsl\AA_K^\times}\bfW_S(\phi,\bft(t)g)\Lam(t)^{-1}\,\d t. 
\end{align*}
Here $e$ is the identity element in $\GU_2^D(\AA_F)$. 
We will write $B_S^\Lam(\phi)=B_S^\Lam(\phi,e)$. 


\subsection{The refined Gross-Prasad conjecture for the Bessel periods}

For each place $v$ we normalize the local Bessel integrals by 
\begin{align*}
\calb_S^{\Lam_v}&=c(\pi_v,\Lam_v)^{-1}B_S^{\Lam_v}, & 
c(\pi_v,\Lam_v)&=\zet_{F_v}(2)\zet_{F_v}(4)\frac{L\bigl(\frac{1}{2},\Spn(\pi_v)_{K_v}\otimes\Lam_v\bigl)}{L(1,\tau_{K_v/F_v})L(1,\pi_v,\ad)}.
\end{align*}

We denote the complete Dedekind zeta function of $F$  by $\xi_F(s)$, the complete adjoint $L$-function of $\pi$ by $\varLambda(s,\pi,\ad)$ and the complete Godement-Jacquet $L$-function of an automorphic representation $\vPi$ of a general linear group by $\varLambda(s,\vPi)$. 
A special case of \cite[Theorem 1.2]{FM3} is stated as follows: 

\begin{theorem}[Furusawa-Morimoto]\label{coj:41}
Assume that $\pi_v$ is tempered for all $v$. 
If $\phi=\otimes_v\phi^{}_v, \phi'=\otimes_v\phi'_v\in V$ satisfy $\La\phi,\phi'\Ra\neq 0$, then 
\[\frac{B_S^\Lam(\phi)B_S^\Lam(\phi')}{\La\phi,\phi'\Ra}=\xi_F(2)\xi_F(4)\frac{\varLambda\bigl(\frac{1}{2},\Spn(\pi)_K\otimes\Lam\bigl)}{(N^-)^2\Del_F^{3/2}2^{\ell(\pi)}\varLambda(1,\pi,\ad)}\prod_v\frac{\calb_S^{\Lam_v}(\phi^{}_v,\phi_v')}{\La\phi^{}_v,\phi_v'\Ra_v}. \]
\end{theorem}

\begin{remark}\label{rem:41}
\begin{enumerate}
\renewcommand\labelenumi{(\theenumi)}
\item\label{rem:411} Theorem \ref{coj:41} is a special case of the refined Gross-Prasad conjecture formulated by Liu \cite{L1} for $\SO(m)\times\SO(l)$. 
It is easily seen that $\d t=\prod_v\d t_v$. 
Thus $C_{G_0}=1$ and $\scrb_{\pi_0}(\vph_0,\bar\vph_0)=2\varLambda(1,\tau_{K/F})$ in the notation of \cite{L1}. 
\item\label{rem:412} More generally, Furusawa and Morimoto \cite{FM,FM2,FM3} 
proved the Liu's conjecture for irreducible cuspidal tempered representations of $\SO(2n+1)$ and characters of $\SO(2)$. 
In the course of the proof they verified that $\pi$ has the weak lift $\Spn(\pi)$ to $\GL_{2n}(\AA_F)$ and obtain $L(s,\pi,\ad)$ to be the symmetric square $L$-function of $\Spn(\pi)$, which is holomorphic and nonzero at $s=1$, for the exterior square $L$-function of $\Spn(\pi)$ has a pole at $s=1$. 
\end{enumerate}
\end{remark}


\subsection{A central value formula}

Let $\kap\in\NN^\vSi$ be a tuple of $d$ natural numbers indexed by $\vSi$. 
We define the action of $\GU_2^D(F_\infty)^\circ$ on the space 
\[\frkH_2^*:=\{Z\in \Mat_2(F\ot_\QQ\CC)\;|\; \trs(ZJ_1^{-1})=ZJ_1^{-1},\; \Im(ZJ_1^{-1})>0\}\]
and the automorphy factor $J_\kap:\GU_2^D(F_\infty)^\circ\times\frkH_2^*\to\CC^\times$ by 
\begin{align*}
hZ&=(h_vZ_v)_{v\in\vSi}, & 
h_vZ_v&=(a_vZ_v+b_v)(c_vZ_v+d_v)^{-1}, \\
J_\kap(h,Z)&=\prod_{v\in\vSi}j(h_v,Z_v)^{\kap_v}, &
j(h_v,Z_v)&=\Nr^{D_v}_{F_v}(c_vZ_v+d_v)/\Nr^{D_v}_{F_v}(h_v)^{1/2},  
\end{align*}
where we write $h_v=\begin{pmatrix} a_v & b_v \\ c_v & d_v \end{pmatrix}$. 
Let $\bfi=\sqrt{-1}J_1\in\frkH_2^*$. 
Put 
\[{\rm U}_2^\vSi=\{g\in \U_2^D(F_\infty)\;|\; g(\bfi)=\bfi\}. \]

The open compact subgroup $\rmK(\frkl)$ (resp. $\rmK(\frkP_\frkl)$) of $\GU_2^D(F_\frkl)$ is defined in (\ref{tag:paramodular1}) (resp. (\ref{tag:paramodular2})). 
The paramodular subgroup of level $\frkN$ is defined by 
\[\calk_D(\frkN)=\prod_{\frakl\mid \frkN^+}J_\star^{-1}\rmK(\frkl)J_\star\times\prod_{\frkl\mid \frkN^-}\rmK(\frkP_\frkl)\times\prod_{\frakl\nmid \frkN}J_\star^{-1}\GSp_4(\frko_{F_\frakl})J_\star^{}. \]  
From now on let $\pi$ be an irreducible cuspidal automorphic representation of $\PGU_2^D(\AA_F)$ whose archimedean component is $\otimes_{v\in\vSi}D^{(2)}_{\kap_v}$ and such that $\pi_\frkl$ is generic for each finite prime $\frkl$. 
Let $V_\kap(\pi,\frkN)$ denote the subspace of $V$ on which the group $\U_2^\vSi\times \calk_D(\frkN)$ acts by the character $k\mapsto J_\kap(k_\infty,\bfi)^{-1}$. 

\begin{definition}\label{Pattersson}
For $\phi,\phi'\in V_\kap(\pi,\frkN)$ we normalize the pairing by 
\[\La\phi,\phi'\Ra_{\calk_D(\frkN)}=\La\phi,\phi'\Ra\prod_{\frkl|\frkN^+}(q_\frkl^2+1)\prod_{\frkl|\frkN^-}(q_\frkl^2-1). \]
\end{definition}

Suppose that $\dim V_\kap(\pi,\frkN)=1$. 
Fix $0\neq\phi_\pi=\otimes_v\phi_v^0\in V_\kap(\pi,\frkN)$. 
Put  
\begin{align*}
\eps_\frkl(\pi)&
=\vep\left(\frac{1}{2},\Spn(\pi_\frkl)\right), & 
\eps_{\frkN^+}(\pi)&=\prod_{\frkl|\frkN^+}\eps_\frkl(\pi). 
\end{align*}
Take $\chi_1,\chi_2,\sig\in\Ome(F_\frkp^\times)^\circ$ so that $\pi_\frkp\simeq\chi_1\times\chi_2\rtimes\sig$. 
Put $\alp_\frkp=\chi_1(\vpi_\frkp)$ and $\gam_\frkp=\sig(\vpi_\frkp)$. 
Define the $(\alp_\frkp^{-1},\gam_\frkp)$-stabilization of $\phi_\pi$ by 
\[e^0_{\ord,\frkp}\phi_\pi=(\otimes_{v\neq\frkp}\phi_v^0)\otimes e^0_{\ord,\frkp}\phi_\frkp^0, \]
where the ordinary projector $e^0_{\ord,\frkp}$ is defined in Definition \ref{def:24} with respect to $(\alp_\frkp^{-1},\gam_\frkp)$. 
For $a\in F_\infty^\times$ put $a^\kap=\prod_{v\in\vSi}|a|_{F_v}^{\kap_v}$. 
For $\frkl\nmid\frkN$ we take $\tht_\frkl\in K_\frkl$ and $A_\frkl\in\GL_2(F_\frkl)$ so that $\frko_{K_\frkl}=\frko_{F_\frkl}+\frko_{F_\frkl}\tht_\frkl$ and $J_1S_{\tht_\frkl}=(J_1S_\tht)[A_\frkl]$. 
Recall 
\[\vsi=\begin{pmatrix} 1 & -\bar\tht \\ -1 & \tht \end{pmatrix}. \]

\begin{definition}
For each positive integer $n$ we define $\vsi^{(n)}_\frkp\in\GL_2(F_\frkp)$ by 
\[\vsi^{(n)}_\frkp=\begin{pmatrix} \vth & -1 \\ 1 & 0 \end{pmatrix}\begin{pmatrix} \vpi_\frkp^n & 0 \\ 0 & 1 \end{pmatrix}, \]
where $\vth=\tht$ if $\frkp$ splits in $K$ and $\vth=0$ otherwise, and define 
\begin{align*}
\zet^{(n)}&=\bfm\Big(i_\frkp\Big(\vsi^{(n)}_\frkp\Big)\Big)\prod_{\frkl|\frkN^+}\bfm(\vsi^{-1})\prod_{\frkl\nmid\frkN}\bfm(i_\frkl(A_\frkl))\prod_\frkl\bfd(d_{F_\frkl})\in\GU_2^D(\widehat{F}). 
\end{align*}
\end{definition}

\begin{theorem}\label{thm:41}
We suppose that $\Lam_\frkl$ is unramified for every prime $\frkl$ distinct from $\frkp$. 
Put $n=\max\{1,c(\Lam_\frkp)\}$. 
Assume that $\pi$ and $\Lam$ satisfy Conjecture \ref{coj:41}. 
Assume that $F_\frkl=\QQ_2$ if $\frkl\neq\frkp$ and $2$ is divisible by $\frkl$. 
Then 
\begin{multline*}
\frac{B_S^\Lam(e^0_{\ord,\frkp}\phi_\pi,\zet^{(n)})^2}{\La\phi_\pi,\phi_\pi\Ra_{\calk_D(\frkN)}e^{4\pi\sqrt{-1}\Tr^D_\QQ(S\bfi)}}
=\frac{\Delta_F^2\xi_F(2)\xi_F(4)\Nr^D_F(4S)^\kap[\frko_K:\frko_F+\tht\frko_F]^{-3}}{\eps_{\frkN^+}(\pi)\Lam(\frkN^+_0)2^{2d+\ell(\pi)}\Delta_K^{1/2}\Nr^D_\QQ(S_\tht)^{3/2}}\\
 \times \frac{e(\pi_\frkp,\Lam_\frkp)^2}{\alp_\frkp^{2n}q_\frkp^{4n}}L(1,\tau_{K_\frkp/F_\frkp})^2\frac{\varLambda\bigl(\frac{1}{2},\Spn(\pi)_K\otimes\Lam\bigl)}{N\varLambda(1,\pi,\ad)}\prod_{\frkl|\frkN^-\cap\frkD^K_F,\;\frkl=\frkl_K^2}(1-\eps_\frkl(\pi)\Lam(\frkl_K)),
\end{multline*}
where $e(\pi_\frkp,\Lam_\frkp)$ is the $\frkp$-adic multiplier \[e(\pi_\frkp,\Lam_\frkp)=\alpha_\frkp^{c(\Lambda_\frkp)}\cdot L\biggl(\frac{1}{2},(\chi_1\sig)_{K_\frkp}\Lam_\frkp\biggl)^{-1}L\biggl(\frac{1}{2},\sig_{K_\frkp}^{-1}\Lam_\frkp\biggl)^{-1}. \]
\end{theorem}

\begin{proof}
Put $\calb_S^{\Lambda_\frkl}(H)=\frac{\BB_S^{\Lambda_\frkl}(H)}{c(\pi_\frkl,\Lam_\frkl)}$, where $\BB_S^{\Lambda_\frkl}$ is defined with respect to an additive character of order $0$. 
It should be remarked that when $F_\frkl$ is of residual characteristic $\ell$, we have defined $B_S^{\Lam_\frkl}$ with respect to the additive character $\bfe_\ell\circ\Tr^{F_\frkl}_{\QQ_\ell}$ on $F_v$. 
Taking Remark \ref{rem:21} into account, we have 
\[\calb_{S_\tht}^{\Lam_\frkl}(\pi_\frkl(\bfm(A_\frkl,d_{F_\frkl}))H\phi_\frkl,\pi_\frkl(\bfm(A_\frkl,d_{F_\frkl}))H\phi_\frkl)=|\frkd_{F_\frkl}|^{-3}|\det A_\frkl|^3\calb_{S_{\tht_\frkl}}^{\Lambda_\frkl}(H). \]

Since $\prod_\frkl|\det A_\frkl|_{F_\frkl}^{-1}=[\frko_K:\frko_F+\tht\frko_F]$, it follows from Conjecture \ref{coj:41} that   
\begin{multline*}
\frac{B_S^\Lam(\phi_\pi^\ddagger,\xi^{(n)})^2}{\La\phi_\pi,\phi_\pi\Ra}
=\xi_F(2)\xi_F(4)\frac{D_F^{3/2}\varLambda\bigl(\frac{1}{2},\Spn(\pi)_K\otimes\Lam\bigl)}{[\frko_K:\frko_F+\tht\frko_F]^3(N^-)^2 2^{\ell(\pi)}\varLambda(1,\pi,\ad)}
\calb^{\Lam_\frkp}_{S_{\tht_\frkp}}(e^0_\mathrm{ord})\\
\times\prod_{v\in\vSi}\frac{\calb_S^{\Lam_v}(\phi^0_v,\phi^0_v)}{\bfr(\phi^0_v,\pi_v(\bfd(-1))\phi^0_v)}
\prod_{\frkl|\frkN^+}\calb^{\Lam_\frkl}_S(\pi_\frkl(\bfm(\varsigma^{-1})))
\prod_{\frkl\nmid\frkp\frkN^+}\calb^{\Lam_\frkl}_{S_{\tht_\frkl}}(\mathrm{Id}). 
\end{multline*}

Taking Remark \ref{rem:21} into account, we deduce from Theorem \ref{thm:21} that  
\[\calb^{\Lam_\frkl}_{S_{\tht_\frkl}}(\mathrm{Id})=|\frkd_{K_\frkl}|_{K_\frkl}^{1/2}|\frkd_{F_\frkl}|^{-1/2} \]
if $\frkl$ and $\frkp\frkN$ are coprime. 
If $\frkl$ divides $\frkN^-$, then by Corollary \ref{cor:quatB} and (\ref{tag:inv}) 
\[\calb^{\Lam_\frkl}_S(\mathrm{Id})=\frac{|\frkd_{K_\frkl}|_{K_\frkl}^{1/2}}{|\frkd_{F_\frkl}|^{1/2}}q_\frkl^3(1-q_\frkl^{-2})\times
\begin{cases} 
1 &\text{if $K_\frkl/F_\frkl$ is unramified, } \\
1-\vep_\frkl(\pi)\Lam(\frkl_K) &\text{if $\frkl=\frkl_K^2$ is ramified. }
\end{cases}\]
If $\frkN^+$ is divisible by $\frkl$, then Proposition \ref{prop:paramodular} gives 
\[\calb^{\Lam_\frkl}_S(\bfm(\vsi^{-1}))=\vep_\frkl(\pi)\Lam(\vpi_\frkl)^{-1}|\frkd_{K_\frkl}|_{K_\frkl}^{1/2}|\frkd_{F_\frkl}|^{-1/2}q_\frkl(1+q_\frkl^{-2}). \]
Since the measure $\d t_v$ gives $F_v^\times\bsl K_v^\times$ the volume $2$, Proposition \ref{realbessel} gives 
\[\frac{\calb_{S_\tht}^{\Lam_v}(\phi^0_v,\phi^0_v)}{\bfr(\phi^0_v,\pi_v(\bfd(-1))\phi^0_v)}=2^{4\kap_v-2}\Nr^{D_v}_{F_v}(S_\tht)^{(2\kap_v-3)/2}e^{4\pi\sqrt{-1}\Tr^D_\QQ(S\bfi)}\]
for $v\in\vSi$ in view of $c(\pi_v,\Lam_v)=\frac{(2\pi)^{2\kap_v}}{\vGm(2\kap_v-1)\pi}$. 
Proposition \ref{P:padiclocal} gives 
\[\calb_{S_{\tht_\frkp}}^{\Lambda_\frkp}(\pi_\frkp(\bfm(\xi^{(n)})e_\mathrm{ord,\frkp}^0)=|\frkd_K|_K^{1/2}|\frkd_F|^{-1/2}L(1,\tau_{K_\frkp/F_\frkp})^2e(\pi_\frkp,\Lambda_\frkp)^2\alpha_\frkp^{-2n}q^{-4n}. \]
Upon combining these calculations we obtain Theorem \ref{thm:41}. 
\end{proof}


\section{Theta elements and $p$-adic $L$-functions}


\subsection{Quaterinionic modular forms}
Let 
\[\SymD^+:=\{S\in\SymD(F)\;|\; J_1 S>0\text{ for every } v\in\vSi\}. \]
Given $B\in\SymD^+$, we define a function $W_B^{(\kap)}:\GU_2^D(F_\infty)^\circ\to\CC$ by 
\beq\label{tag:27}W_B^{(\kap)}(h)=e^{2\pi\iu\Tr^D_\QQ(Bh(\bfi))}J_\kap(h,\bfi)^{-1}. \eeq

\begin{definition}[ad\`{e}lic quaternioic cusp forms]\label{def:51}
Let $\calk$ be an open compact subgroup of $\GU_2^D(\widehat{F})$. 
A quaternionic cusp form of weight $\kap$ and level $\calk$ is a $\CC$-valued function $\phi$ on $\GU_2^D(F)\bsl\GU_2^D(\AA_F)/\calk$ which satisfies 
\[\phi(zhk)=\phi(h)J_\kap(k,\bfi)^{-1}\]
for every $k\in\U_2^\vSi$ and $z\in\AA^\times$ and admits a Fourier expansion of the form 
\[\phi(h)=\sum_{B\in\SymD^+}\bfW_B(\phi,h)=\sum_{B\in\SymD^+}\bfw^{}_B(\phi,h_\bff)W_B^{(\kap)}(h_\infty) \]
for $h\in\GU_2^D(F_\infty)^\circ\GU_2^D(\widehat{F})$, where $\bfw^{}_B(\phi,-):h_\bff\mapsto\bfw^{}_B(\phi,h_\bff)$ is a locally constant $\CC$-valued function on $\GU_2^D(\widehat{F})$. 
\end{definition}

We denote the space of ad\`{e}lic quaternionic cusp forms of weight $\kap$ and level $\calk$ by $\scra^0_\kap(\calk)$. 
The space $\scra^0_\kap(\calk)$ is contained in the subspace of $\scra_\cusp(\PGU_2^D)$ which consists of right $\calk$-invariant cuspidal automorphic forms with scalar $K$-type $k\mapsto J_\kap(k,\bfi)^{-1}$ (cf. \cite{AS}). 
The finite ad\`{e}le group $\PGU_2^D(\widehat{F})$ acts on the space $\scra^0_\kap=\bigcup_\calk \scra^0_\kap(\calk)$ by right translation. 
If an irreducible cuspidal automorphic representation of $\PGU_2^D(\AA_F)$ has the lowest weight representation with minimal $K$-type $k\mapsto J_\kap(k,\bfi)^{\pm 1}$ as its archimedean component, then its non-archimedean component appears in the decomposition of $\scra^0_\kap$. 



\subsection{Theta elements}

Let $(\pi,V)$ be an irreducible cuspidal automorphic representation of $\PGU_2^D(\AA_F)$ such that $\pi_v\simeq D_{\kap_v}^{(2)}$ for $v\in\vSi$, such that $\pi_\frkl$ is generic for every finite prime $\frkl$ and such that $\dim V_\kap(\pi,\frkN)=1$. 
Fix a basis vector $\phi_\pi=\otimes_v\phi_v^0\in V_\kap(\pi,\frkN)$. 
Let $\calo_{\frkp^n}=\frko_F+\frkp^n\frko_K$ be the order of $\frko_K$ of conductor $\frkp^n$ and $\calg_n=K^\times\bsl\widehat{K}^\times/\hat\calo_{\frkp^n}^\times$ its Picard group. 
We identify $\calg_n$ with the Galois group of the ring class field $K_{\frkp^n}$ of conductor $\frkp^n$ over $K$ via geometrically normalized reciprocity law. 
Denote by $[\,\cdot\,]_n:\widehat{K}^\times\to\calg_n$ the natural projection map. Define \[x_n:\widehat{K}^\times\to\GU_2^D(\widehat{F}),\quad x_n(t)=\bft(t)\zet^{(n)}.\]

\begin{definition}\label{def:53}
Let $\alp_\calq=q_\frkp^{\kap-1}\alp_\frkp^{-1}$. 
Define the $n$th theta element by 
\[\Tht^S_n(\phi_\pi)= \alp_\calq^{-n}\sum_{[a]_n\in\calg_n}q_\frakp^{\kap n}\bfw_S(e^0_{\mathrm{ord},\frkp}\phi_\pi,x_n(a))[a]_n\in\CC[\calg_n]. \]
\end{definition}

The sequence $\{\Tht^S_n(\phi_\pi)\}_n$ satisfies the following compatibility condition: 

\begin{lemma}\label{lem:51}
Let $\vPi^{n+1}_n:\calg_{n+1}\twoheadrightarrow\calg_n$ be the natural quotient map. 
Then 
\[\vPi^{n+1}_n(\Tht^S_{n+1}(\phi_\pi))=\Tht^S_n(\phi_\pi). \]
\end{lemma}

\begin{proof}
For $n'>n$, let $K^{n'}_n$ be the kernel of the quotient map $\calg_{n'}\twoheadrightarrow\calg_n$. 
Recall that 
\beq
U_\frkp^\calq\phi=\sum_{x\in\frko_F/\frkp}\sum_{y\in\frko_F/\frkp}\sum_{z\in\frko_F/\frkp^2}\pi_\frkp\left(
\bfn\left(\begin{pmatrix}
 z & y \\ 
 y & 0 
\end{pmatrix}\right)
\bfm\left(\begin{pmatrix}
\vpi_\frakp & x  \\ 
0 & 1 
\end{pmatrix}\right)
\right)\phi. \label{tag:51}
\eeq
Since $U_\frkp^\calq\phi^\ddagger_\pi=q_\frkp^2\alp_\frkp^{-1}\phi^\ddagger_\pi$, we have 
\[\sum_{x\in\frko_F/\frkp}\bfw_S\left(\phi^\ddagger_\pi,
\bft(a)\zet^{(n)}\bfm\left(\begin{pmatrix}
\vpi_\frkp & x  \\ 
0 & 1 
\end{pmatrix}\right)\right)=q_\frkp^{-1}\alp_\frkp^{-1}\bfw_S\Big(\phi^\ddagger_\pi,\bft(a)\zet^{(n)}\Big). \]
Observing that 
\[(\vsi_\frkp^{(n)})^{-1}(1+\vpi_\frkp^n x\tht)\vsi_\frkp^{(n)}=\begin{cases}
\begin{pmatrix} 1+\vpi_\frkp^nx\tht & -x \\ 0 & 1+\vpi_\frkp^n x\bar\tht\end{pmatrix} &\text{if $\frkp$ splits in $K$, }\\
\begin{pmatrix} 1 & -x \\ \vpi_\frkp^nx\Nr(\tht) & 1+\vpi_\frkp^n x\Tr(\tht)\end{pmatrix} &\text{otherwise, }
\end{cases} \]
we get 
\[\sum_{x\in\frko_F/\frkp}\bfw_S\Big(\phi^\ddagger_\pi,
\bft(a(1+\vpi_\frkp^n x\tht))\zet^{(n+1)}\Big)=q_\frkp^{-1}\alp_\frkp^{-1}\bfw_S\Big(\phi^\ddagger_\pi,\bft(a)\zet^{(n)}\Big). \]
The left hand side is $\sum_{u\in K^{n+1}_n}\bfw_S\Big(\phi^\ddagger_\pi,x_{n+1}(au)\Big)$ in view of the description
\[K^{n'}_n=[\hat\calo_{\frkp^n}^\times]_{n'}=\{[1+\vpi^n_\frkp x\tht]_{n'}\;|\;x\in\frko_F/\frkp^{n'-n}\}. \]
The proof is complete by Definition \ref{def:53}. 
\end{proof}

Put $\calg_\infty=\displaystyle{\lim_{\stackrel{\longleftarrow}{n}}}\,\calg_n$. 
Lemma \ref{lem:51} enables us to define 
\[\Tht^S(\phi_\pi):=\{\Tht^S_n(\phi_\pi)\}_n\in \lim_{\stackrel{\longleftarrow}{n}}\CC\powerseries{\calg_n}. \]
Assuming that $c(\Lam_\frkl)=0$ for $\frkl\neq\frkp$, we will write $c(\Lam)=c(\Lam_\frkp)$. 
When $n\geq c(\Lam)$, we can view $\Lam$ as a character of $\calg_n$ and extend it linearly to a function $\Lam:\CC[\calg_n]\to\CC$. 
Let $W_K$ be the group of roots of unity in $K$ and $w_K$ its order. 
Put $Q_K=[\frko_K^\times:W_K\frko_F^\times]\in\{1,2\}$. 

\begin{proposition}\label{prop:51}
Assume that $\pi$ and $\Lam$ satisfy Conjecture \ref{coj:41}. 
Assume that $F_\frkl=\QQ_2$ if $\frkl\neq\frkp$ and $2$ is divisible by $\frkl$. 
If $n\geq 1$ and $n\geq c(\Lam)$, then 
\begin{align*}
\frac{\Lam(\Tht^S_n(\phi_\pi))^2}{\La\phi_\pi,\phi_\pi\Ra_{\calk_D(\frkN)}}=&Q_K^2w_K^2\frac{\Delta^{}_F\Delta_K^{1/2}\Nr^D_F(4S)^\kap}{2^{4d+2+\ell(\pi)}\Nr^D_\QQ(S)^{3/2}} \xi_F(2)\xi_F(4)\frac{\varLambda\bigl(\frac{1}{2},\Spn(\pi)_K\otimes\Lam\bigl)}{N\varLambda(1,\pi,\ad)}\\
&\times \frac{e(\pi_\frkp,\Lam_\frkp)^2\eps_{\frkN^+}(\pi)}{[\frko_K:\frko_F+\tht\frko_F]^3\Lam(\frkN^+_0)}\prod_{\frkl|\frkN^-\cap\frkD^K_F,\;\frkl=\frkl_K^2}(1-\eps_\frkl(\pi)\Lam(\frkl_K)). 
 \end{align*}
 \end{proposition}

\begin{proof}
We may assume that $n=\max\{1,c(\Lam)\}$ by Lemma \ref{lem:51}. 
Denote by $\mathrm{vol}(\hat\calo_{\frkp^n}^\times)$ the volume of the image of $K^\times_\infty\hat\calo_{\frkp^n}^\times$ in $K^\times\AA_F^\times\bsl\AA_K^\times$ with respect to the measure $\d t$. 
Remark \ref{rem:41}(\ref{rem:411}) together with the class number formula gives 
\[\mathrm{vol}(\hat\calo_{\frkp^n}^\times)=\mathrm{vol}(\hat\frko_K^\times)L(1,\tau_{K_\frkp/F_\frkp})q_\frkp^{-n}=2^{d+1}Q_K^{-1}w_K^{-1}\sqrt{\Delta^{}_F\Delta_K^{-1}}L(1,\tau_{K_\frkp/F_\frkp})q_\frkp^{-n}. \]
Since $W^{(\kap)}_S(\bft(t)g)=W^{(\kap)}_S(g)$ for $t\in K_\infty^\times$ by (\ref{tag:27}), 
\begin{align*}
B_S^\Lam\Big(\phi^\ddagger_\pi,\zet^{(n)}\Big)&=\frac{(\det S)^{\kap/2}}{e^{2\pi\sqrt{-1}\Tr^D_\QQ(S\bfi)}}\int_{K^\times \widehat{F}^\times\bsl\widehat{K}^\times}\bfw_S\Big(\phi^\ddagger_\pi,\bft(t)\zet^{(n)}\Big)\Lam(t)^{-1}\,\d t\\
&=e^{-2\pi\sqrt{-1}\Tr^D_\QQ(S\bfi)}\mathrm{vol}(\hat\calo_{\frkp^n}^\times)q_\frkp^{-n}\alp_\frkp^{-n}\Lam(\Tht^S_n(\phi_\pi)).   
\end{align*}
Theorem \ref{thm:41} gives the declared formula. 
\end{proof}



\subsection{Classical quaternionic cusp froms}
Hereafter let $F=\QQ$.  
Thus 
\begin{align*}
\frkN&=N=N^+N^-, &
K&=\QQ(\sqrt{-\Delta_K}), & 
\rmK_D(N)&=\calk_D(N)\cap \U_2^D(\QQ).  
\end{align*} 
It is important to note that 
\begin{align}
D^\times(\AA)&=D^\times\cdot D_\infty^{\times\circ}\widehat{R}^\times, & 
\GU_2^D(\AA)&=\GU_2^D(\QQ)\GU_2^D(\RR)^\circ\calk_D(N). \label{tag:approximation}
\end{align}

We associate to $h_\infty\in\U_2^D(\RR)$ and a function $f:\frkH_2^*\to\CC$ another function 
\begin{align*}
f|_\kap h_\infty&:\frkH_2^*\to\CC, & 
f|_\kap h_\infty(Z)&=f(h_\infty Z)J_\kap(h_\infty,Z)^{-1}. 
\end{align*}
Symbolically, we will abbreviate $q^B=e^{2\pi\iu\Tr^{D_\infty\otimes\CC}_\CC(BZ)}$ for $B\in D_-^+$. 

\begin{definition}[classical quaternionic cusp forms]\label{def:52}
A quaternionic cusp form of weight $\kap$ with respect to a discontinuous subgroup $\rmK\subset\U_2^D(\QQ)$ is a holomorphic function $f$ on $\frkH_2^*$ which satisfies $f|_\kap\gam=f$ for every $\gam\in{\rm K}$ and admits for every $\bet\in\U_2^D(\QQ)$ a Fourier expansion of the form 
\[f|_\kap\bet(Z)=\sum_{B\in D_-^+}\bfc_B(f|_\kap\bet)q^B. \] 
Let $S_\kap({\rm K},\CC)$ stand for the space of such cusp forms. 
\end{definition}
Let $\calk$ be an open compact subgroup of $\GU_2^D(\widehat{\QQ})$. 
Set $\rmK=\U_2^D(\QQ)\cap\calk$. 
If
\[\GU_2^D(\AA)=\GU_2^D(\QQ)\GU_2^D(\RR)\calk, \] 
then we can associate to each $f\in S_\kap({\rm K},\CC)$ a unique $\phi_f\in \scra^0_\kap(\calk)$ such that
\begin{align*}
f(Z)&=\phi_f(h_\infty)J_\kappa(h_\infty,\bfi) & 
(h_\infty&\in\GU_2^D(\RR)^\circ,\;h_\infty(\bfi)=Z). 
\end{align*}
We shall call $\phi_f$ the ad\`{e}lic lift of $f$. By definition $\bfw_B(\phi_f,e)=\bfc_B(f)$. 
Let $\II_p$ be the standard Iwahori subgroup in $\GSp_4(\ZZ_p)$ in Section \ref{sec:7}. 
Put 
\begin{align*}
\calk_D(N,p)&=\{g\in \calk_D(N)\mid g_p\in \II_p\}, & 
\rmK_D(N,p)&=\calk_D(N,p)\cap \U_2^D(\QQ). 
\end{align*}

Recall that $\frkI_\ell$ is the Iwahori subgroup of $\GL_2(\ZZ_\ell)$. 
Let $R$ be an Eichler order of level $N^+$ in $\frko_D$. 
We identify $R_\ell=R\otimes_\ZZ\ZZ_\ell$ with $\Mat_2(\ZZ_\ell)$ or $\frkI_\ell$ via $i_\ell$ according to whether $\ell\nmid N$ or $\ell|N^+$. 
Put 
\begin{align*}
R^\perp&=\{x\in D\;|\;\Tr^D_\QQ(xy)\in\ZZ\text{ for all }y\in R\}, \\ 
\SymR&=\{x\in \SymD\;|\;\Tr^D_\QQ(xy)\in\ZZ\text{ for all }y\in R^\perp\cap D_-\}. 
\end{align*}
Observe that if $N^+$ is divisible by $\ell$, then 
\[J_\star^{-1}\rmK(\ell)J_\star=\biggl\{\begin{pmatrix} a & b \\ c & d \end{pmatrix}\biggl|\; a, d\in R_\ell,\; b\in R^\perp_\ell,\;c\in \ell R^\perp_\ell\biggl\}. \]
It follows that 
\[\rmK_D(N)=\biggl\{\begin{pmatrix} a & b \\ c & d \end{pmatrix}\biggl|\; a, d\in R,\; b\in R^\perp,\;c\in N R^\perp\biggl\}. \]
Thus the Fourier coefficients of cusp forms in the spaces $S_\kap(\rmK_D(N),\CC)$ and $S_\kap(\rmK_D(N,p),\CC)$ are indexed by $\SymR^+=\SymR\cap D^+_-$. 

The operators $U_p^\calp$ and $U_p^\calq$ on the space $\scra^0_\kap(\calk_D(N,p))$ are defined in Definition \ref{def:23}. 
We define the operators $\bfU_p^\calp$ and $\bfU_p^\calq$ on $S_\kap(\rmK_D(N,p),\CC)$ by 
\begin{align*}
[\bfU_p^\calp f](Z)&=p^{\kap-3}\cdot [U_p^\calp\phi_f](h_\infty)\cdot J_\kap(h_\infty,\bfi),\\
[\bfU_p^\calq f](Z)&=p^{\kap-3}\cdot [U_p^\calq\phi_f](h_\infty)\cdot J_\kap(h_\infty,\bfi), \end{align*}
where $f\in S_\kap(\rmK_D(N,p),\CC)$ and $h_\infty\in\GU_2^D(\RR)^\circ$ with $h_\infty(\bfi)=Z$. 

\begin{proposition}\label{prop:52}
Let $f\in S_\kap(\rmK_D(N,p),\CC)$.  Then 
\begin{align*}
[\bfU_p^\calp f](Z)&=\sum_{B\in \SymR^+}\bfc_{pB}(f)q^{pB}, &
[\bfU_p^\calq f](Z)&=\sum_{x=1}^p\sum_{B\in \SymR^+}\bfc_{\bar\gam_x B\gam_x}(f)q^{B},   
\end{align*}
where $\gamma_x\in D^\x$ is such that $\gam_x\in i_p\left(\begin{pmatrix} p & x \\ 0 & 1 \end{pmatrix}\right)D_\infty^{\times\circ} \widehat{R}^\x$. 
\end{proposition}

\begin{proof}
The first formula is easy to prove. 
We see by (\ref{tag:51}) and (\ref{tag:27}) that 
\begin{align*}
[U_p^\calq\phi](h_\infty)=&\sum_{x=1}^p\sum_{X\in R_-/\gamma_x R_-\bar\gamma_x}\phi\left(\pDII{\gamma_x^{-1}}{\bar\gamma_x}_\infty\pMX{1}{X}{0}{1}_\infty h_\infty\right)\\
=&p^{-\kap}\sum_{x=1}^p\sum_{B\in \SymR^+}\sum_{X\in R_-/\gamma_x R_-\bar\gamma_x}\bfc_B(f)e^{2\pi\iu\Tr^D_\QQ(B\gamma_x^{-1}(Z+X)\bar\gamma_x^{-1})}\\
=&p^{-\kap}\sum_{x=1}^p\sum_{B\in \SymR^+}\bfc_{B}(f) q^{\bar\gam_x^{-1}B\gamma_x^{-1}}\sum_{X\in R_-/\gamma_x R_-\bar\gamma_x}e^{2\pi\iu\Tr^D_\QQ(\bar\gamma_x^{-1}B\gamma_x^{-1}X)}
\end{align*}
where $R^-=R\cap D_-$. 
Note that 
\[\sum_{X\in R_-/\gamma_x R_-\bar\gamma_x}e^{2\pi\iu\Tr^D_\QQ(\bar\gamma_x^{-1}B\gamma_x^{-1}X)}=\begin{cases}
\#(R_-/\gamma_x R_-\bar\gamma_x)&\text{ if }B\in \bar\gamma_x \breve R_-\gamma_x,\\
0&\text{ otherwise}.
\end{cases}\]
On the other hand, 
\[R_-/\gamma_x R_-\bar\gamma_x\simeq \Sym_2(\ZZ_p)/\pMX{p}{0}{x}{1}\Sym_2(\ZZ_p)\pMX{p}{x}{0}{1}. \]
We find that $\#(R/\gamma_x R\bar\gamma_x)=p^3$.
\end{proof}

\begin{definition}
For each subring $A\subset\CC$ the space $S_\kap(\rmK,A)$ consists of cusp forms $f\in S_\kap(\rmK,\CC)$ such that $\bfc_B(f)\in A$ for every $B\in \SymD^+$. 
\end{definition}

The following result follows from Proposition \ref{prop:52} immediately.

\begin{corollary}\label{cor:51}
$\bfU^\calq_p$ and $\bfU^\calp_p$ stabilize $S_\kap(\rmK_D(N,p),A)$ for any $A$. 
\end{corollary}

\begin{lemma}\label{lem:52}
If $f\in S_\kap(\rmK_D(N,p),A)$, then for every $B\in \SymD^+$ and $t\in\widehat{K}^\times$  
\[p^{n\kap}\bfw_B(\phi_f,x_n(t))\in A. \]
\end{lemma}
\begin{proof}
Let $\calr\subset R$ be the Eichler order of level $pN^+$.
Given $t\in \widehat{K}^\times$, we use (\ref{tag:approximation}) to write $t\vsi^{(n)}_p=\gamma_\bff u$ with $\gamma\in D^\times(\QQ)$, $\Nr^D_\QQ(\gam)>0$ and $u\in \widehat{\calr}^\x$. 
Then $\bft(\gam_\bff)^{-1}x_n(t)\bfd(p^n)\in\calk_D(N,p)$. 
Let $h_\infty\in\GU_2^D(\RR)^\circ$. 
Put $Z=h_\infty(\bfi)$. 
Then  
\begin{align*}
\phi_f(h_\infty x_n(t))
&=\phi_f\left(\begin{pmatrix} \gamma^{-1} & \\ & p^n\gamma^{-1}\end{pmatrix}_\infty\cdot h_\infty\right)\\
&=\frac{f(p^{-n}\gamma^{-1} Z \gamma)}{J_\kap\left(\pDII{\gamma^{-1}}{p^n\gamma^{-1}},Z\right)}=p^{-n\kap}f(p^{-n}\gamma^{-1}Z\gamma).\end{align*}
Thus $p^{n\kap}\bfw_B(\phi_f,x_n(t))=\bfw_{p^n\gamma^{-1} B\gamma}(\phi_f,e)=\bfc_{p^n\gamma^{-1} B\gamma}(f)\in A$.
\end{proof}

\subsection{Anticyclotomic $p$-adic $L$-functions}

Let $f\in S_\kap(\rmK_D(N),\CC)$ be a Hecke eigenform and $\pi$ an irreducible cuspidal automorphic representation of $\PGU_2^D(\AA)$ generated by the associated ad\`{e}lic lift  $\phi_\pi:=\phi_f\in \scra^0_\kap(\calk_D(N))$. 
Denote the ring of integers of the Hecke field of $\pi$ by $\frko_\pi$.  
We may further assume that $f$ belongs to $S_\kap(\rmK_D(N),\frko_\pi)$ (cf. \cite[Proposition 1.8 on p.~146]{FC} or \cite{Lan}). 
Since $\overline{\phi_\pi}$ equals $\pi(\tau_\infty)\phi_\pi$ up to scalar by the multiplicity one, we may assume that $\overline{\phi_\pi}=\pi(\tau_\infty)\phi_\pi$. 

Let $\{\alp^{}_p,\alp_p^{-1}\gam_p^{-2},\gamma^{}_p\}$ be the Satake parameters of $\pi_p$. 
Put 
\begin{align*}
\alp_\calp&=p^{\kap-3/2}\gam_p, & 
\bet_\calp&=p^{\kap-3/2}\gam_p^{-1}\alp_p^{-1}, \\
\alp_\calq&=p^{2-\kap}\alp_\calp\bet_\calp=p^{\kap-1}\alp_p^{-1}, &  
\bet_\calq&=p^{\kap-1}\alp^{}_\calp\bet_\calp^{-1}=p^{\kap-1}\alp^{}_p\gam_p^2. 
\end{align*}

\begin{definition}\label{def:stabilization}
Let 
\[f^\ddagger:=\alp_\calp^{-3}\alp_\calq^{-1}\cdot (\bfU_p^\calq-\bet_\calq)(\bfU^\calp_p-p^{2\kap-3}\alp^{-1}_\calp)(\bfU_p^\calp-p^{2\kap-3}\bet^{-1}_\calp)(\bfU^\calp_p-\bet_\calp)f.\]
\end{definition}

Let $\ord_p:\overline{\QQ}_p^\times\twoheadrightarrow\QQ^\times_+$ denote the normalized additive valuation.
From now on we assume one of the parameters of $\pi_p$ to satisfy 
\beq
\ord_p\iot_p(\alp_\calq)=0. \tag{$\calq$-ord}
\eeq
It is convenient to suppose that another parameter satisfies
\beq
\ord_p\iot_p(\alp_\calp)=0. \tag{$\calp$-ord}
\eeq 

\begin{remark}\label{rem:51}
\begin{enumerate}
\renewcommand\labelenumi{(\theenumi)}
\item\label{rem:514} The eigenvalues of $\bfU^\calq_p$ are $\alp_\calq,\bet_\calq,p^{2\kap-2}\alp^{-1}_\calq,p^{2\kap-2}\bet^{-1}_\calq$ and those of $\bfU^\calp_p$ are $\alp_\calp,\bet_\calp,p^{2\kap-3}\alp^{-1}_\calp,p^{2\kap-3}\bet^{-1}_\calp$ by the proof of Proposition \ref{prop:21}. 
\item\label{rem:512} The eigenvalue of $f^\ddagger$ for $\bfU^\calq_p$ is a $p$-adic unit if and only if ($\calq$-ord) holds. 
\item\label{rem:513} The eigenvalue of $f^\ddagger$ for $\bfU^\calp_p$ is a $p$-adic unit if and only if ($\calp$-ord) holds. 
\end{enumerate}
\end{remark}

\begin{lemma}\label{lem:53}
If $\pi_p$ satisfies ($\calq$-ord) and $A$ contains $\frko_\pi$ and eigenvalues of $\bfU_p^\calp$ and $\bfU_p^\calq$, then $\alp_\calp^3\cdot f^\ddagger\in S_\kap(\rmK_D(N,p),A)$. 
\end{lemma}

\begin{proof}
Remark \ref{rem:51}(\ref{rem:514}) and Corollary \ref{cor:51} imply that 
\[\bet_\calq,\;\bet_\calp,\;p^{2\kap-3}\alp_\calp^{-1},\;p^{2\kap-3}\bet^{-1}_\calp\in\bar\ZZ, \]
and the lemma follows from ($\calq$-ord) and Definition \ref{def:stabilization}.
\end{proof}

Let $\Gam^-$ be the maximal $\ZZ_p$-free quotient group of $\calg_\infty$ and $\Del$ the torsion subgroup of $\calg_\infty$. 
We have an exact sequence 
\[1\to\Del\to\calg_\infty\to\Gam^-\to 1. \]
Fix a noncanonical isomorphism $\calg_\infty\simeq\Del\times\Gam^-$ once and for all. 
If $n\geq 1$, then the map $\Del\to\calg_\infty\to\calg_n$ is injective and hence 
\begin{align*}
\calg_n&\simeq\Del\times\Gam^-_n, & 
\Gam^-&\twoheadrightarrow\Gam^-_n=\calg_n/\Del. 
\end{align*}
Let $\chi:\Del\to\bar\QQ^\times$ be a branch character.  
Define the $\chi$-branch of $\Tht^S_n(\phi_\pi)$ by 
\[\Tht^S_n(\phi_\pi,\chi)=\chi(\Tht^S_n(\phi_\pi))\in\CC[\Gam_n^-]. \]
 Enlarge $\frko_\pi$ to a ring $A$ so that $A$ contains values of $\chi$ and eigenvalues of $\bfU_p^\calp$ and $\bfU_p^\calq$. 
 By Lemma \ref{lem:52}, $\Tht^S_n(\phi_\pi,\chi)$ belongs to $A[\Gam_n^-]$, and hence 
 \[\Tht^S(\phi_\pi,\chi):=\lim_{\stackrel{\longleftarrow}{n}}\Tht^S_n(\phi_\pi,\chi)\in A\powerseries{\Gam_\infty^-}. \] 

\begin{definition}[periods]\label{def:period}
We normalize $f\in S_\kap(\rmK_D(N),\frko_\pi)$ so that not all the Fourier coefficients vanish modulo the maximal ideal of the completion of $\frko_\pi$ with respect to $\iot_p$. 
Define a period $\Ome_{\pi,N^-}$ of $\pi$ by 
\[\Ome_{\pi,N^-}:=\Lam(1,\pi,\mathrm{ad})/\La f,f\Ra_{\rmK_D(N)}, \]
where we define the Petersson norm of $f$ by 
\[\La f,f\Ra_{\rmK_D(N)}:=\int_{\rmK_D(N)\bsl\frkH_2^*}|f(Z)|^2\, (\det Y)^{\kap-3}\d X\d Y. \]
\end{definition}

\begin{proposition}\label{prop:Mass}
Let $N=N^+N^-$ be a square-free integer. 
Then
\[\mathrm{vol}(\rmK_D(N)\bsl\frkH_2^*)=2\xi_\QQ(2)\xi_\QQ(4)\prod_{q|N^+}(q^2+1)\prod_{\ell|N^-}(\ell^2-1). \]
Let $f\in S_\kap(\rmK_D(N),\CC)$. 
Then 
\[\La f,f\Ra_{\rmK_D(N)}=\xi_\QQ(2)\xi_\QQ(4)\La\phi_\pi,\phi_\pi\Ra_{\calk_D(N)}. 
\]
\end{proposition}

\begin{proof}
Recall that a motive $M$ of Artin-Tate type is attached to $\U_2^D$ in Section 1 of \cite{G1} and a canonical Haar measure $|\ome_v|$ on $\U_2^D(\QQ_v)$ is defined in Section 4 of \cite{G1}. 
For each rational prime $q$, let $\mu_q$ be the Haar measure $L_q(M^\vee(1))|\ome_q|$ on $\U_2^D(\QQ_q)$. 
Let $\mu_\infty$ be Euler-Poincar\'{e} measure on $\U_2^D(\RR)$. 
Then $\mu=\otimes_v\mu_v$ defines a Haar measure on $\U_2^D(\AA)$. 
Since the Tamagawa number of $\U_2^D$ is $1$, we have 
\beq
\int_{\U_2^D(\QQ)\bsl\U_2^D(\AA)}\mu=L_\infty(M)/c(\Sp_4(\RR)) \label{tag:Gross}
\eeq
by Theorem 9.9 of \cite{G1}, where $c(\Sp_4(\RR))$ is a cohomological invariant attached to the real symplectic group of rank $2$. 

Let $\HH=D^2$ be a left $D$-vector space with Hermitian form 
\[\La (x,y),(x',y')\Ra=x\bar y'+y\bar x'. \]
Let $L=\frko_D\oplus\breve{\frko}_D$ be a maximal lattice in $\HH(\QQ)$, where 
\[\breve{\frko}_D=\{x\in D(\QQ)\;|\;\Tr^D_\QQ(xy)\in\frko_D\text{ for every }y\in\frko_D\}. \]
Put $L_q=L\otimes_\ZZ\ZZ_q$. 
Define an open compact subgroup $\calk(L)$ of $\U_2^D(\widehat{\QQ})$ by 
\begin{align*}
\calk(L)&=\prod_q\calk(L_q), & 
\calk(L_q)&=\{g\in\GU_2^D(\QQ_q)\;|\;L_q g=L_q\}. 
\end{align*}
Then $\calk(L)\simeq\calk_D(N^-)$. 
By the strong approximation property of $\U_2^D$ 
\[\U_2(\QQ)\bsl\U_2^D(\AA)\simeq(\rmK\bsl\frkH_2^*)\times \U_2^\vSi\times\calk\] 
for any open compact subgroup $\calk=\prod_q\calk_q$ of $\U_2^D(\widehat{\QQ})$, where we put $\rmK=\calk\cap\U_2^D(\QQ)$. 
Now we see from (\ref{tag:Gross}) that  
\[\mathrm{vol}(\rmK\bsl\frkH_2^*)\mathrm{vol}(\U_2^\vSi)\prod_q\mu_q(\calk_q)=L_\infty(M)/c(\Sp_4(\RR)). \]
Taking Lemma 3.3.3 of \cite{RS} into account, we get 
\[\frac{\mathrm{vol}(\rmK_D(N)\bsl\frkH_2^*)}{\mathrm{vol}(\rmK_D(N^-)\bsl\frkH_2^*)}=\prod_{q|N^+}(q^2+1). \]
Proposition 9.3 of \cite{GHY} says that $\mu_q(\calk(L_q))=1$ or $q^2-1$ according to whether $D$ is split over $\QQ_q$ or not. 
It is well-known that 
\[\mathrm{vol}(\Sp_4(\ZZ)\bsl\frkH_2)=2\xi_\QQ(2)\xi_\QQ(4). \]
Combining these results, we obtain the first equality. 
Proposition 3.1 of \cite{DPSS} combined with this equality and Definition \ref{Pattersson} gives the second identity. 
\end{proof}


Take $\tht$ so that $\frko_K=\ZZ+\ZZ\tht$. 
Recall the decomposition $N^+\frko_K=\frkN_0^+\overline{\frkN_0^+}$. 

\begin{theorem}\label{thm:52}
Let $A$ be a subring of $\bar\QQ_p$ which contains $\frko_\pi$, values of $\chi$ and eigenvalues of $\bfU_p^\calp$ and $\bfU_p^\calq$.  
If $\pi_p$ satisfies ($\calq$-ord), then
\[\alp_\calp^3\cdot \Tht^S(\phi_\pi,\chi)\in A\powerseries{\calg_\infty}. \]
Let $\hat\nu:\Gam^-\to\bar\QQ^\times_p$ be a $p$-adic character of finite order. 
Then 
\begin{align*}
\frac{\hat\nu(\Tht^S(\phi_\pi,\chi))^2}{\La f,f\Ra_{\rmK_D(N)}}
=&w_K^22^{2\kap-3}\Delta_K^{\kap-1}e(\pi_p,\chi_p\nu_p)^2\frac{\varLambda\bigl(\frac{1}{2},\Spn(\pi)_K\otimes\chi\nu\bigl)}{N2^{s_\pi}\varLambda(1,\pi,\ad)} \\
&\times \eps_{N^+}(\pi)(\chi\nu)(\frkN_0^+)^{-1}\prod_{\ell|(N^-,\Delta_K),\;\ell=\frkl_K^2}(1-\eps_\frkl(\pi)(\chi\nu)(\frkl_K)). 
\end{align*} 
\end{theorem}

\begin{proof}
By Definitions \ref{def:24} and \ref{def:stabilization}, $\alp_\calp^3\cdot e^0_{\ord,p} \phi_\pi$ is the ad\`{e}lic lift of $\alp_\calp^3\cdot f^\ddagger$. In view of Lemmas \ref{lem:52} and \ref{lem:53}, we conclude that 
\[\alp_\calp^3\cdot p^{n\kap}\bfw_S(e^0_{\ord,p} \phi_\pi,x_n(t))\in A \]
for every $t\in\widehat{K}^\times$ and nonnegative integers $n$. 
Since $F=\QQ$, we have $Q_K=1$. 
We have $\det S=\frac{\Delta_K}{4}$ for our choice of $\tht$.  
We finally get the stated formula by Propositions \ref{prop:51} and \ref{prop:Mass}. 
\end{proof}


\subsection{Reformulation in terms of optimal embeddings}

We explain theta elements in Definition \ref{def:53} agrees with the one given in the introduction. 
When $\calo$ is an order of $\frko$, an embedding $\iot:\calo\hookrightarrow R$ is said to be optimal if $\iot(K)\cap R=\iot(\calo)$. 
Fix an optimal embedding $\Psi:\frko_K\hookrightarrow R$. 
Recall that $\calr\subset R$ is the Eichler order of level $pN^+$. For any positive integer $n$, write $\varsigma_p^{(n)}\in\gamma_n\wh \calr^\x$ for some $\gamma_n\in D^\times$. 
Then one verifies directly that the embedding $\Psi_n\in\Hom(K,D)$ defined by $\Psi_n(x)=\gamma_n^{-1} \Psi(x)\gamma_n$ is an embedding from $\calo_{p^n}$ to $\calr$ of conductor $p^n$, namely an optimal embedding in $\Hom(\calo_{p^n},\calr)$. For $\sigma\in\Gal(K_{p^n}/K)$, write $\sigma=\mathrm{rec}_K(t)|_{K_{p^n}}$ for some $t\in\wh K^\times$. 
Write $\Psi_n(t)\in\gamma\wh\calr^\times$. By definition,
\[x_n(t)=\iota(t)\varsigma_p^{(n)}\in \gamma_n\Psi_n(t)\gamma_n^{-1}\varsigma_p^{(n)}\in \gamma_n\gamma\wh\calr^\times.\]
On the other hand, according to the recipe of the Galois action on $\Psi_n$,  \[S_{\Psi_n^\sigma}=\gamma^{-1}\Psi_n(p^n\sqrt{-\Delta_K}/2)\gamma=p^n(\gamma_n\gamma)^{-1} S_\Psi\gamma_n\gamma.\]
By Lemma \ref{lem:52}, we find that 
\[p^{n\kap}\bfw_{S_\Psi}(x_n(t),f^\ddagger)=\bfc_{p^n(\gamma_n\gamma)^{-1}S_\Psi \gamma_n\gamma}(f^\ddagger)=\bfc_{S_{\Psi_n}^\sigma}(f^\ddagger).\]
This shows that the theta element $\Theta_n^S(\phi_\pi,1)$ with $S=S_\Psi$ agrees with the one described in the introduction.

\end{document}